\documentclass[ps]{imsart}

\usepackage{amsthm,amsmath,amssymb,tikz,mathtools,mathrsfs}
\RequirePackage[numbers]{natbib}
\RequirePackage[colorlinks,citecolor=blue,urlcolor=blue]{hyperref}

\startlocaldefs

\newtheorem{lemma}{Lemma}
\newtheorem{theorem}[lemma]{Theorem}
\newtheorem{proposition}[lemma]{Proposition}
\newtheorem{remark}[lemma]{Remark}
\newtheorem{assumption}[lemma]{Assumption}
\newtheorem{corollary}[lemma]{Corollary}
\newtheorem{definition}[lemma]{Definition}

\DeclareMathOperator*{\esslim}{ess lim}
\DeclareMathOperator{\sgn}{sgn}
\DeclareMathOperator\argth{artanh}

\newcommand{\argsh}{\mbox{argsh\,}}

\def\E{\mathbb{E}}
\def\P{\mathbb{P}}
\def\R{\mathbb{R}}
\def\Z{\mathbb{Z}}
\def\N{\mathbb{N}}

\def\Var{\mathrm{Var}}

\def\tun{\mathbf{1}}

\newcommand{\bbC}{\mathbb{C}}
\newcommand{\bbD}{\mathbb{D}}
\newcommand{\bbE}{\mathbb{E}}

\newcommand{\bbN}{\mathbb{N}}

\newcommand{\bbP}{\mathbb{P}}

\newcommand{\bbZ}{\mathbb{Z}}

\newcommand{\cC}{\mathcal{C}}

\newcommand{\cF}{\mathcal{F}}

\newcommand{\cI}{\mathcal{I}}
\newcommand{\cJ}{\mathcal{J}}

\newcommand{\cL}{\mathcal{L}}
\newcommand{\cM}{\mathcal{M}}
\newcommand{\cN}{\mathcal{N}}
\newcommand{\cO}{\mathcal{O}}

\newcommand{\cQ}{\mathcal{Q}}

\newcommand{\ccM}{\mathscr{M}}

\endlocaldefs

\begin{document}

\begin{frontmatter}

\title{On the scaling limits of weakly asymmetric bridges}
\runtitle{Scaling limits of bridges}


\author{\fnms{Cyril} \snm{Labb\'e}\ead[label=e1]{labbe@ceremade.dauphine.fr}}
\address{Universit\'e Paris Dauphine,\\ PSL University, CNRS, UMR 7534,\\ CEREMADE, 75016 Paris, France.\\ \printead{e1}}

\runauthor{C.Labb\'e}

\begin{abstract}
We consider a discrete bridge from $(0,0)$ to $(2N,0)$ evolving according to the corner growth dynamics, where the jump rates are subject to an upward asymmetry of order $N^{-\alpha}$ with $\alpha \in (0,\infty)$. We provide a classification of the asymptotic behaviours - invariant measure, hydrodynamic limit and fluctuations - of this model according to the value of the parameter $\alpha$.
\end{abstract}




\setcounter{tocdepth}{2}
\tableofcontents

\end{frontmatter}

\section{Introduction}

The simple exclusion process is a statistical physics model that has received much attention from physicists and probabilists over the years. In the present article, we consider the case where the lattice is a segment of length $2N$, the total number of particles is $N$ and the process is endowed with zero-flux boundary conditions. We present a classification of the scaling limits of this model according to the asymmetry imposed on the jump rates.\\

We start with a precise definition of the model. Consider a system of $N$ particles on the linear lattice $\{1,\ldots,2N\}$, subject to the exclusion rule that prevents any two particles from sharing a same site. Each particle, independently of the others, jumps to its left at rate $p_N$ and to its right at rate $1-p_N$ as long as the target site is not occupied. Additionally, we impose a ``zero-flux" boundary condition to the system: a particle located at site $1$, resp.~at site $2N$, is not allowed to jump to its left, resp.~to its right. At any given time $t$, let $X_i(t)$ be equal to $+1$ if the $i$-th site is occupied, and to $-1$ otherwise.\\
It is standard to associate to such a particle system a so-called height function, defined by
\begin{equation*}
S(0) = 0\;,\quad S(k)=\sum_{i=1}^k X_i\;,\quad k=1,\ldots, 2N\;.
\end{equation*}
Since the total number of particles is $N$, we necessarily have $S(2N)=0$ so that $S$ is a discrete bridge. The dynamics of the particle system can easily be expressed at the level of the height function: at rate $p_N$, resp.~$1-p_N$, each downwards corner, resp.~upwards corner, flips into its opposite: we refer to Figure \ref{Fig1} for an illustration. The law of the corresponding dynamical interface will be denoted by $\P^N$.

This dynamics admits a unique reversible probability measure:
\begin{equation}\label{Eq:muN}
\mu_N(S) = \frac1{Z_N} \Big(\frac{p_N}{1-p_N}\Big)^{\frac{1}{2}A(S)}\;,
\end{equation}
where $A(S) = \sum_{k=1}^{2N} S(k)$ is the area under the discrete bridge $S$, and $Z_N$ is a normalisation constant, usually referred to as the partition function. This observation appears in various forms in the literature, see for instance~\cite{JanLeb,FunaSasa,EthLab15}. Notice that the dynamics is reversible w.r.t.~$\mu_N$ even if the jump rates are asymmetric: this feature of the model is a consequence of our ``zero-flux" boundary condition.

\begin{figure}\centering\label{Fig1}
	\begin{tikzpicture}[xscale=0.4, yscale=0.4, >=stealth]
	
	\draw[-,thin,color=gray] (0,0) -> (14,0);
	
	\draw[-,gray] (0,0) -- (7,7) -- (14,0);
	\draw[-,gray] (0,0) -- (7,-7) -- (14,0);
	
	\draw[-,thick,color=black] (0,0) node[below left] {$0$} -- (1,1) -- (2,2) -- (3,1) -- (4,0) -- (5,1) -- (6,0) -- (7,-1) -- (8,0) -- (9,-1) -- (10,-2) -- (11,-1) -- (12,0) -- (13,1) -- (14,0) node[below right] {$2N$};
	
	\draw[-,style=dotted] (3,1) -- (4,2) node[above] {\tiny \mbox{rate} $p_N$} -- (5,1);
	\draw[->] (4,0.5) -- (4,1.5);
	
	\draw[-,style=dotted] (7,-1) -- (8,-2)node[below] {\tiny \mbox{rate} $1-p_N$} -- (9,-1);
	\draw[->] (8,-0.5) -- (8,-1.5);
	
	
	\end{tikzpicture}
	\caption{An example of interface.}
\end{figure}
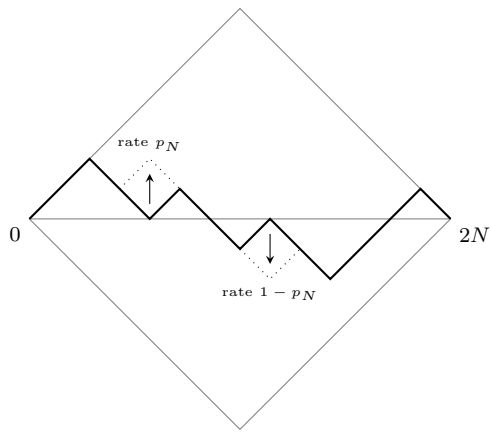

From now on, we only consider ``upwards'' asymmetries, that is, $p_N \geq 1/2$, and we aim at understanding the behaviour of the interface according to the strength of the asymmetry. It is clear that the interface will be pushed higher and higher as the asymmetry increases. On the other hand, the interface is subject to some geometric restrictions: it is bound to $0$ at both ends, and it is lower than the deterministic shape $k\mapsto k\wedge(2N-k)$. Actually, it is simple to check that under a strong asymmetry, that is, $p_N=p > 1/2$, the interface is essentially stuck to the latter deterministic shape. Therefore, to see non-trivial behaviours we need to consider asymmetries that vanish with $N$. We make the following choice of parametrisation:
\begin{equation*}
\frac{p_N}{1-p_N} = \exp\Big(\frac{4\sigma}{(2N)^\alpha}\Big)\;,\qquad \sigma > 0\;,\quad \alpha \in (0,\infty)\;,
\end{equation*}
so that
\begin{equation*}
p_N = \frac12 + \frac{\sigma}{(2N)^\alpha} + \cO\Big(\frac1{N^{2\alpha}}\Big)\;.
\end{equation*}
The important parameter is $\alpha$. When it equals $+\infty$, we are in the symmetric regime, while $\alpha = 0$ corresponds to a strong asymmetry. In the present paper, we investigate the whole range $\alpha \in (0,\infty)$.\\
The results are divided into three parts: first, we characterise the scaling limit of the invariant measure; second, the scaling limit of the fluctuations at equilibrium; and third, we investigate the scaling limit of the dynamics out of equilibrium. As we will see, the model displays a large variety of limiting behaviours, most of them already appear in related contexts of the literature. Let us mention that two sections of the present paper have been taken from the recent work~\cite{LabbeKPZ}, and have been enriched with more details and comments.

\medskip

From now on, we extend $S$ into a piecewise affine map from $[0,2N]$ into $\R$: namely, $S$ is affine on every interval $[k,k+1]$. We also let $L$ be the log-Laplace functional associated to the Bernoulli $\pm 1$ distribution with parameter $1/2$, namely $L(h) = \log \cosh h$. 

\subsection{The invariant measure}

The main result of this section is a Central Limit Theorem for the interface under $\mu_N$. To state this result, we need to rescale appropriately the interface according to the strength of the asymmetry. For $\alpha\geq 1$, the space variable will be rescaled by $2N$ so that the rescaled space variable will live in $I_\alpha = [0,1]$. On the other hand, for $\alpha < 1$, we will zoom in a window of order $(2N)^\alpha$ around the center of the lattice, hence the rescaled space variable will live in $I^N_\alpha = [-N/(2N)^\alpha,N/(2N)^\alpha]$ for any $N\geq 1$, and $I_\alpha = \R$ in the limit $N\rightarrow \infty$.\\
This being given, we introduce the curve $\Sigma^N_\alpha$ around which the fluctuations occur. One would have expected this curve to be defined as the mean of $S$ under $\mu_N$, but it is actually more convenient to opt for a different definition. However, $\Sigma^N_\alpha$ coincides with the mean under $\mu_N$ up to some negligible terms, see Remark \ref{Rk:Mean} below. For all $k\in\{0,\ldots,2N\}$, we set $x_k=k/(2N)$ if $\alpha\geq 1$, $x_k=(k-N)/(2N)^\alpha$ if $\alpha < 1$, and
\begin{equation}\label{Eq:DefSigmaN}
\Sigma^N_\alpha(x_k) = \sum_{i=1}^{k} L'(h^N_i)\;,\quad h^N_i = \frac{2\sigma}{(2N)^\alpha}\Big(N-i+\frac{1}{2}\Big)\;,\quad i\in\{1,\ldots,2N\}\;.
\end{equation}
In between these discrete values $x_k$'s, $\Sigma^N_\alpha$ is defined by affine interpolation. Let us mention that $\Sigma^N_\alpha(x)\sim (2N)^{2-\alpha}\sigma x(1-x)$ when $\alpha > 1$, and $\Sigma^N_\alpha(x) \sim 2N\int_0^x L'(\sigma(1-2y)) dy$ when $\alpha =1$. On the other hand, when $\alpha<1$, $\Sigma^N_\alpha$ differs from the maximal curve $k\mapsto k\wedge (2N-k)$ only in a window of order $N^\alpha$ around the center of the lattice. We refer to Figure \ref{Fig:Static} for an illustration and to Equations (\ref{Eq:Mean1}) and (\ref{Eq:Mean2}) for precise formulae.\\
We are now ready to introduce the rescaling for the fluctuations. For $\alpha \geq 1$, we set
\begin{equation*}
u^N(x) := \frac{S(x2N) - \Sigma_\alpha^N(x)}{\sqrt{2N}}\;,\quad x\in [0,1]\;,
\end{equation*}
and for $\alpha < 1$, we set
\begin{equation*}
u^N(x) := \frac{S(N + x(2N)^\alpha)-\Sigma_\alpha^N(x)}{(2N)^{\frac{\alpha}{2}}} \;,\quad x\in I^N_\alpha\;.
\end{equation*}

\begin{theorem}\label{Th:Static}
Under the invariant measure $\mu_N$, we have $u^N \xRightarrow{(d)} B_\alpha$ as $N\rightarrow\infty$. The process $B_\alpha$ is a centered Gaussian process on $I_\alpha$ with covariance
\begin{equation}\label{Eq:CovB}
\E\big[B_\alpha(x)B_\alpha(y)\big] = \frac{q_\alpha(0,x)\, q_\alpha(y,1)}{q_\alpha(0,1)}\;,\quad\forall x\leq y \in [0,1]\;,\quad \alpha\in [1,\infty)\;,
\end{equation}
and
\begin{equation}\label{Eq:CovB2}
\E\big[B_\alpha(x)B_\alpha(y)\big] = \frac{q_\alpha(-\infty,x)\, q_\alpha(y,+\infty)}{q_\alpha(-\infty,+\infty)}\;,\quad\forall x\leq y \in \R\;,\quad\alpha \in (0,1)\;,
\end{equation}
where
\begin{equation*}
q_\alpha(x,y) =
\begin{cases}
x\vee y - x\wedge y &\mbox{ if } \alpha \in (1,\infty)\\
\int_{x\wedge y}^{x\vee y} L''(\sigma(1-2u)) du &\mbox{ if } \alpha = 1\\
\int_{x \wedge y}^{x\vee y} L''(2\sigma u) du &\mbox{ if } \alpha \in (0,1)\;.
\end{cases}
\end{equation*}
\end{theorem}

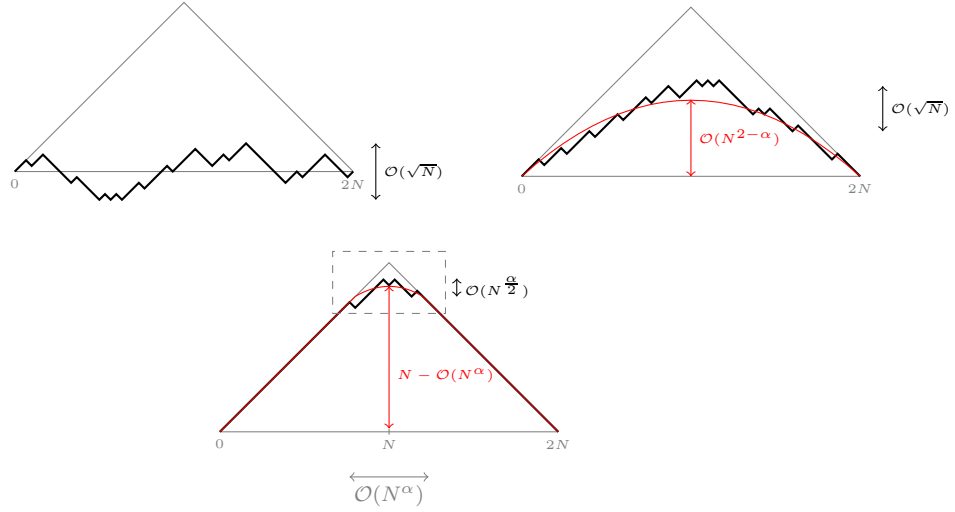
\begin{figure}
\begin{minipage}[b]{0.3\linewidth}
      \centering
\begin{tikzpicture}[scale=1.5]
\draw[-,gray] (0,0) node[below]{\tiny$0$} -- (3,0) node[below]{\tiny$2N$};
\draw[-,gray] (0,0) -- (1.5,1.5) -- (3,0);

\draw[-,thick] (0,0) -- (0.05,0.05) -- (0.1,0.1) -- (0.15,0.05) -- (0.2,0.1) -- (0.25,0.15) -- (0.3,0.1) -- (0.35,0.05) -- (0.4,0) -- (0.45,-0.05) -- (0.5,-0.1) -- (0.55,-0.05) -- (0.6,-0.1) -- (0.65,-0.15) -- (0.7,-0.2) -- (0.75,-0.25) -- (0.8,-0.2) -- (0.85,-0.25) -- (0.9,-0.2) -- (0.95,-0.25) -- (1,-0.2) -- (1.05,-0.15) -- (1.1,-0.1) -- (1.15,-0.15) -- (1.2,-0.1) -- (1.25,-0.05) -- (1.3,0) -- (1.35,0.05) -- (1.4,0) -- (1.45,0.05) -- (1.5,0.1) -- (1.55,0.15) -- (1.6,0.2) -- (1.65,0.15) -- (1.7,0.1) -- (1.75,0.05) -- (1.8,0.1) -- (1.85,0.15) -- (1.9,0.1) -- (1.95,0.15) -- (2,0.2) -- (2.05,0.25) -- (2.1,0.2) -- (2.15,0.15) -- (2.2,0.1) -- (2.25,0.05) -- (2.3,0) -- (2.35,-0.05) -- (2.4,-0.1) -- (2.45,-0.05) -- (2.5,0) -- (2.55,-0.05) -- (2.6, 0) -- (2.65,0.05) -- (2.7,0.1) -- (2.75,0.15) -- (2.8,0.1) -- (2.85,0.05) -- (2.9,0) -- (2.95,-0.05) -- (3,0);

\draw[<->,black] (3.2,-0.25) -- (3.2,0) node[right]{\tiny$\cO(\sqrt{N})$} -- (3.2,0.25);
\end{tikzpicture}
\end{minipage}\hspace{3cm}
\begin{minipage}[b]{0.3\linewidth}   
\centering
\begin{tikzpicture}[scale=1.5]
\draw[-,gray] (0,0) node[below]{\tiny$0$} -- (3,0) node[below]{\tiny$2N$};
\draw[-,gray] (0,0) -- (1.5,1.5) -- (3,0);

\draw[-,thick] (0,0) -- (0.05,0.05) -- (0.1,0.1) -- (0.15,0.15) -- (0.2,0.1) -- (0.25,0.15) -- (0.3,0.2) -- (0.35,0.25) -- (0.4,0.2) -- (0.45,0.25) -- (0.5,0.3) -- (0.55,0.35) -- (0.6,0.4) -- (0.65,0.35) -- (0.7,0.4) -- (0.75,0.45) -- (0.8,0.5) -- (0.85,0.55) -- (0.9,0.5) -- (0.95,0.55) -- (1,0.6) -- (1.05,0.65) -- (1.1,0.7) -- (1.15,0.65) -- (1.2,0.7) -- (1.25,0.75) -- (1.3,0.8) -- (1.35,0.75) -- (1.4,0.7) -- (1.45,0.75) -- (1.5,0.8) -- (1.55,0.85) -- (1.6,0.8) -- (1.65,0.85) -- (1.7,0.8) -- (1.75,0.85) -- (1.8,0.8) -- (1.85,0.75) -- (1.9,0.7) -- (1.95,0.65) -- (2,0.6) -- (2.05,0.55) -- (2.1,0.6) -- (2.15,0.55) -- (2.2,0.6) -- (2.25,0.55) -- (2.3,0.5) -- (2.35,0.45) -- (2.4,0.4) -- (2.45,0.45) -- (2.5,0.4) -- (2.55,0.35) -- (2.6, 0.3) -- (2.65,0.25) -- (2.7,0.2) -- (2.75,0.15) -- (2.8,0.2) -- (2.85,0.15) -- (2.9,0.1) -- (2.95,0.05) -- (3,0);

\draw[-,red] [domain=0:3] plot(\x,{0.3*\x*(3-\x)});

\draw[<->,red] (1.5,0) -- (1.5,0.34) node[right]{\tiny $\cO(N^{2-\alpha})$} -- (1.5,0.68);

\draw[<->,black] (3.2,0.4) -- (3.2,0.6) node[right]{\tiny$\cO(\sqrt{N})$} -- (3.2,0.8);

\end{tikzpicture}

\end{minipage}

\bigskip

\begin{minipage}[b]{1\linewidth}   
\centering
\begin{tikzpicture}[scale=1.5]
\draw[-,gray] (0,0) node[below]{\tiny$0$} -- (3,0) node[below]{\tiny$2N$};
\draw[-,gray] (0,0) -- (1.5,1.5) -- (3,0);

\draw[-,thick] (0,0) -- (0.05,0.05) -- (0.1,0.1) -- (0.15,0.15) -- (0.2,0.2) -- (0.25,0.25) -- (0.3,0.3) -- (0.35,0.35) -- (0.4,0.4) -- (0.45,0.45) -- (0.5,0.5) -- (0.55,0.55) -- (0.6,0.6) -- (0.65,0.65) -- (0.7,0.7) -- (0.75,0.75) -- (0.8,0.8) -- (0.85,0.85) -- (0.9,0.9) -- (0.95,0.95) -- (1,1) -- (1.05,1.05) -- (1.1,1.1) -- (1.15,1.15) -- (1.2,1.1) -- (1.25,1.15) -- (1.3,1.2) -- (1.35,1.25) -- (1.4,1.3) -- (1.45,1.35) -- (1.5,1.3) -- (1.55,1.35) -- (1.6,1.3) -- (1.65,1.25) -- (1.7,1.2) -- (1.75,1.25) -- (1.8,1.2) -- (1.85,1.15) -- (1.9,1.1) -- (1.95,1.05) -- (2,1) -- (2.05,0.95) -- (2.1,0.9) -- (2.15,0.85) -- (2.2,0.8) -- (2.25,0.75) -- (2.3,0.7) -- (2.35,0.65) -- (2.4,0.6) -- (2.45,0.55) -- (2.5,0.5) -- (2.55,0.45) -- (2.6, 0.4) -- (2.65,0.35) -- (2.7,0.3) -- (2.75,0.25) -- (2.8,0.2) -- (2.85,0.15) -- (2.9,0.1) -- (2.95,0.05) -- (3,0);

\draw[-,red] (0,0)--(1.2,1.2);
\draw[-,red][domain=1.2:1.8] plot(\x, {(\x-1.2)*(1.8-\x) + 1.2});
\draw[-,red] (1.8,1.2)--(3,0);

\draw[-,gray] (1.5,-0.025) -- (1.5,0.025);
\draw[-,gray] (1.5,0) node[below] {\tiny$N$};
\draw[<->,gray] (1.15,-0.4) -- (1.5,-0.4) node[below]{$\cO(N^{\alpha})$} -- (1.85,-0.4);

\draw[<->,red] (1.5,0.03) -- (1.5,0.5) node[right]{\tiny$N - \cO(N^\alpha)$} -- (1.5,1.29);

\draw[<->,black] (2.1,1.2) -- (2.1,1.28) node[right]{\tiny$\cO(N^{\frac{\alpha}{2}})$} -- (2.1,1.36);

\draw[-,dashed,gray] (1,1.05) -- (2,1.05) -- (2,1.6) -- (1,1.6) -- (1,1.05);

\end{tikzpicture}

\end{minipage}
\caption{\it Upper left $\alpha > 3/2$, upper right $\alpha \in [1,3/2]$, bottom $\alpha < 1$. The red curve is $\Sigma^N_\alpha$: in the first case, it is negligible compared to the fluctuations so we have not drawn it.}\label{Fig:Static}
\end{figure}

\begin{remark}\label{Rk:Cov}
In all cases, the process $(c_\alpha B_\alpha(r_\alpha(x)),x\in[0,1])$ is a Brownian bridge where
$$ c_\alpha = \begin{cases} 1 &\mbox{ if } \alpha > 1\;,\\
\sqrt{\frac{\sigma}{\tanh(\sigma)}} &\mbox{ if } \alpha = 1\;,\\
\sqrt{\sigma} &\mbox{ if } \alpha \in (0,1)\;.
\end{cases}
$$
and
$$
r_\alpha(x) = \begin{cases} x &\mbox{ if } \alpha > 1\;,\\
\frac12\big(1+\sigma^{-1} \argth((2x-1)\tanh(\sigma))\big) &\mbox{ if } \alpha = 1\;,\\
\frac{1}{2\sigma} \argth(2x-1) &\mbox{ if } \alpha \in (0,1)\;.
\end{cases}
$$
\end{remark}
Let us make a few comments on this result. Recall that the mean shape $\Sigma^N_\alpha$ is of the order $N^{(2-\alpha)\wedge 1}$. The scaling of our process $u^N$ together with the convergence result stated in Theorem \ref{Th:Static} show that the fluctuations of the interface around this mean shape are of the order $N^{(1\wedge \alpha)/2}$. Consequently, according as $\alpha$ is larger than $3/2$, resp. equal to $3/2$, resp. smaller than $3/2$, the mean shape is negligible, resp.~of the same order, resp.~dominant, compared to the fluctuations.\\
Notice that for $\alpha \le 1$, the features of the limiting mean shape and fluctuations depend on the step distribution of our interface (through the log-Laplace functional) while they are ``universal'' for $\alpha > 1$. The case $\alpha = 3/2$ is already covered in~\cite{EthLab15}, while the case $\alpha = 1$ can be deduced from previous results of Dobrushin and Hryniv~\cite{DobrushinHryniv} on paths of random walks conditioned on having a given large area.

\smallskip

\noindent We also derive the asymptotics of the partition function $Z_N$.
\begin{proposition}\label{Prop:Partition}
	As $N\rightarrow\infty$ we have
	\begin{equation*}
	\log \frac{Z_N}{2^{2N}} = \begin{cases} \frac{\sigma^2}{6}(2N)^{3-2\alpha} + \cO(N^{1-2\alpha}\vee N^{5-4\alpha})&\mbox{ if }\alpha \in (1,\infty)\;,\\
	(2N) \int_0^1 L\big(\sigma (1-2x)\big) dx +\cO(1)&\mbox{ if }\alpha=1\;,\\
	\frac{\sigma}{2}(2N)^{2-\alpha} - 2N \log 2 + \cO(N^{\alpha})&\mbox{ if }\alpha \in (0,1)\;.
	\end{cases}
	\end{equation*}
\end{proposition}

Finally, let us observe that all the results presented above can be extended to a more general class of static models: namely, to paths of random walks having positive probability of coming back to $0$ after $2N$ steps and whose step distribution admits exponential moments.

\subsection{Fluctuations at equilibrium}

We turn our attention to the dynamics. Below, $\dot{W}$ will denote a space-time white noise on $[0,\infty)\times I_\alpha$, that is, a centred Gaussian random distribution such that for any two functions $f,g \in L^2\big([0,\infty)\times I_\alpha\big)$, we have $\E\dot{W}(f)\dot{W}(g) = \langle f,g\rangle$. We will denote by $S(t,k)$ the dynamic height function and by $X(t,k) = X_k(t) \in \{-1,1\}$ the occupation variables at time $t$ and site $k$.\\

For $\alpha \geq 1$, we set
\begin{equation*}
u^N(t,x) := \frac{S(t(2N)^2,x2N) - \Sigma_\alpha^N(x)}{\sqrt{2N}}\;,\quad x\in [0,1]\;,\; t\geq 0\;,
\end{equation*}
while for $\alpha < 1$, we set
\begin{equation*}
u^N(t,x) := \frac{S(t(2N)^{2\alpha},N + x(2N)^\alpha)-\Sigma_\alpha^N(x)}{(2N)^{\frac{\alpha}{2}}} \;,\quad x\in I^N_\alpha\;,\; t\geq 0\;.
\end{equation*}
For convenience, we set $\Sigma_1(x) = \lim_{N\rightarrow\infty} \Sigma^N_1(x)/(2N) = \int_0^x L'(\sigma(1-2y)) dy$ for all $x\in [0,1]$, and, for $\alpha < 1$, $\Sigma_\alpha(x) = \lim_{N\rightarrow\infty} (\Sigma^N_\alpha(x) - N)/(2N)^\alpha = x + \int_{-x}^\infty \big(L'(2\sigma y) - 1\big) dy$ for all $x\in\R$.

\begin{theorem}\label{Th:Dynamic}
Assume that the process starts from the invariant measure $\mu_N$. Then, as $N\rightarrow\infty$, the process $u^N$ converges in distribution to the process $u$ where
\begin{enumerate}
	\item For $\alpha \in (1,\infty)$, $u$ solves
	\begin{align}\label{SHE}
	\begin{cases}
	\partial_t u = \frac{1}{2} \partial^2_x u + \dot{W}\;,\quad x\in(0,1)\;,\\
	u(t,0)=u(t,1) = 0\;,
	\end{cases}
	\end{align}
	started from an independent realisation of $B_\alpha$,
	\item For $\alpha = 1$, $u$ solves
	\begin{align}\label{SHE2}
	\begin{cases}
	\partial_t u = \frac{1}{2} \partial^2_x u -2\sigma \partial_x \Sigma_1\, \partial_x u + \sqrt{1-(\partial_x \Sigma_1)^2}\,\dot{W}\;,\quad x\in(0,1)\;,\\
	u(t,0)=u(t,1) = 0\;,
	\end{cases}
	\end{align}
	started from an independent realisation of $B_1$,
	\item For $\alpha \in (0,1)$, $u$ solves
	\begin{equation}\label{SHE3}
	\partial_t u = \frac{1}{2} \partial^2_x u -2\sigma \partial_x \Sigma_\alpha\, \partial_x u +  \sqrt{1-(\partial_x \Sigma_\alpha)^2}\,\dot{W}\;,\quad x\in\R\;,
	\end{equation}
	started from an independent realisation of $B_\alpha$.
\end{enumerate}
In all cases, convergence holds in the Skorohod space $\bbD([0,\infty),\cC(I_\alpha))$.
\end{theorem}
Once again, notice the specific behaviour when $\alpha \leq 1$. De Masi, Presutti and Scacciatelli~\cite{DeMasi89} and Dittrich and G\"artner~\cite{DitGar91} prove convergence of the fluctuations of the weakly asymmetric simple exclusion process (WASEP) on the full line when the asymmetry scales like $\epsilon$ and time is sped up by $\epsilon^{-2}$: the limiting fluctuations are then given by an equation similar to \eqref{SHE3}. Let us also cite the work of Derrida, Enaud, Landim and Olla~\cite{DerridaEnaudLandimOlla} on a related model interacting with reservoirs.\\
An important ingredient in the proof of this theorem is the Boltzmann-Gibbs principle, which is adapted to the present setting in Proposition \ref{Prop:BG}.

\subsection{Hydrodynamic limit}

The subsequent question we address concerns the convergence to equilibrium: suppose we start from some initial profile $S_0$ at time $0$, how does the interface reach its stationary state ?
We consider the following rescaled height function
\begin{equation*}
m^N(t,x) := \frac{S\big(t(2N)^{(\alpha+1)\wedge 2},x2N\big)}{2N}\;,\quad t\geq 0\;,\quad x\in[0,1]\;.
\end{equation*}
Notice that under this scaling, at any time $t\geq 0$ the profile $x\mapsto m^N(t,x)$ is $1$-Lipschitz.
\begin{theorem}\label{Th:Hydro}
Let $\alpha\in(0,\infty)$. We assume that the initial profile $m^N(0,\cdot)$ is deterministic and converges uniformly to some continuous profile $m_0(\cdot)$. Then, the process $m^N$ converges in probability, in the Skorohod space $\bbD([0,\infty),\cC([0,1]))$, to the deterministic process $m$ where:
\begin{enumerate}
\item If $\alpha \in (1,\infty)$, $m$ is the unique solution of the linear heat equation
	\begin{equation}\label{PDEHeat}
	\begin{cases}\partial_t m = \frac{1}{2} \partial^2_x m\;,\\
	m(t,0)=m(t,1)=0\;,\quad m(0,\cdot) = m_0(\cdot)\;.\end{cases}	
	\end{equation}
\item If $\alpha = 1$, $m$ is the solution of the following heat equation with non-linear drift
	\begin{equation}\label{PDEHC}
	\begin{cases}\partial_t m = \frac{1}{2} \partial^2_x m + \sigma\big(1- (\partial_x m)^2\big)\;,\\
	m(t,0)=m(t,1)=0\;,\quad m(0,\cdot) = m_0(\cdot)\;.\end{cases}
	\end{equation}
\item If $\alpha \in (0,1)$, $m$ is the solution of the following Hamilton-Jacobi equation
	\begin{equation}\label{PDEHJ}
	\begin{cases}\partial_t m = \sigma\big(1- (\partial_x m)^2\big)\;,\\
	m(t,0)=m(t,1)=0\;,\quad m(0,\cdot) = m_0(\cdot)\;.\end{cases}
	\end{equation}
\end{enumerate}
\end{theorem}
Compare (\ref{PDEHeat}), (\ref{PDEHC}) and (\ref{PDEHJ}) and observe the competition between the Laplacian and the non-linear term: as $\alpha$ decreases, the non-linear term becomes predominant. Notice that (\ref{PDEHeat}) and (\ref{PDEHC}) are well-posed parabolic PDEs, while (\ref{PDEHJ}) does not admit unique weak solutions so that one needs to specify the notion of solutions considered, see below. The convergence result in the case $\alpha=1$ is similar to the results of Kipnis, Olla and Varadhan~\cite{KOV} and of G\"artner~\cite{Gartner88} who consider the WASEP respectively on the torus and on the line $\Z$; let us also cite the work of Enaud and Derrida~\cite{Enaud} on a similar model interacting with reservoirs. 

Let us now be more precise on the notion of solution of (\ref{PDEHJ}) that we consider here. For any Lipschitz function $m_0$, we let $\eta_0(\cdot) = (\partial_x m_0(\cdot) + 1)/2$ and we say that $m$ is solution of (\ref{PDEHJ}) if $m(t,x)=\int_0^x \big(2\eta(t,y)-1\big) dy$ where $\eta$ is the entropy solution of the Burgers equation with zero-flux boundary condition
\begin{equation}\label{PDEBurgersDensity}
	\begin{cases}\partial_t \eta = 2\sigma\partial_x\big(\eta(1-\eta)\big) \;,\quad x\in (0,1)\;,\quad t>0\;,\\
	\eta(t,x)(1-\eta(t,x))=0\;,\quad x\in\{0,1\}\;,\quad t>0\;,\\
	\eta(0,\cdot) = \eta_0\;.\end{cases}
\end{equation}
The precise formulation of the associated entropy conditions is given in Proposition \ref{Prop:EntropySolution}. Let us recall that this conservation law does not have unique solutions in general, and that one needs to impose further conditions - here the entropy conditions - to recover uniqueness. Let us mention that the interpretation of the boundary conditions for this type of equations is not elementary. In the case of Dirichlet boundary conditions, say $\eta(t,0)=a$ and $\eta(t,1)=b$, the solution theory does not yield solutions that satisfy the prescribed values at the boundaries at all times, but they rather satisfy the so-called BLN conditions at the boundaries: we refer to Bardos, Le Roux and N\'ed\'elec~\cite{Bardos} for the BV setting and to Otto~\cite{Otto} for the $L^\infty$ setting. On the other hand, zero-flux boundary conditions are simpler and one can impose to the solution to indeed have a zero-flux at the boundaries at almost every time. We refer to B\"urger, Frid and Karlsen~\cite{BFK07} for the complete solution theory.

It turns out that the solution of (\ref{PDEBurgersDensity}) coincides with the solution of the same PDE with appropriate Dirichlet boundary conditions: one simply needs to impose $\eta(t,0)=1$ and $\eta(t,1)=0$. This fact can be heuristically explained as follows.\\
If we consider the same dynamics on the whole line $\bbZ$, then the hydrodynamic limit is given by the same PDE but on the whole line $\R$: this was proved by Rezakhanlou~\cite{Reza} when the asymmetry does not depend on $N$, and, as we will see, the proof extends to the case where the asymmetry is not too weak. Now, if we start with particles at all negative sites and no particle at positive sites, then the hydrodynamic limit is constant in time: the system does not evolve at the macroscopic scale, and non-trivial behaviours can be observed only at a smaller scale around the origin. A coupling argument then shows that adding particles at positive sites does not modify the hydrodynamic behaviour on negative sites. Similarly, if we start with particles at all sites $k\le 2N$ and no particle after site $2N$, then the hydrodynamic limit remains constant. A coupling argument shows that if we remove some particles on $\{1,\ldots,2N\}$, the hydrodynamic behaviour on $\{2N+1,2N+2,\ldots\}$ remains unchanged.\\
Consequently, if we consider the same dynamics on the whole line $\bbZ$ and if we start with particles at all negatives sites, no particle after site $2N$, and the same initial configuration of particles on $\{1,\ldots,2N\}$ as in our original system, then the initial condition outside the segment remains invariant at the level of the hydrodynamic limit so that one should get the ``effective'' boundary conditions $\eta(t,0)=1$ and $\eta(t,1)=0$.\\

The proof of the theorem in the case $\alpha < 1$ thus mainly consists in showing convergence of the density of particles
\begin{equation}\label{Eq:DefrhoN}
\rho^N(t,dx)=\frac1{2N}\sum_{k=1}^{2N} \eta^N_t(k)\, \delta_{\frac{k}{2N}}(dx)\;,\quad \eta^N_t(k) = \frac{X\big(t(2N)^{1+\alpha},k\big) + 1}{2}\;,
\end{equation}
towards the deterministic process $\rho(t,dx)=\eta(t,x)dx$ where $\eta$ is the entropy solution of the Burgers equation with the above Dirichlet conditions.

This convergence result is taken from~\cite{LabbeKPZ}, it is in the flavour of the works of Rezakhanlou~\cite{Reza} and Bahadoran~\cite{Baha}.

\bigskip

Observe that in the case $\alpha \in (1,3/2)$ and under the invariant measure $\mu_N$, the interface is of order $N^{2-\alpha} \ll N$. Therefore, when the process starts from an initial condition which is at most of order $N^{2-\alpha}$, it is natural to derive the hydrodynamic limit at this finer scale $N^{2-\alpha}$. (This is no longer relevant when $\alpha \geq 3/2$ since, then, the fluctuations are dominant.) If we set
\begin{equation*}
v^N(t,x) := \frac{S(t(2N)^2,x2N)}{(2N)^{2-\alpha}}\;,\quad t\geq 0\;,\quad x\in[0,1]\;,
\end{equation*}
then it is possible to show that the sequence $v_N$ converges to the unique solution of the following linear heat equation 
\begin{equation}\label{PDEHeat2}
\begin{cases}\partial_t v = \frac{1}{2} \partial^2_x v + \sigma\;,\\
v(t,0)=v(t,1)=0\;,\quad v(0,\cdot) = v_0(\cdot)\;.\end{cases}	
\end{equation}
We do not provide the details on this convergence, but it relies on essentially the same arguments as the case $\alpha =1$ of the previous theorem.

\subsection{KPZ fluctuations}

From now on, we consider the flat initial condition
\begin{equation*}
S(0,k) = k \mbox{ mod } 2\;,\quad k\in\{0,\ldots,2N\}\;.
\end{equation*}
Let us provide explicitly the solution of the Hamilton-Jacobi equation (\ref{PDEHJ}) starting from the flat initial condition:
\begin{equation}\label{Eq:ExplicitHydroFlat}
m(t,x) = x\wedge (1-x) \wedge (\sigma t)\;,\quad t > 0\;,\quad x\in [0,1]\;,
\end{equation}
see Figure \ref{FigKPZ} for an illustration. Notice that the macroscopic stationary state is reached at the finite time $T=1/(2\sigma)$. This is an important feature of the hydrodynamic limit for $\alpha \in (0,1)$: indeed, when $\alpha \geq 1$, the hydrodynamic limit is parabolic and reaches its stationary state in \textit{infinite} time.
\begin{figure}
\begin{center}
\begin{tikzpicture}[scale=2]
\draw[-,gray] (0,0) node[below]{\tiny $0$} -- (3,0) node[below]{\tiny $2N$};
\draw[-,gray] (0,0) -- (1.5,1.5) -- (3,0);

\draw[-,thick] (0,0) -- (3,0);

\draw[-,dashed,thick] (0,0) -- (0.75,0.75) -- (2.25,0.75) -- (3,0);

\draw[-,dotted,thick] (0,0) -- (1.5,1.5) -- (3,0);

\draw[gray] (1.5,-0.04)node[below]{\tiny $N$}--(1.5,0.04);
\draw[gray] (0.75,-0.04)node[below]{\tiny$\frac{\lambda_N}{\gamma_N}t$}--(0.75,0.04);
\draw[gray] (2.25,-0.04)node[below]{\tiny$2N-\frac{\lambda_N}{\gamma_N}t$}--(2.25,0.04);

\draw[thin,blue] (1.25,0.65) -- (1.75,0.65) -- (1.75,0.85) -- (1.25, 0.85) -- (1.25,0.65);
\draw[<->,blue] (1.25,0.55) --(1.5,0.55)node[below]{\tiny $N^{2\alpha}$} -- (1.75,0.55);
\draw[<->,blue] (2.5,0.65) --(2.5,0.75)node[right]{\tiny $N^{\alpha}$} -- (2.5,0.85);

\end{tikzpicture}
\end{center}
\caption{A plot of (\ref{Eq:ExplicitHydroFlat}): the bold black line is the initial condition, the dashed line is the solution at some time $0 < t < 1/(2\sigma)$, and the dotted line is the solution at the terminal time $1/(2\sigma)$. The blue box corresponds to the window where we see KPZ fluctuations.}\label{FigKPZ}
\end{figure}
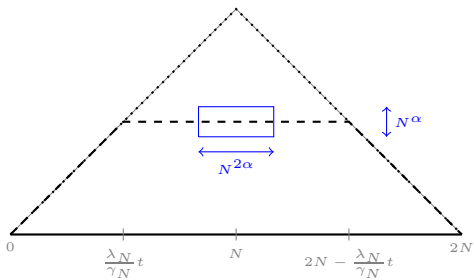

We are now interested in fluctuations around this hydrodynamic limit. The reader familiar with the Kardar Parisi Zhang (KPZ) equation would probably guess that it should arise in our setting. Let us first recall the famous result of Bertini and Giacomin~\cite{BG97} in that direction. Consider the WASEP on the infinite lattice $\Z$ with jump rates $1/2 + \sqrt{\epsilon}$ to the left and $1/2$ to the right. If one starts from a flat initial profile, then results in~\cite{Gartner88,DeMasi89} ensure that the hydrodynamic limit grows evenly at speed $\sqrt{\epsilon}$. Then, Bertini and Giacomin look at the fluctuations around this hydrodynamic limit and show that the random process $\sqrt{\epsilon}\,(S(t\epsilon^{-2}, x\epsilon^{-1}) - \epsilon^{-3/2} t)$ converges to the solution of the KPZ equation, whose expression is given in (\ref{KPZ}) below (in Bertini and Giacomin's case, $\sigma =1/2$).

Although our setting is similar to the one considered by Bertini and Giacomin, the ``zero-flux" boundary condition induces a major difference: our process admits a reversible probability measure, while this is not the case on the infinite lattice $\Z$. However, if one starts the interface ``far" from equilibrium, then we are in an irreversible setting up to the time needed by the interface to reach equilibrium, and one would expect the fluctuations to be described by the KPZ equation.

Bertini and Giacomin's result suggests to rescale the height function by $1/(2N)^\alpha$, the space variable by $(2N)^{2\alpha}$ and the time variable by $(2N)^{4\alpha}$. The space scaling immediately forces one to take $\alpha\leq 1/2$ since, otherwise, the lattice $\{0,1,\ldots,2N\}$ would be mapped onto a singleton in the limit. It happens that the geometry of our model imposes a further constraint: Theorem \ref{Th:Hydro} and Equation (\ref{Eq:ExplicitHydroFlat}) show that the interface reaches the stationary state in finite time in the time scale $(2N)^{\alpha +1}$; therefore, as soon as $4\alpha > \alpha + 1$, Bertini and Giacomin's scaling yields an interface which is immediately at equilibrium in the limit $N\rightarrow\infty$. Consequently, we have to restrict $\alpha$ to $(0,1/3]$ for this scaling to be non-trivial.

We set
\begin{equation}\label{Eq:hN}
h^N(t,x) := \gamma_N S\big(t(2N)^{4\alpha},N + x(2N)^{2\alpha}\big) - \lambda_N t \;,
\end{equation}
where
\begin{equation}\label{Eq:GammaLambda}
	\gamma_N := \frac12 \log \frac{p_N}{1-p_N}\;,\quad c_N := \frac{(2N)^{4\alpha}}{e^{\gamma_N} + e^{-\gamma_N}} \;,\quad  \lambda_N := c_N( e^{\gamma_N} -2 + e^{-\gamma_N})\;.
\end{equation}

\noindent The following result was established in~\cite{LabbeKPZ}

\begin{theorem}\label{Th:KPZ}
Take $\alpha \in (0,1/3]$ and consider the flat initial condition. As $N\rightarrow\infty$, the sequence $h^N$ converges in distribution to the solution of the KPZ equation:
\begin{equation}\label{KPZ}
\begin{cases}
\partial_t h = \frac{1}{2} \partial^2_x h - \sigma (\partial_x h)^2 + \dot{W}\;,\quad x\in\R\;,\quad t > 0\;,\\
h(0,x)=0\;.
\end{cases}
\end{equation}
The convergence holds on $\bbD\big([0,T),\cC(\R)\big)$ where $T=1/(2\sigma)$ when $\alpha=1/3$, and $T=\infty$ when $\alpha < 1/3$. Here $\bbD\big([0,T),\cC(\R)\big)$ is endowed with the topology of uniform convergence on compact subsets of $[0,T)$.
\end{theorem}

Observe that for $\alpha = 1/3$, $T$ is the time needed by the hydrodynamic limit to reach the stationary state. Indeed, in that case the time-scale of the hydrodynamic limit coincides with the time-scale of the KPZ fluctuations. Although one could have thought that the fluctuations continuously vanish as $t\uparrow T$, our result show that they don't: the limiting fluctuations are given by the solution of the KPZ equation, restricted to the time interval $[0,T)$. This means that the fluctuations suddenly vanish at time $T$; let us give a simple explanation for this phenomenon. At any time $t\in [0,T)$, the particle system is split into three zones: a high density zone $\{1,\ldots,\frac{\lambda_N}{\gamma_N} t\}$, a low density zone $\{2N- \frac{\lambda_N}{\gamma_N} t, \ldots,2N\}$ and, in between, the bulk where the density of particles is approximately $1/2$, we refer to Figure \ref{FigKPZ}. The KPZ fluctuations occur in a window of order $N^{2\alpha}$ around the middle point of the bulk: from the point of view of this window, the boundaries of the bulk are ``at infinity" but move ``at infinite speed". Therefore, inside this window the system does not feel the effect of the boundary conditions until the very final time $T$ where the boundaries of the bulk merge.

Let us recall that the KPZ equation is a singular SPDE: indeed, the solution of the linearised equation is not differentiable in space so that the non-linear term would involve the square of a distribution. While it was introduced in the physics literature~\cite{KPZ86} by Kardar, Parisi and Zhang, a first rigorous definition was given by Bertini and Giacomin~\cite{BG97} through the so-called Hopf-Cole transform $h\mapsto \xi = e^{-2\sigma h}$ that maps formally the equation (\ref{KPZ}) onto
\begin{equation}\label{mSHE}
\begin{cases}
	\partial_t \xi = \frac{1}{2} \partial^2_x \xi +  2\sigma \xi\dot{W}\;,\quad x\in\R\;,\quad t > 0\;,\\
	\xi(0,x)=1\;.
\end{cases}
\end{equation}
This SPDE is usually referred to as the multiplicative stochastic heat equation: it admits a notion of solution via It\^o integration, see for instance~\cite{DPZ,Walsh}. M\"uller~\cite{Mueller} showed that the solution is strictly positive at all times, if the initial condition is non-negative and non-zero. This allows to take the logarithm of the solution, and then, one can define the solution of (\ref{KPZ}) to be $h:=-\log \xi / 2\sigma$. This is the notion of solution that we consider in Theorem \ref{Th:KPZ}.\\
There exists a more direct definition of this SPDE (restricted to a bounded domain) due to Hairer~\cite{HairerKPZ, HairerReg} via his theory of regularity structures. Let us also mention the notion of ``energy solution" introduced by Gon\c{c}alves and Jara~\cite{GJEnergy}, for which uniqueness has been proved by Gubinelli and Perkowski~\cite{GubPerEnergy}. It provides a new framework for characterising the solution to the KPZ equation but it requires the equation to be taken under its stationary measure.\\
For related convergence results towards KPZ, we refer to Amir, Corwin and Quastel~\cite{ACQ}, Dembo and Tsai~\cite{DemboTsai}, Corwin and Tsai~\cite{CT} and Corwin, Shen and Tsai~\cite{CST}. We also point out the reviews of Corwin~\cite{Corwin}, Quastel~\cite{QuastelKPZ} and Spohn~\cite{Spohn}.

\bigskip

The paper is organised as follows. In Section \ref{Section:InvMeas}, we study the scaling limit of the invariant measure. Section \ref{Section:Fluctuations} is devoted to the fluctuations at equilibrium. In Section \ref{Section:CVEq} we prove Theorem \ref{Th:Hydro} on the hydrodynamic limit, and in Section \ref{Section:KPZ} we present the proof of the convergence of the fluctuations to the KPZ equation. Some technical bounds are postponed to the Appendix. The sections are essentially independent: at some localised places, we will rely on results obtained on the static of the model in Section 2.

\section{The invariant measure}\label{Section:InvMeas}

Let $\mbox{Be}(q)$ denote the Bernoulli $\pm 1$ distribution with parameter $q\in[0,1]$, and let $L$ be the log-Laplace functional associated with $\mbox{Be}(1/2)$, namely $L(h) = \log \cosh h$ for all $h\in\R$. For each given $N\geq 1$, we will work on the set $\{-1,+1\}^{2N}$ endowed with its natural sigma-field. The canonical process will be denoted by $X_1,\ldots,X_{2N}$, and will be viewed as the steps of the walk $S(n):= \sum_{k\leq n} X_k$. Recall that $A(S)$ is the area under the walk $S$, defined in (\ref{Eq:muN}) and recall that $p_N = 1/2 + \sigma/(2N)^\alpha + \cO(1/N^{2\alpha})$.\\

The strategy of the proof consists in introducing an auxiliary measure $\nu_N$ which is the same as $\mu_N$ except that, under $\nu_N$, the walk is not conditioned on coming back to $0$ but satisfies $\nu_N[S(2N)]=0$. This makes $\nu_N$ more amenable to limit theorems. In particular, we establish a Central Limit Theorem and a Local Limit Theorem for the marginals of the walk under $\nu_N$. Since $\mu_N$ is equal to $\nu_N$ conditioned on the event $S(2N)=0$, and since the $\nu_N$-probability of this event can be estimated with the Local Limit Theorem, we are able to get the convergence of the marginals of the walk under $\mu_N$. The tightness is obtained by similar considerations. Let us now provide the details.\\

Let $\pi_N$ be the law of the simple random walk, that is
\begin{equation*}
\pi_N := \underset{k=1}{\overset{2N}{\otimes}} \mbox{Be}(1/2)\;,
\end{equation*}
and let $\nu_N$ be the measure defined by
\begin{equation}\label{Eq:DefnuN}
\frac{d\nu_N}{d\pi_N} = \frac1{Z'_N} \Big(\frac{p_N}{1-p_N}\Big)^{\frac{A(S)}{2}} e^{\rho_N S(2N)}\;,\qquad Z'_N = e^{L_S(h^N)}\;,
\end{equation}
where $L_S(h^N):= \sum_{k=1}^{2N} L(h^N_k)$ and
\begin{equation*}
\rho_N = -\frac{2\sigma}{(2N)^\alpha}\Big(N+\frac{1}{2}\Big)\;,\quad h^N_{k} = \frac{2\sigma}{(2N)^\alpha}\Big(N-k+\frac{1}{2}\Big)\;,\quad k\in\{1,\ldots,2N\}\;.
\end{equation*}
\begin{remark}
Under the measure $\nu_N$, the total number of particles is not equal to $N$ almost surely, but is equal to $N$ in mean. The measure $\nu_N$ can be seen as a mixture of $2N+1$ measures, each of them being supported by an hyperplane of configurations with $\ell \in\{0,\ldots,2N\}$ particles. It is easy to check that our dynamics is reversible with respect to each of these measures, and therefore, with respect to $\nu_N$.
\end{remark}
\begin{lemma}
The measure $\nu_N$ satisfies
\begin{equation}\label{Eq:piNnuN}
\frac{d\nu_N}{d\pi_N} = \frac1{Z'_N} e^{\sum_k h^N_k X_k}\;,\quad \nu_N = \underset{k=1}{\overset{2N}{\otimes}} \mbox{Be}(q^N_k)\;,
\end{equation}
where $q^N_k = \big(L'(h^N_k)+1\big)/2$.
\end{lemma}
\begin{proof}
Notice that
$$ \Big(\frac{p_N}{1-p_N}\Big)^{\frac{A(S)}{2}} e^{\rho_N S(2N)} = e^{\frac{2\sigma}{(2N)^\alpha} A(S) + \rho_N S(2N)}\;,$$
and
$$ \frac{2\sigma}{(2N)^\alpha} A(S) + \rho_N S(2N) = \sum_{k=1}^{2N} X_k \Big((2N-k+1)\frac{2\sigma}{(2N)^\alpha} + \rho_N\Big) = \sum_{k=1}^{2N} X_k h^N_k\;.$$
Since $\pi_N$ is a product of Bernoulli measures and given the expression of the Radon-Nikodym derivative above, we deduce that $\nu_N$ is also a product measure of Bernoulli distributions with parameters
$$ \nu_N(X_k=1) = \frac{e^{h^N_k}}{e^{L(h^N_k)}} \pi_N(X_k=1) = \frac{e^{h^N_k}}{2 \cosh h^N_k} = \frac{L'(h^N_k)+1}{2}\;.$$
\end{proof}
From there, simple calculations yield
\begin{equation}
\nu_N\big[S(k)\big] = \sum_{i=1}^k L'(h^N_i)\;,\quad \Var_{\nu_N}\big[S(k),S(\ell)\big] = \sum_{i=1}^{k\wedge\ell} L''(h^N_i)\;.
\end{equation}
Observe that the curve $\Sigma_\alpha^N$ defined in (\ref{Eq:DefSigmaN}) is nothing but the mean of $S$ under $\nu_N$. Recall the definition of $q_\alpha$ from Theorem \ref{Th:Static}. A simple calculation yields the following asymptotics. For $\alpha \geq 1$, we have
\begin{equation}\label{Eq:Mean1}
\Sigma_\alpha^N(x) =\begin{cases}
(2N)^{2-\alpha}\sigma x(1-x) + \cO(N^{4-3\alpha}) \quad&\alpha > 1\;,\\
2N\int_0^x L'(\sigma(1-2y)) dy + \cO(1) \quad& \alpha = 1\;,
\end{cases} 
\end{equation}
and \begin{equation}
\Var_{\nu_N}\big[S(x2N),S(y2N)\big] = \begin{cases} (2N)\, q_\alpha(0,x\wedge y) + \cO(N^{3-2\alpha})\quad &\alpha > 1\;,\\
(2N)\, q_\alpha(0,x\wedge y) + \cO(1)\quad &\alpha = 1\;,\end{cases}
\end{equation}
for all $x,y \in [0,1]$.  For $\alpha < 1$, we find
\begin{equation}\label{Eq:Mean2}
\Sigma_\alpha^N(x) = N + (2N)^\alpha \Big(x + \int_{-x}^\infty \big(L'(2\sigma y) - 1\big) dy \Big) + \cO(1)\quad\alpha < 1\;,\quad
\end{equation}
and
\begin{equation}\label{Eq:Var2}
\Var_{\nu_N}\big[S(N + x(2N)^\alpha),S(N+y(2N)^\alpha)\big] = (2N)^\alpha q_\alpha(-\infty,x\wedge y) + \cO(1)\;,\quad \alpha < 1\;,
\end{equation}
for all $x,y\in\R$.

\smallskip

The important observation for the sequel is the following result.
\begin{lemma}
\begin{equation}\label{Eq:muNpiN}
\mu_N(\cdot) = \nu_N(\cdot \,|\, S(2N)=0)\;.
\end{equation}
\end{lemma}
\begin{proof}
Let $C_N$ be the set of all discrete bridges from $(0,0)$ to $(2N,0)$, then we have
\begin{equation*}
Z_N = \sum_{S\in C_N} \Big(\frac{p_N}{1-p_N}\Big)^{\frac12 A(S)} = 2^{2N}\pi_N\Big[\Big(\frac{p_N}{1-p_N}\Big)^{\frac12 A(S)} \tun_{\{S(2N)=0\}}(S)\Big]\;,
\end{equation*}
since $\pi_N$ is the uniform measure on the set of all lattice paths with $2N$-steps and since the latter set has cardinal $2^{2N}$. Hence, for any subset $D$ of $C_N$, we have (for all $S\in D$, $S(2N)=0$)
\begin{align*}
\nu_N(D \,|\, S_{2N}=0) &= \frac{\pi_N\big[\big(\frac{p_N}{1-p_N}\big)^{\frac12 A(S)} e^{-L_s(h^N)} \tun_D(S)\big]}{\pi_N\big[\big(\frac{p_N}{1-p_N}\big)^{\frac12 A(S)} e^{-L_s(h^N)} \tun_{\{S(2N)=0\}}(S)\big]}\\
&= \frac1{Z_N} 2^{2N}\pi_N\Big[\Big(\frac{p_N}{1-p_N}\Big)^{\frac12 A(S)} \tun_D(S)\Big]\\
&= \frac{1}{Z_N} \sum_{S\in C_N} \Big(\frac{p_N}{1-p_N}\Big)^{\frac12 A(S)} \tun_D(S)\\
&= \mu_N(D)\;,
\end{align*}
and (\ref{Eq:muNpiN}) follows.
\end{proof}

Recall that for $\alpha \geq 1$ we have
\begin{equation*}
u^N(x) = \frac{S(x2N) - \Sigma_\alpha^N(x)}{\sqrt{2N}}\;,\quad x\in [0,1]\;,
\end{equation*}
while for $\alpha < 1$
\begin{equation*}
u^N(x) = \frac{S(N + x(2N)^\alpha)-\Sigma_\alpha^N(x)}{(2N)^{\frac{\alpha}{2}}} \;,\quad x\in [-N/(2N)^\alpha,N/(2N)^\alpha]\;.
\end{equation*}

Until the end of the section, $k$ will denote an integer and $\vec{x}=(x_1,\ldots,x_k)$ will be an element of $(0,1]^k$ if $\alpha \ge 1$, of $(-\infty,+\infty]^k$ if $\alpha \in (0,1)$. It will be convenient to write $u^N(\vec{x}) = (u^N(x_1),\ldots,u^N(x_k))$. Also, we use the convenient notation $u^N(+\infty)$ to denote $u^N(N/(2N)^\alpha)$ when $\alpha \in (0,1)$. For $\alpha \in (0,1)$, we define the lattice approximation $x^N$ of $x$ by setting
\begin{equation*}
x^N := \frac{\lfloor x(2N)^{\alpha}\rfloor}{(2N)^{\alpha}}\;, \quad x\in\R\;,
\end{equation*}
$x^N :=+\infty$ for $x=+\infty$, and $\vec{x}^N := ( x_1^N,\ldots,x_k^N)$ for all $\vec{x}$ as above. We also let $x_f = 1$ when $\alpha \ge 1$ and $x_f = +\infty$ when $\alpha \in (0,1)$.\\
If $\alpha \ge 1$, we let $\tilde{B}_\alpha$ be the centred Gaussian process on $[0,1]$ whose covariance is given by $q_\alpha(0,\cdot\wedge\cdot)$. If $\alpha \in (0,1)$, we adopt the same definition except that the process lives on $\R$ and that its covariance is given by $q_\alpha(-\infty,\cdot\wedge\cdot)$. Notice that in this last case, $\tilde{B}_\alpha(x)$ converges to a finite limit when $x\to +\infty$ since this is a martingale bounded in $L^2$.\\
It is simple to check that $B_\alpha$ is obtained by conditioning $\tilde{B}_\alpha$ to vanish at $x_f$. More precisely, if we denote by $g^{x_1,\ldots,x_k}_\alpha$ and $\tilde{g}^{x_1,\ldots,x_k,x_f}_\alpha$ the probability densities of the vectors $(B_\alpha(x_1),\ldots,B_\alpha(x_k))$ and $(\tilde{B}_\alpha(x_1),\ldots,\tilde{B}_\alpha(x_k),\tilde{B}_\alpha(x_f))$, then we have
\begin{equation}\label{Eq:Densities}
g^{x_1,\ldots,x_k}_\alpha(y_1,\ldots,y_k) = \frac{\tilde{g}^{x_1,\ldots,x_k,x_f}_\alpha(y_1,\ldots,y_k,0)}{\tilde{g}^{x_f}_\alpha(0)} \;.
\end{equation}

The proof of Theorem \ref{Th:Static} is divided into the following four steps.
\paragraph{Step 1: Convergence of the marginals under $\nu_N$.}
\begin{lemma}\label{Lemma:CLT}
The vector $u^N(\vec{x})$ under $\nu_N$ converges in distribution to $\tilde{B}_\alpha(\vec{x})$.
\end{lemma}
\begin{proof}
The proof is classical, we only provide the details for the case $\alpha \in (0,1)$ as the other case is treated similarly. Until the end of the proof, $i$ denotes the complex number $\sqrt{-1}$. For each $j\in\{1,\ldots,k\}$, we define $k_j:=N+\lfloor x_j(2N)^\alpha\rfloor$. For all $\vec{t} = (t_1,\ldots,t_k) \in \R^{k}$, let
\begin{equation*}
L_k(\vec{t}) := \sum_{\ell=1}^{2N} L\Big(i \sum_{j=1}^k t_j \tun_{\{\ell \leq k_j\}} + h_\ell^N\Big)\;,
\end{equation*}
so that $L_k(0) = L_S(h^N) = \sum_{k=1}^{2N} L(h^N_k)$ and
\begin{equation*}
\log\nu_N\Big[e^{\sum_{j=1}^k i t_j S(N+x_j^N (2N)^\alpha)}\Big] = L_k(\vec{t})-L_k(0)\;.
\end{equation*}
It is simple to check that
\begin{equation*}
{\partial^2_{t_j,t_m} L_k}_{|\vec{t}=0} = -\Var_{\nu_N}\big[S(k_j),S(k_m) \big]\;.
\end{equation*}
Fix $\vec{t}\in\R^k$. Using \eqref{Eq:Var2} and a Taylor expansion at the second line, we get
\begin{align*}
\log \nu_N\Big[ e^{ i\langle \vec{t}, u^N(\vec{x}^N) \rangle} \Big] &= L_k\big(\vec{t}(2N)^{-\frac{\alpha}{2}}\big)-L_k(0) - \frac{1}{(2N)^{\frac{\alpha}{2}}}\langle \vec{t} , \nabla L_k(0) \rangle\\
&= -\frac12\sum_{j,\ell=1}^{k} t_j t_\ell\, q_\alpha(-\infty,x_j\wedge x_\ell) + \cO\Big(\frac{|t|^3}{N^{\alpha/2}} + \frac{|t|^2}{N^{\alpha}}\Big)\;,
\end{align*}
so that the characteristic function of the vector $u^N(\vec{x}^N)$ converges pointwise to the characteristic function of the Gaussian vector of the statement. Since the difference between $u^N(\vec{x}^N)$ and $u^N(\vec{x})$ is negligible, the lemma follows.
\end{proof}

\paragraph{Step 2: Local limit theorems under $\nu_N$.} We have the following Local Limit Theorems under $\nu_N$. Let $D^{\vec{x},N}_\alpha$ be the finite set of all $\vec{y}=(y_1,\ldots,y_k) \in \R^k$ such that $\nu_N(u^N(\vec{x}^N)=\vec{y}) > 0$.
\begin{lemma}\label{Lemma:LLT}
Uniformly over all $\vec{y}\in D^{\vec{x},N}_\alpha$ and all $N\geq 1$, we have
\begin{equation*}
\frac{(2N)^{\frac{k}{2}(\alpha\wedge 1)}}{2^k} \nu_N\big(u^N(\vec{x}^N)=\vec{y}\,\big) - \tilde{g}^{\vec{x}}_\alpha(\vec{y}) = o(1)\;.
\end{equation*}
\end{lemma}
\noindent In the case $k=1$, let $E^{x,N}_\alpha$ be the set of values $y$ such that $\nu_N(u^N(x_f)-u^N(x^N)=y) > 0$.
\begin{lemma}\label{Lemma:LLTReversed}
Uniformly over all $y\in E^{x,N}_\alpha$ and all $N\geq 1$, we have
\begin{equation*}
\frac{(2N)^{\frac{1}{2}(\alpha\wedge 1)}}{2}\, \nu_N\big(u^N(x_f)-u^N(x^N)=y\,\big) - \int_{z\in\R}\tilde{g}^{(x,x_f)}_\alpha(z,z+y)dz = o(1)\;.
\end{equation*}
\end{lemma}
\noindent Below, we provide the proof of the first lemma. The second lemma follows from exactly the same arguments, one simply has to notice that $u^N(x_f)-u^N(x^N)$ converges in law to $\tilde{B}_\alpha(x_f)-\tilde{B}_\alpha(x)$, and that $\int_{z\in\R}\tilde{g}^{(x,x_f)}_\alpha(z,z+y)dz$ is the density at $y$ of this limiting r.v.
\begin{proof}[Proof of Lemma \ref{Lemma:LLT}]
Let us prove the case $\alpha \in (0,1)$ which is the most involved. The main difference in the proof with the case $\alpha \geq 1$ lies in the fact that the forthcoming bound (\ref{Eq:BoundCaract2}) cannot be applied to all $h^N_i$ simultaneously when $\alpha \in (0,1)$. Indeed, these coefficients are not bounded uniformly over $i$ and $N$ when $\alpha \in (0,1)$. However, for any given $a>0$, they are bounded uniformly over all $i\in I_{N,a}:=[N-a(2N)^\alpha,N+a(2N)^\alpha]$ and all $N\geq 1$.\\
Without loss of generality, we can assume that $x_1 < x_2 <\ldots < x_k$ so that only $x_k$ can take the value $+\infty$. Let $\phi_h(t) = \exp(L(h+it) - L(h))$ for $t,h\in\R$. This is the characteristic function of the Bernoulli $\pm 1$ r.v.~with mean $L'(h)$ so that
\begin{equation}\label{Eq:Charact}
\phi_h(t) = \cos(t) + iL'(h)\sin(t)\;,\quad t\in\R\;,\quad h\in\R\;.
\end{equation}
In particular, the characteristic function of the r.v.~$X_i$ under $\nu_N$ is given by $\phi_{h^N_i}$. The function $\phi$ is $2\pi$-periodic and $|\phi_h(t)|\leq 1$ for all $h,t$. From (\ref{Eq:Charact}), one deduces that for any compact set $K\subset \R$, there exists $r(K) > 0$ such that
\begin{equation}\label{Eq:BoundCaract2}
\big|\phi_h(t)\big| \leq \exp(-r t^2 L''(h))\;,\qquad\forall t\in \Big[-\frac{2\pi}{3},\frac{2\pi}{3}\Big]\;,\quad\forall h\in K\;.
\end{equation}
Let $\Phi_\alpha$ denote the characteristic function of the Gaussian vector $\tilde{B}_\alpha(\vec{x})$. Classical arguments from Fourier analysis entail that for all $\vec{y}\in D^{\vec{x},N}_\alpha$
\begin{equation*}
R_N := (2N)^{\frac{\alpha k}{2}} \nu_N\big(u^N(\vec{x}^N)=\vec{y}\big) - 2^k\tilde{g}^{\vec{x}}_\alpha(\vec{y})\;,
\end{equation*}
can be rewritten as
\begin{equation*}
R_N = \frac{1}{\pi^k} \int_{D}\Phi_N(\vec{t}) e^{-i\langle \vec{t} , \vec{y} \rangle} d\vec{t} - \frac{1}{\pi^k} \int_{\R^k}\Phi_\alpha(\vec{t}) e^{-i\langle \vec{t} , \vec{y} \rangle} d\vec{t}\;,
\end{equation*}
where $\Phi_N$ is the characteristic function of $u^N(\vec{x}^N)$ under $\nu_N$ and
\begin{equation*}
D:= \Big\{\vec{t} \in \R^{k}: |t_\ell| \leq \frac{\pi}{2} (2N)^{\frac{\alpha}{2}}, \ell=1,\ldots,k \Big\}\;.
\end{equation*}
Notice that the factor $1/2$ in the definition of $D$ comes from the simple fact that our step distribution charges $\{-1,1\}$, and therefore has a maximal span equal to $2$. Then, we take $\rho \in (0,\frac{1}{2(3+k)})$ and we bound $|R_N| \pi^k$ by the sum of the following three terms
\begin{align*}
J_1 &= \int_{D_1} |\Phi_N(\vec{t})-\Phi_\alpha(\vec{t})| d\vec{t}\;,\quad D_1=[-N^{\rho\alpha},N^{\rho\alpha}]^{k}\;,\\
J_2 &= \int_{D_2} |\Phi_\alpha(\vec{t})| d\vec{t}\;,\quad D_2=\R^{k}\backslash D_1\;,\\
J_3 &= \int_{D_3} |\Phi_N(\vec{t})| d\vec{t}\;,\quad D_3=D\backslash D_1\;.
\end{align*}
It suffices to show that these three terms vanish as $N\rightarrow\infty$. Regarding $J_1$, the proof of Lemma \ref{Lemma:CLT} shows that
\begin{equation*}
|\Phi_N(\vec{t})-\Phi_\alpha(\vec{t})| \lesssim |\Phi_\alpha(\vec{t})| \Big(\frac{|\vec{t}\,|^3}{N^{\frac{\alpha}{2}}} + \frac{|\vec{t}\,|^2}{N^{\alpha}} \Big)\;,
\end{equation*}
uniformly over all $|\vec{t}|^3 = o(N^{\alpha/2})$. Since $\rho(3+k) < \frac12$, a simple calculation shows that $J_1$ goes to $0$ as $N\rightarrow\infty$. The convergence of $J_2$ to $0$ as $N\rightarrow\infty$ is a consequence of the exponential decay of the characteristic functions of Gaussian r.v. We turn to $J_3$. For each $\ell\in\{1,\ldots,k\}$, we set
\begin{equation*}
D_{3,\ell} = D_3 \cap \Big\{|t_\ell| > 3^{-\ell} N^{\rho\alpha}; \forall j>\ell, |t_j| \leq 3^{-j} N^{\rho\alpha}\Big\}\;,
\end{equation*}
so that $D_3=\cup_\ell D_{3,\ell}$. The important feature of these sets is that for all $N$ large enough
\begin{equation}\label{Eq:BoundD3l}
\frac{|t_\ell|}{2} \leq |t_\ell+\ldots+t_k| \leq \frac{2\pi}{3}(2N)^{\frac{\alpha}{2}}\;, \quad \forall \vec{t}\in D_{3,\ell}\;, \quad \forall \ell\in\{1,\ldots,k\}\;.
\end{equation}
We bound separately each term $J_{3,\ell}$ arising from the restriction of the integral in $J_3$ to $D_{3,\ell}$. Take $a>0$ such that $-a < x_1 < x_{k-1} < a$ and recall that $x_k$ can be infinite. Let $K$ be a compact set that contains all the values $h^N_i$, $i\in I_{N,a}$, and let $r$ be the corresponding constant introduced above (\ref{Eq:BoundCaract2}). We also define $j_p = N+\lfloor x_p(2N)^\alpha\rfloor$ for all $p\in\{1,\ldots,k\}$ and
\begin{equation*}
I_{N,a,\ell} :=I_{N,a} \cap (N+x_{\ell-1}(2N)^\alpha, N+x_\ell (2N)^\alpha]\;.
\end{equation*}
Using the independence of the $X_i$'s under $\nu_N$ and the fact that the modulus of a characteristic function is smaller than $1$ at the second line, as well as (\ref{Eq:BoundCaract2}) and (\ref{Eq:BoundD3l}) at the third line, we get
\begin{align*}
\Big| \Phi_N(\vec{t}) \Big| &= \Big| \nu_N\Big[\exp\Big(i\sum_{j=1}^{2N} X(j)\sum_{p=1}^k \tun_{\{j\leq j_p\}}\frac{t_p}{(2N)^\frac{\alpha}{2}} \Big)\Big] \Big|\\
&\leq\prod_{j=1}^{2N} \Big|\phi_{h^N_j}\bigg(\frac{\sum_{p=1}^k \tun_{\{j\le j_p\}} t_p}{(2N)^\frac{\alpha}{2}} \bigg)\Big|\\
&\leq\prod_{j\in I_{N,a,\ell}} \Big|\phi_{h^N_j}\Big(\frac{t_\ell + \ldots + t_k}{(2N)^\frac{\alpha}{2}} \Big)\Big|\\
&\leq\exp\Big(-r \frac{(t_\ell+\ldots+t_k)^2}{(2N)^\alpha}\sum_{j\in I_{N,a,\ell}} L''(h^N_j)\Big)\;.
\end{align*}
Since $\frac{1}{(2N)^\alpha}\sum_{j\in I_{N,a,\ell}} L''(h^N_j) \rightarrow q(x_{\ell-1},x_\ell \wedge a)$ as $N\rightarrow\infty$ and since the limit is strictly positive, we have for $N$ large enough
\begin{equation*}
\frac{1}{(2N)^\alpha}\sum_{j\in I_{N,a,\ell}} L''(h^N_j) \geq \frac12 q(x_{\ell-1},x_\ell \wedge a)\;,
\end{equation*}
so that, using (\ref{Eq:BoundD3l}), we get
\begin{equation*}
J_{3,\ell} \leq \int_{D_{3,\ell}} e^{-\frac{r}{8} t_\ell^2 q(x_{\ell-1},x_\ell \wedge a)} d\vec{t} \lesssim N^\frac{(k-1)\alpha}{2} \int_{|t_\ell| > 3^{-\ell} N^{\rho\alpha}}e^{-\frac{r}{8} t_\ell^2 q(x_{\ell-1},x_\ell \wedge a)} dt_\ell \;,
\end{equation*}
which goes to $0$ as $N\rightarrow\infty$. This concludes the proof.
\end{proof}

\begin{remark}\label{Rk:Mean}
It is possible to push the expansion of the local limit theorem one step further, in the spirit of~\cite[Thm VII.12]{Petrov}. Then, a simple calculation shows the following. Let $x\in(0,1)$ if $\alpha \geq 1$, and $x\in \R$ if $\alpha < 1$. We have
\begin{equation*}
\mu_N\big[S(k)\big] - \nu_N\big[S(k)\big] = o(N^{\frac{1\wedge\alpha}{2}})\;,
\end{equation*}
uniformly over all $k\leq x(2N)$ if $\alpha \geq 1$, and all $k\leq N+x(2N)^\alpha$ if $\alpha < 1$.
\end{remark}

\begin{corollary}\label{Cor:AbsCont}
The Radon-Nikodym derivative of $\mu_N$ with respect to $\nu_N$, restricted to $\sigma(X_1,\ldots,X_N)$, is bounded uniformly over all $N\geq 1$.
\end{corollary}
\begin{proof}
Using (\ref{Eq:muNpiN}) at the first line and the independence of the $X_i$'s at the second line, we get
\begin{align*}
\mu_N\big(\underset{i=1}{\overset{N}{\cap}}\{S(i)=y_i\}\big) &= \frac{\nu_N\big(\underset{i=1}{\overset{N}{\cap}}\{S(i)=y_i\}; S(2N) = 0\big)}{\nu_N(S(2N)=0)}\\
&= \nu_N\big(\underset{i=1}{\overset{N}{\cap}}\{S(i)=y_i\}\big) \,\frac{\nu_N(S(2N)-S(N) = -y_N)}{\nu_N(S(2N)=0)}\;,
\end{align*} 
for all $y_1,\ldots,y_N \in \R$. By Lemmas \ref{Lemma:LLT} and \ref{Lemma:LLTReversed}, the fraction on the r.h.s.~is uniformly bounded over all $y_N \in \R$ and all $N\geq 1$, thus yielding the statement of the corollary.
\end{proof}

\paragraph{Step 3: Convergence of the marginals under $\mu_N$.} From (\ref{Eq:muNpiN}), we deduce that
\begin{equation}\label{Eq:AbsContMarg}
\mu_N\big(u^N(\vec{x}^N)=\vec{y}\,\big) = \frac{\nu_N\big(u^N(\vec{x}^N)=\vec{y};\, u^N(x_f)=0\big)}{\nu_N\big(u^N(x_f)=0\big)}\;.
\end{equation}
By Lemma \ref{Lemma:LLT} and Equation (\ref{Eq:Densities}), we have
\begin{equation*}
\mu_N\big(u^N(\vec{x}^N)=\vec{y}\,\big) = 2^k(2N)^{-\frac{k}{2}(\alpha\wedge 1)}g^{\vec{x}}_\alpha(\vec{y}) \big(1+o(1)\big)\;,
\end{equation*}
uniformly over all $N\geq 1$, and all $\vec{y}$ lying in the intersection of $D^{k,N}_\alpha$ with a compact domain of $\R^k$. Thus, we deduce that for all $\vec{v} < \vec{w} \in \R^{k}$, we have
\begin{align*}
\mu_N\big(u^N(\vec{x}^N) \in [\vec{v},\vec{w}]\big) &= \sum_{\vec{y} \in [\vec{v},\vec{w}] \cap D^{k,N}_\alpha} \mu_N\big(u^N(\vec{x}^N)=\vec{y}\,\big)\\
&= \sum_{\vec{y} \in [\vec{v},\vec{w}] \cap D^{k,N}_\alpha} \Big(\frac{2}{(2N)^{\frac{\alpha\wedge 1}{2}}}\Big)^k g^{\vec{x}}_\alpha(\vec{y}) \big(1+o(1)\big)\\
&\longrightarrow \int_{\vec{y} \in [\vec{v},\vec{w}]} g^{\vec{x}}_\alpha(\vec{y}) d\vec{y}\;,
\end{align*}
as $N\rightarrow\infty$. Since $| u^N(\vec{x}) - u^N(\vec{x}^N)| = \cO(N^{-\frac{\alpha\wedge 1}{2}})$ uniformly over all $\vec{x}$, we deduce that the finite dimensional marginals of $u^N$ under $\mu_N$ converge to those of $B_\alpha$.

\paragraph{Step 4: Tightness of the sequence $\mu_N$.} For convenience, we restrict to the case $\alpha \in (0,1)$; but the case $\alpha \ge 1$ is actually simpler. We start by showing tightness of $(u^N(x),x\in \R)$ under $\nu_N$. More precisely, we are going to show that there exists $\beta >0$ such that
$$ \sup_{N\geq 1} \nu_N\bigg[|u^N(0)|+\sup_{x\ne y \in [-A,+A]} \frac{|u^N(x)-u^N(y)|}{|x-y|^\beta}\bigg] < \infty\;,$$
for all $A>0$, which ensures tightness in $\cC(\R,\R)$.\\
We have for any $\lambda \in\R$
\begin{align*}
\log \nu_N\big[e^{\lambda u^N(0)}\big] &= \sum_{k=1}^N \log\nu_N\Big[\exp\Big(\frac{\lambda}{(2N)^{\frac{\alpha}{2}}}\big(X_k - L'(h^N_k)\big)\Big)\Big]\\
&= \sum_{k=1}^N \bigg(L\Big(h^N_k+\frac{\lambda}{(2N)^{\frac{\alpha}{2}}}\Big) -L\big(h^N_k\big)-\frac{\lambda}{(2N)^{\frac{\alpha}{2}}}L'\big(h^N_k\big)\bigg)\\
&=\frac{\lambda^2}{2} q_\alpha(-\infty,0) + \cO(N^{-\alpha/2})\;,
\end{align*}
uniformly over all $N\geq 1$, which ensures that all the moments of $u^N(0)$ are uniformly bounded in $N\geq 1$.\\
Regarding the H\"older semi-norm, a direct computation shows that for all $A,\delta > 0$
\begin{equation*}
\log \nu_N \Big[\exp\Big(\frac{u^N(y) - u^N(x)}{|y-x|^\delta}\Big)\Big] \leq \|L''\|_\infty |x-y|^ {1-2\delta}\;,
\end{equation*}
uniformly over all $x,y$ of the form $(k-N)/(2N)^\alpha$ with
$$ k\in\{N-\lfloor A(2N)^\alpha\rfloor,\ldots,N+\lfloor A(2N)^\alpha\rfloor\}\;.$$
Taking $\delta \in (0,1/2)$, this yields a finite bound uniformly over all $N\geq 1$ and all such discrete $x,y$. Using classical interpolation arguments, we deduce that this bound is still finite for non-discrete $x,y$ lying in $[-A,A]$. Henceforth, the Kolmogorov Continuity Theorem ensures that for any $\beta \in (0,1/2)$ and any $p\geq 1$ we have:
\begin{equation*}
\sup_{N\geq 1} \nu_N\bigg[\sup_{x\ne y \in J} \frac{|u^N(x)-u^N(y)|^p}{|x-y|^{p\beta}}\bigg] < \infty\;,
\end{equation*}
where $J=[-A,+A]$. Notice that we did not introduce a different notation for the modification built from the Kolmogorov Continuity Theorem, since it necessarily coincides almost surely with the continuous process $u^N$. Since $J$ is arbitrary, tightness in $\cC(\R,\R)$ of $u^N$ under $\nu_N$ follows.\\
By Corollary \ref{Cor:AbsCont}, the law of $(u^N(x),x\in [-\frac{N}{(2N)^\alpha},0])$ under $\mu_N$ is absolutely continuous w.r.t. the law of the same process under $\nu_N$. This ensures that the sequence of the laws of $(u^N(x),x\in [-\frac{N}{(2N)^\alpha},0])$ under $\mu_N$ is tight. Since the laws of the processes $(u^N(-x),x\in [0,\frac{N}{(2N)^\alpha}])$ and $(u^N(x),x\in [0,\frac{N}{(2N)^\alpha}])$ under $\mu_N$ coincide, we deduce the tightness of the whole process.\\

This concludes the proof of Theorem \ref{Th:Static}.

\begin{proof}[Proof of Proposition \ref{Prop:Partition}]
Recall that $\pi_N$ is the uniform measure on the set of lattice paths that make $2N$ steps and start from $0$. We write
\begin{equation}\label{Eq:Partition}
Z_N = 2^{2N} \pi^N\Big[\Big(\frac{p_N}{q_N}\Big)^{\frac{1}{2}A(S)} \tun_{\{S(2N)=0\}} \Big]= 2^{2N} \nu_N(S(2N)=0)\, e^{L_S(h^N)}\;.
\end{equation}
Since $\nu_N(S(2N)=0) \rightarrow 0$ as $N\rightarrow\infty$, it suffices to estimate the exponential term. When $\alpha > 1$, we use the fact that $L(0)=L'(0)=L^{(3)}(0)=0$, $L''(0)= 1$ and $\|L^{(4)}\|_\infty < \infty$, to get
\begin{align*}
	L_{S}(h^N) &=\sum_{i=1}^{2N} L(h^N_i)=\frac{2\sigma^2}{(2N)^{2\alpha}}L''(0) \sum_{i=1}^{2N}\Big(N-i+\frac{1}{2}\Big)^2 + \cO(N^{5-4\alpha})\\
	&= \frac{\sigma^2}{6}(2N)^{3-2\alpha} + \cO(N^{(5-4\alpha)\vee (1-2\alpha)})\;,
\end{align*}
and the asserted result follows in that case. For $\alpha = 1$, the result follows from the convergence of Riemann approximations of integrals. Finally, when $\alpha < 1$, we use the simple facts that $L$ is even and that $L(x)-x+\log 2$ is integrable on $[0,\infty)$ to get
\begin{align*}
L_S(h^N) &= 2\sum_{i=1}^N L(h^N_i) = 2 \sum_{i=1}^N h^N_i - 2N \log 2 + 2 \sum_{i=1}^N \big(L(h^N_i) - h^N_i+\log 2)\\
&= \frac{\sigma}{2}(2N)^{2-\alpha} - 2N \log 2 + \cO(N^{\alpha})\;,
\end{align*}
thus concluding the proof.
\end{proof}

\section{Equilibrium fluctuations}\label{Section:Fluctuations}

The goal of this section is to establish Theorem \ref{Th:Dynamic}. Our method of proof is standard: first, we show tightness of the sequence of processes $u^N$, then we identify the limit via a martingale problem. Recall that we work under the reversible measure $\mu_N$.

\subsection{Tightness}\label{Subsec:TightFluct}
From now on, we set $J=[0,1]$ when $\alpha \geq 1$ and $J=[-A,+A]$ for an arbitrary value $A>0$ when $\alpha < 1$, along with
\begin{equation*}
e_n(x)=\begin{cases} \sqrt{2} \sin(n\pi x)\quad&\mbox{ for } \alpha \geq 1\;,\\
\frac{1}{\sqrt{A}} \sin\Big(\frac{n\pi}{2A}(x+A)\Big)\quad&\mbox{ for }\alpha < 1\;.
\end{cases}
\end{equation*}
This is an orthonormal basis of $L^2(J)$. For all $\beta > 0$, we define the associated Sobolev spaces
\begin{equation*}
H^{-\beta}(J) := \Big\{ f \in S'(J): \|f\|_{H^{-\beta}}^2 := \sum_{n\geq 1} n^{-2\beta} \langle f, e_n\rangle^2 < \infty \Big\}\;.
\end{equation*}
Recall that for $\alpha < 1$, the value $A>0$ is arbitrary. In order to prove tightness of the sequence $u^N$ in the Skorohod space $\bbD([0,\infty),\cC([0,1])$ for $\alpha> 1$ and in $\bbD([0,\infty),\cC(\R))$ for $\alpha < 1$, it suffices to show that the sequence of laws of $u^N(t=0,\cdot)$ is tight in $\cC(J)$, and that for any $T>0$ there exists $p > 0$ such that
\begin{equation}\label{Eq:TightnessCriterionFluct}
	\lim_{h\downarrow 0}\varlimsup_{N\rightarrow\infty} \E^N_{\mu_N}\bigg[ \sup_{\substack{s,t \leq T\\|t-s|\leq h}} \| u^N(t)-u^N(s)\|^p_{\cC(J)} \bigg] = 0\;,
\end{equation}
see for instance~\cite[Thm 13.2]{Billingsley}.\\
Since we start from the stationary measure, the first condition is ensured by what we proved in Step 4 of the preceding section. To check the second condition, we proceed as follows. We introduce a piecewise linear interpolation in time $\bar{u}^N$ of our original process by setting
\begin{align*}
\bar{u}^N(t,\cdot) &:= \big( t_N +1 - t(2N)^{2\alpha\wedge 2}\big)u^N\Big(\frac{t_N}{(2N)^{2\alpha\wedge 2}},\cdot\Big)\\
&\quad+ \big(t(2N)^{2\alpha\wedge 2}-t_N) u^N\Big(\frac{t_N+1}{(2N)^{2\alpha\wedge 2}},\cdot\Big)\;,
\end{align*}
where  $t_N:=\lfloor t(2N)^{2\alpha\wedge 2} \rfloor$.

\begin{lemma}\label{Lemma:HolderFluct}
For all $\beta > 1/2$  and all $p \geq 1$, we have
\begin{equation*}
\E^N_{\mu_N}\Big[\| \bar{u}^N(t)-\bar{u}^N(s) \|_{H^{-\beta}(J)}^p\Big]^{\frac{1}{p}} \lesssim \sqrt{t-s}\;,
\end{equation*}
uniformly over all $0 \leq s \leq t \leq T$ and all $N\geq 1$.
\end{lemma}

\begin{proof}
Assume that we have the bound
\begin{equation}\label{Eq:BoundInterpo}
\E^N_{\mu_N}\Big[\|u^N(t)-u^N(s) \|_{H^{-\beta}(J)}^p\Big]^{\frac{1}{p}} \lesssim \sqrt{t-s} + N^{-\frac{3}{2} (1\wedge \alpha)}\;,
\end{equation}
uniformly over all $0\leq s \leq t$ and all $N\geq 1$. Let $0 \leq s \leq t \leq T$. We distinguish two cases. If $t_N=s_N$ or $t=(s_N+1)/(2N)^{2\alpha\wedge 2}$, then $t-s \leq 1/(2N)^{2\alpha\wedge 2}$ and
\begin{equation*}
\bar{u}^N(t,\cdot)-\bar{u}^N(s,\cdot) = (t-s)(2N)^{2\alpha\wedge 2} \bigg(u^N\Big(\frac{s_N+1}{(2N)^{2\alpha\wedge 2}},\cdot\Big)-u^N\Big(\frac{s_N}{(2N)^{2\alpha\wedge 2}},\cdot\Big)\bigg)\;,
\end{equation*}
so that the asserted bound follows from (\ref{Eq:BoundInterpo}) and the fact that $(t-s)(2N)^{2\alpha\wedge 2} \le \sqrt{t-s} (2N)^{\alpha\wedge 1}$ in that case. If $t_N\geq s_N+1$, then we write
\begin{align*}
\bar{u}^N(t,\cdot)-\bar{u}^N(s,\cdot) &= u^N\Big(\frac{t_N}{(2N)^{2\alpha\wedge 2}},\cdot\Big)-u^N\Big(\frac{s_N+1}{(2N)^{2\alpha\wedge 2}},\cdot\Big) \\
&+ \bar{u}^N(t,\cdot)-\bar{u}^N\Big(\frac{t_N}{(2N)^{2\alpha\wedge 2}},\cdot\Big)\\
&+ \bar{u}^N\Big(\frac{s_N+1}{(2N)^{2\alpha\wedge 2}},\cdot\Big) - \bar{u}^N(s,\cdot)\;.
\end{align*}
The second and third increments on the r.h.s.~can be bounded using the first case above, yielding a term of order $\sqrt{t-s}$. Regarding the first increment, either $t_N=s_{N}+1$ and it vanishes, or $t_N\ge s_{N}+2$ and (\ref{Eq:BoundInterpo}) yields a bound of order
\begin{equation*}
\sqrt{\frac{t_N}{(2N)^{2\alpha\wedge 2}}-\frac{s_N+1}{(2N)^{2\alpha\wedge 2}}} + N^{-\frac{3}{2} (1\wedge \alpha)} \lesssim \sqrt{t-s} + (t-s)^{\frac{3}{4}} \lesssim \sqrt{t-s}\;,
\end{equation*}
as required. To complete the proof of the lemma, it suffices to show (\ref{Eq:BoundInterpo}).\\
For all $n\geq 1$, we let $\hat{u}(t,n):= \int_J u(t,x)e_n(x) dx$. Since $\beta > 1/2$, (\ref{Eq:BoundInterpo}) is proved as soon as we show that for all $p\geq 1$ 
\begin{equation}\label{Eq:BoundFourier}
\E^N_{\mu_N}\Big[|\hat{u}(t,n)-\hat{u}(s,n) |^p\Big]^{\frac{1}{p}} \lesssim \sqrt{t-s} + N^{-\frac{3}{2} (1\wedge \alpha)}\;,
\end{equation}
uniformly over all $N\geq 1$, all $n\geq 1$ and all $0 \leq s \leq t$.\\
Let $\cL^N$ be the generator of $u^N$. Using the reversibility of the process, we have the following identities
\begin{align*}
\hat{u}(t,n) - \hat{u}(s,n) &= \int_s^t \cL^N \hat{u}(r,n) dr + M^N_{s,t}(n)\;,\\
\hat{u}(T-(T-t),n) - \hat{u}(T-(T-s),n) &= -\int_s^t \cL^N \hat{u}(r,n) dr + \tilde{M}^N_{t,s}(n)\;,
\end{align*}
where $M^N_{s,t}(n),t\geq s$ is a martingale adapted to the natural filtration of $u^N$, and $\tilde{M}^N_{t,s},s\leq t$ is a martingale in the reversed filtration. Summing up these two identities, we deduce that it suffices to control the $p$-th moment of the martingales $M^N_{s,t}(n)$ and $\tilde{M}^N_{t,s}(n)$. Using the Burkh\"older-Davis-Gundy inequality (\ref{Eq:BDG3}), we get
\begin{equation*}
\E^N_{\mu_N}\Big[|M^N_{s,t}(n) |^p\Big]^{\frac{1}{p}} \!\lesssim \E^N_{\mu_N}\Big[\langle M^N_{s,\cdot}(n)\rangle_t^{p/2}\Big]^{\frac{1}{p}}+ \E^N_{\mu_N}\Big[\sup_{r\in(s,t]}|M^N_{s,r}(n)-M^N_{s,r-}(n)|^p\Big]^{\frac{1}{p}},
\end{equation*}
uniformly over all $N\geq 1$, all $n\geq 1$ and all $0 \leq s \leq t$. It is then a simple calculation to check that almost surely $\langle M^N_{s,\cdot}(n)\rangle_t$ is bounded by a term of order $t-s$ and $\sup_{r\in(s,t]}|M^N_{s,r}(n)-M^N_{s,r-}(n)|$ by a term of order $N^{-\frac32 (1\wedge \alpha)}$, uniformly over all $n,N\ge 1$, thus yielding (\ref{Eq:BoundFourier}). The same bound holds for the reversed martingale by symmetry, thus concluding the proof.
\end{proof}

We need an interpolation inequality to conclude the proof of the tightness.
\begin{lemma}\label{Lemma:Interpo}
Let $\eta=1/2-\epsilon$ and $\beta=1/2+\epsilon$. For $\epsilon>0$ small enough, there exist $c>0$ and $\gamma, \kappa \in (0,1)$ such that
\begin{equation}\label{Interpo}
\|f\|_{\cC^\gamma(J)} \leq c\, \|f\|_{\cC^\eta(J)}^\kappa \|f\|_{H^{-\beta}(J)}^{1-\kappa}\;,\quad \forall f \in \cC^\eta(J) \cap H^{-\beta}(J)\;.
\end{equation}
\end{lemma}
\begin{proof}
We rely on two standard interpolation results, we refer to the book of Triebel~\cite{Triebel} for the proofs. For $q\geq 1$ and $\delta \in (0,1)$, let $W^{\delta,q}(J)$ be the space of functions $f:J\rightarrow\R$ such that
\begin{equation*}
\|f\|_{W^{\delta,q}} := \| f\|_{L^q} + \Big(\int_x\int_y \frac{|f(x)-f(y)|^q}{|x-y|^{\delta q +1}}dx\,dy\Big)^{\frac{1}{q}} < \infty\;.
\end{equation*}
For $\eta,\beta > 0$ and $\kappa\in(0,1)$, we set $\delta := \kappa\eta - (1-\kappa)\beta$ as well as $q:=2/(1-\kappa)$. Then, there exists $c'>0$ such that
\begin{equation*}
\|f\|_{W^{\delta,q}} \leq c' \|f\|_{\cC^\eta}^\kappa \|f\|_{H^{-\beta}}^{1-\kappa}\;,\quad \forall f \in \cC^\eta\cap H^{-\beta}\;.
\end{equation*}
Furthermore, for any $\gamma>0$ such that $(\delta-\gamma)q > 1$ there exists $c''> 0$ such that
\begin{equation*}
\|f\|_{\cC^\gamma} \leq c'' \|f\|_{W^{\delta,q}}\;,\quad \forall f \in W^{\delta,q}\;.
\end{equation*}
Therefore, taking $\kappa\in (2/3,1)$, $\eta = 1/2-\epsilon$ and $\beta=1/2+\epsilon$ with $\epsilon$ small enough, we deduce the statement of the lemma.
\end{proof}

By the triangle inequality at the first line, we deduce that for all $p\geq 1$ and all $\eta \in (0,1/2)$
\begin{align*} \sup_{0 \leq s \leq t} \E^N_{\mu_N}\Big[\|\bar{u}^N(t)-\bar{u}^N(s)\|^p_{\cC^\eta(J)}\Big] &\le \sup_{0 \leq s \leq t} \E^N_{\mu_N}\Big[\Big(\|\bar{u}^N(t)\|_{\cC^\eta(J)}+\|\bar{u}^N(s)\|_{\cC^\eta(J)}\Big)^p\Big]\\
&\lesssim \sup_{t\ge 0} \E^N_{\mu_N}\Big[\|\bar{u}^N(t)\|_{\cC^\eta(J)}^p\Big]\;.
\end{align*}
Using the H\"older regularity of the interface under $\mu_N$ proved in Step 4 of the previous section, the stationarity of the process $u^N$ and the definition of $\bar{u}^N$, we then deduce that for all $p\geq 1$ and all $\eta \in (0,1/2)$
\begin{equation*}
\sup_{N\geq 1} \sup_{0 \leq s \leq t} \E^N_{\mu_N}\Big[\|\bar{u}^N(t)-\bar{u}^N(s)\|^p_{\cC^\eta(J)}\Big] < \infty\;.
\end{equation*}
Using Lemmas \ref{Lemma:HolderFluct} and \ref{Lemma:Interpo} together with H\"older's inequality, we deduce that there exist $\gamma,\kappa \in (0,1)$ such that for all $p\geq 1$
\begin{equation*}
\E^N_{\mu_N}\Big[\|\bar{u}^N(t)-\bar{u}^N(s)\|^p_{\cC^\gamma(J)}\Big] \lesssim (t-s)^{\frac{p(1-\kappa)}{2}}\;,
\end{equation*}
uniformly over all $N\geq 1$ and all $0 \leq s \leq t \leq T$. Applying Kolmogorov's Continuity Theorem, we deduce that for all $\nu \in \big(0,(1-\kappa)/2)$ and all $p\geq 1$, we have
\begin{equation*}
\sup_{N\geq 1}\E^N_{\mu_N}\bigg[\sup_{s\ne t \in [0,T]}\frac{\|\bar{u}^N(t)-\bar{u}^N(s)\|^p_{\cC^\gamma(J)}}{|t-s|^{\nu p}}\bigg] < \infty\;.
\end{equation*}
We deduce that condition (\ref{Eq:TightnessCriterionFluct}) is fulfilled by the process $\bar{u}^N$. The next lemma shows that $u^N$ and $\bar{u}^N$ are uniformly close on compact sets, so that (\ref{Eq:TightnessCriterionFluct}) is also fulfilled by the process $u^N$, thus concluding the proof of tightness.

\begin{lemma}\label{Lemma:uubar}
For all $p\geq 1$, $\lim_{N\rightarrow\infty} \E^N_{\mu_N}\big[\sup_{t\leq T} \|u^N(t)-\bar{u}^N(t)\|_{\cC(J)}^p\big] = 0$.
\end{lemma}

\begin{proof}
For all $k\in\{0,\ldots,2N-1\}$ and all $i\in\N$, we set
\begin{align*}
	B_{i,k} &:= \Big[\frac{i}{(2N)^{2}},\frac{i+1}{(2N)^{2}}\Big] \times \Big[\frac{k}{2N},\frac{k+1}{2N}\Big]\;,\quad&&\alpha \geq 1\;,\\
	B_{i,k} &:= \Big[\frac{i}{(2N)^{2\alpha}},\frac{i+1}{(2N)^{2\alpha}}\Big] \times \Big[\frac{k-N}{(2N)^\alpha},\frac{k+1-N}{(2N)^\alpha}\Big]\;,\quad&&\alpha < 1\;.
\end{align*}
Suppose that for all $p\geq 1$ we have
\begin{equation}\label{Eq:Approxubaru}
\E^N_{\mu_N}\Big[\sup_{(t,x)\in B_{i,k}} |u^N(t,x)-\bar{u}^N(t,x) |^p\Big] \lesssim (2N)^{-(\alpha\wedge 1)\frac{p}{2}}\;,
\end{equation}
uniformly over all $i\in\N$ and all $k\in\{0,\ldots,2N-1\}$. Then, we deduce that
\begin{equation*}
\E^N_{\mu_N}\Big[\sup_{t\leq T} \|u^N(t)-\bar{u}^N(t)\|_{\cC(J)}^p\Big] \lesssim (2N)^{(\alpha\wedge 1)\big(3-\frac{p}{2}\big)}\;,
\end{equation*}
uniformly over all $N\geq 1$. This yields the statement of the lemma for $p$ large enough, and in turn, Jensen's inequality ensures that it holds for all $p\geq 1$. Therefore, we are left with the proof of (\ref{Eq:Approxubaru}). For notational convenience, let us consider the case $\alpha \ge 1$. We have
\begin{align*}
|u^N(t,x)-\bar{u}^N(t,x) | \leq \sum_{j,\ell\in\{0,1\}}|u^N(t,k+\ell) - u^N((i+j)(2N)^{-2},k+\ell)|\;,
\end{align*}
for all $(t,x)\in B_{i,k}$, all $i\in\N$, all $k\in\{0,\ldots,2N-1\}$ and all $N\geq 1$. There are four terms in the sum. For each of them, the supremum over $(t,x)\in B_{i,k}$ of the corresponding increment~is stochastically bounded by $2/(2N)^{(\alpha\wedge 1)/2}$ times a Poisson r.v.~with mean $1$. Computing the $p$-th moment of the latter yields (\ref{Eq:Approxubaru}).
\end{proof}

\subsection{The Boltzmann-Gibbs principle}

The next result is the main ingredient that we need for the identification of the limit. We will work at the level of the particle system $\eta \in \{0,1\}^{2N}$. Under the measure $\nu_N$ defined in (\ref{Eq:piNnuN}), the $\eta(k)$'s are independent Bernoulli r.v.~with parameter $q^N_k$.\\
Let $\tau_k$ denote the shift by $k$ modulo $2N$: namely, $\tau_k \eta(j) = \eta(j+k)$ for all $j\in \{1,\ldots,2N\}$. Let $\Psi$ be a cylinder function, that is, a function $\Psi:\{0,1\}^r \rightarrow \R$ for some $r\in\bbN$. As soon as $r\leq 2N$, we can define $\Psi(\eta) = \Psi(\eta(1),\ldots,\eta(r))$. Then, we set
\begin{equation*}
V^N_\Psi(\eta):=\Psi(\eta) - \tilde{\Psi}_N - r \tilde{\Psi}_N'(\eta(1)-q^N_1)\;,\quad \tilde{\Psi}_N := \nu_N\big[ \Psi\big]\;,\quad \tilde{\Psi}_N' := \partial_{q^N_1}\nu_N\big[ \Psi\big]\;.
\end{equation*}
as well as its shift by $k$
\begin{equation*}
\tau_k V^N_\Psi(\eta) :=\Psi(\tau_k\eta) - \tau_k\tilde{\Psi}_N - r (\tau_k\tilde{\Psi}_N')(\eta(k+1)-q^N_{k+1})\;,
\end{equation*}
where $\tau_k\tilde{\Psi}_N := \nu_N\big[ \Psi(\tau_k\cdot)\big]$ and $\tau_k\tilde{\Psi}_N' := \partial_{q^N_{k+1}}\nu_N\big[ \Psi(\tau_k\cdot)\big]$. Notice that $V^N_\Psi$ and all its shifts have zero expectation under $\nu_N$.
\begin{proposition}[Boltzmann-Gibbs principle]\label{Prop:BG}
Let $\varphi$ be a continuous function on $[0,1]$ if $\alpha \geq 1$, a continuous and compactly supported function on $\R$ if $\alpha < 1$. Then for every $t>0$ we have
\begin{equation}\label{Eq:BoltzmannGibbs}
\begin{split}
\lim_{N\rightarrow\infty} \E^N_{\mu_N} \bigg[\Big(\int_0^t \frac1{\sqrt{2N}} \sum_{k=1}^{2N} \tau_k V^N_\Psi(\eta_s) \varphi\Big(\frac{k}{2N}\Big)\Big)^2\bigg] = 0\;,\quad \alpha \geq 1\;,\\
\lim_{N\rightarrow\infty} \E^N_{\mu_N} \bigg[\Big(\int_0^t \frac1{(2N)^{\frac{\alpha}{2}}} \sum_{k=1}^{2N} \tau_k V^N_\Psi(\eta_s) \varphi\Big(\frac{k-N}{(2N)^\alpha}\Big)\Big)^2\bigg] = 0\;,\quad \alpha < 1\;.
\end{split}
\end{equation}
\end{proposition}
\noindent Note the specific scaling of the test function for $\alpha < 1$: this is because in the scaling limit we only ever look at the interface within a window of size $N^\alpha$ around site $N$. When $\alpha\ge 1$, we consider the whole interface in the scaling limit.\\

This type of result is classical in the literature on fluctuations of particle systems. However, our setting presents some specificities. First, our stationary measure is not a product measure, but it can be obtained by conditioning the product measure $\nu_N$ on the hyperplane of all configurations with $N$ particles, as we did in Section \ref{Section:InvMeas}. Second, $\nu_N$ is the product of independent but non-identically distributed Bernoulli measures; however the means of these Bernoulli measures vary ``smoothly" in space. Given these differences with the usual setting, we provide the details of the proof, following the structure of the classical proof provided in~\cite[Thm 11.1.1]{KipLan}. We restrict ourselves to proving the case $\alpha < 1$, as the case $\alpha \geq 1$ is actually simpler.
\begin{proof}
Let $A>0$ be such that supp $\varphi \subset [-A,A]$. We adopt the notation $\varphi(k)$ for $\varphi((k-N)/(2N)^\alpha)$ for simplicity. An important argument in the proof will be the uniform absolute continuity of $\mu_N$ w.r.t.~$\nu_N$, when the measures are restricted to the filtration generated by $\eta(1),\ldots,\eta(N+A(2N)^\alpha)$, which follows as an immediate adaptation of Corollary \ref{Cor:AbsCont}. To prove the proposition, we let $K$ be an integer and we decompose $\{N-A(2N)^\alpha,\ldots, N+A(2N)^\alpha\}$ into $M$ disjoint, consecutive boxes of size $2K+1$ (except the last box that may be of smaller size), that we denote by $B_i$, $i=1\ldots M$. Necessarily $M$ is of order $N^\alpha/K$. For each box $B_i$, we define its interior $B_i^\circ$ as the subset of all points in $B_i$ which are at distance at least $r+1$ from the complement of $B_i$. This being given, we denote by $B^c= \cup_i (B_i \backslash B_i^\circ)$. We also let $k_i$ be an arbitrary point in $B_i$, for each $i$. Then, we write
\begin{equation}\label{Eq:DecompoBG}\begin{split}
\frac1{(2N)^{\frac{\alpha}{2}}}\sum_{k=1}^{2N} \tau_k V^N_\Psi(\eta) \varphi(k) &= \frac1{(2N)^{\frac{\alpha}{2}}}\sum_{k\in B^c} \tau_k V^N_\Psi(\eta)\varphi(k)\\
&+ \frac1{(2N)^{\frac{\alpha}{2}}}\sum_{i=1}^M \sum_{k\in B_i^\circ} \tau_k V^N_\Psi(\eta)\big(\varphi(k)-\varphi(k_i)\big)\\
&+ \frac1{(2N)^{\frac{\alpha}{2}}}\sum_{i=1}^M \sum_{k\in B_i^\circ} \tau_k V^N_\Psi(\eta)\varphi(k_i)\;.
\end{split}\end{equation}
The contribution to (\ref{Eq:BoltzmannGibbs}) of the first term on the right gives (using Jensen's inequality on the time integral, the stationarity of $\mu_N$, the absolute continuity of $\mu_N$ w.r.t.~$\nu_N$):
\begin{align*}
&\E^N_{\mu_N} \bigg[\Big(\int_0^t \frac1{(2N)^{\frac{\alpha}{2}}} \sum_{k\in B^c} \tau_k V^N_\Psi(\eta_s)\varphi(k) ds\Big)^2\bigg]\\
&\lesssim t^2 \mu_N\bigg[\Big(\frac1{(2N)^{\frac{\alpha}{2}}} \sum_{k\in B^c} \tau_k V^N_\Psi(\eta)\varphi(k)\Big)^2\bigg]\\
&\lesssim t^2 \nu_N\bigg[\Big(\frac1{(2N)^{\frac{\alpha}{2}}} \sum_{k\in B^c} \tau_k V^N_\Psi(\eta)\varphi(k)\Big)^2\bigg]\;,
\end{align*}
Using the independence of the $\eta(i)$'s under $\nu_N$ and the fact that the $\nu_N$-expectation of $V^N_\Psi$ is zero, we get
\begin{align*}
&\nu_N\bigg[\Big(\frac1{(2N)^{\frac{\alpha}{2}}} \sum_{k\in B^c} \tau_k V^N_\Psi(\eta)\varphi(k)\Big)^2\bigg]\\
&\lesssim \frac1{(2N)^{\alpha}} \sum_{k\in B^c} \sum_{\ell \in B^c}\nu_N\Big[ \tau_k V^N_\Psi(\eta)\tau_\ell V^N_\Psi(\eta)\varphi(k)\varphi(\ell)\Big]\\
&\lesssim \frac1{(2N)^{\alpha}} \sum_{k\in B^c} \sum_{\ell \in B^c:|\ell-k|\le r}\nu_N\Big[ \tau_k V^N_\Psi(\eta)\tau_\ell V^N_\Psi(\eta)\varphi(k)\varphi(\ell)\Big]\\
&\lesssim r^2 \frac{M}{(2N)^{\alpha}}\;,
\end{align*}
so that it vanishes when $N\rightarrow\infty$ and then $K\rightarrow\infty$. Similarly, the contribution to (\ref{Eq:BoltzmannGibbs}) of the second term on the right of \eqref{Eq:DecompoBG} vanishes $N\rightarrow\infty$ and then $K\rightarrow\infty$. Let us deal with the third term, which is more delicate. For each $i$, we set $\xi_i=(\eta(k),k\in B_i)$ and we let $L^N_{B_i}$ be the generator of our process restricted to $B_i$ and \textit{not} sped up by $(2N)^{2\alpha}$:
\begin{align*}
L^N_{B_i} f (\xi_i) = \sum_{k,k+1 \in B_i} \big(f(\xi_i^{k,k+1}) - f(\xi)\big)&\Big(p_N (1-\xi_i(k))\xi_i(k+1)\\
&+ (1-p_N) \xi_i(k)(1-\xi_i(k+1))\Big)\;.
\end{align*}
Following the calculations made at Equation (1.2) and below, in the proof of~\cite[Thm 11.1.1]{KipLan}, we deduce that
\begin{align*}
\lim_{K\rightarrow\infty} \inf_f \lim_{N\rightarrow\infty} \E^N_{\mu_N} \bigg[ \Big(\int_0^t (2N)^{-\frac{\alpha}{2}} \sum_{i=1}^M \varphi(k_i) L_{B_i}^N f(\xi_i(s)) ds\Big)^2 \bigg] = 0\;,
\end{align*}
so that it suffices to show that
\begin{align*}
\lim_{K\rightarrow\infty} \inf_f \lim_{N\rightarrow\infty} \E^N_{\mu_N} &\bigg[ \Big(\int_0^t (2N)^{-\frac{\alpha}{2}} \sum_{i=1}^M \varphi(k_i)\\
&\qquad\qquad\times\Big(\sum_{k\in B_i^\circ} \tau_k V^N_\Psi(\eta_s)-L_{B_i}^N f(\xi_i(s)) \Big) ds\Big)^2 \bigg] = 0\;,
\end{align*}
where the infimum is taken over all $f:\{0,1\}^{2K+1}\rightarrow\R$. Using Jensen's inequality on the time integral, the stationarity of our dynamics w.r.t.~$\mu_N$ and then the absolute continuity property recalled above, we bound the expectation in the last expression by a term of order
\begin{equation}\label{Eq:LastTerm}
t^2 (2N)^{-\alpha} \nu_N \Big[\Big(\sum_{i=1}^M \varphi(k_i) \Big(\sum_{k\in B_i^\circ} \tau_k V^N_\Psi(\xi_i)-L_{B_i}^N f(\xi_i)\Big) \Big)^2 \Big]\;,
\end{equation}
uniformly over all $N\geq 1$, all $K\geq 1$ and all $t\geq 0$. Recall that the $\nu_N$-expectation of $V^N_\Psi$ is zero, and observe that $\nu_N$ is reversible for our dynamics. Hence the $\nu_N$-expectation of
\begin{equation*}
\sum_{k\in B_i^\circ} \tau_k V^N_\Psi(\xi_i)-L_{B_i}^N f(\xi_i)
\end{equation*}
is also zero. Moreover, $\xi_i$ and $\xi_j$ being independent under $\nu_N$ as soon as $i\ne j$, we deduce that the expression in \eqref{Eq:LastTerm} can be rewritten as
\begin{equation*}
t^2 (2N)^{-\alpha} \sum_{i=1}^M \varphi(k_i)^2 F^N_K(i) \lesssim \frac{t^2\|\varphi\|^2}{K} \frac{1}{M} \sum_{i=1}^M F^N_K(i)\;,
\end{equation*}
where
\begin{equation*}
F^N_K(i) = \nu_N \Big[\Big(\sum_{k\in B_i^\circ} \tau_k V^N_\Psi(\xi_i)-L_{B_i}^N f(\xi_i)\Big)^2 \Big]\;.
\end{equation*}
The main difference with the classical proof presented in~\cite[Thm 11.1.1]{KipLan} lies in the following argument. Let $K$ and $f$ be as above. For every $x\in (-A,A)$ let $j=j(N,x)\in \{1,\ldots,M\}$ such that
\begin{equation*}
|k_j - N - x(2N)^\alpha | = \min_{i\in \{1,\ldots,M\}}(|k_i - N - x(2N)^\alpha |)\;.
\end{equation*}
Recall the definition of $q^N_k$ given below (\ref{Eq:piNnuN}). As $N\rightarrow \infty$, $q^N_{k_j}$ converges to $q(x):=(1+L'(-2\sigma x))/2$ and
\begin{equation*}
F^N_K(j(N,x)) \rightarrow F_K(x) :=\nu^{q(x)}_K\Big[ \Big(\sum_{k=r}^{2K-r+1} \tau_k V_\Psi^{q(x)}(\xi)-L^{\mbox{\tiny sym}}_K f(\xi)\Big)^2\Big]\;,
\end{equation*}
where $\nu_K^{q}$ is the product of $2K+1$ Bernoulli measures with parameter $q$, $V_\Psi^q$ is defined by
\begin{equation*}
V_\Psi^q(\xi) = \Psi(\xi) - \nu^q_K\big[\xi\big] - \partial_q (\nu^q_K\big[\xi\big]) (\xi(1)-q)\;,
\end{equation*}
and $L^{\mbox{\tiny sym}}_K$ is the generator of the simple exclusion process on $\{0,1\}^{2K+1}$, that is
\begin{equation*}
L^{\mbox{\tiny sym}}_K f(\xi) = \frac12 \sum_{k=1}^{2K} \big(f(\xi^{k,k+1})-f(\xi)\big)\;.
\end{equation*}
Since $F^N_K(i)$ is bounded uniformly over all $i$ and all $N\geq 1$, we deduce that
\begin{equation*}
\frac{1}{M} \sum_{i=1}^M F^N_K(i) \rightarrow \frac1{2A}\int_{-A}^{A} F_K(x) dx\;,
\end{equation*}
as $N\rightarrow\infty$. Let $\cQ=\{q(x), x\in[-A,A]\}$, and observe that it is a compact subset of $(0,1)$. Putting everything together, we deduce that
\begin{align*}
{}&\lim_{N\rightarrow\infty} \E^N_{\mu_N} \bigg[ \Big(\int_0^t (2N)^{-\frac{\alpha}{2}} \sum_{i=1}^M \varphi(k_i)\Big(\sum_{k\in B_i^\circ} \tau_k V^N_\Psi(\eta_s)-L_{B_i}^N f(\xi_i(s)) \Big) ds\Big)^2 \bigg]\\
&\lesssim \frac1{K} \sup_{q\in \cQ} \nu^q_K\Big[ \Big(\sum_{k=r}^{2K-r+1} \tau_k V_\Psi(\xi)-L^{\mbox{\tiny sym}}_K f(\xi)\Big)^2\Big]\;,
\end{align*}
uniformly over all $f$ and all $K$ as above. The supremum on the right is achieved for some $q_0$ by continuity and compactness. Then, we can directly apply the arguments below Equation (1.3) in the proof of~\cite[Thm 11.1.1]{KipLan}, which prove that the infimum over $f$ of the latter expression vanishes as $K\rightarrow\infty$, thus concluding the proof of the Boltzmann-Gibbs principle.
\end{proof}

\subsection{Identification of the limit}

We treat in details the convergence of the processes $u^N$ when $\alpha < 1$, the arguments for $\alpha \geq 1$ are essentially the same.  Let us introduce a few notations first. We write $\langle f,g\rangle$ for the usual $L^2(\R,dx)$ product as well as
\begin{equation*}
\langle f ,g \rangle_N := \frac{1}{(2N)^\alpha} \sum_{k=1}^{2N} f\Big(\frac{k-N}{(2N)^\alpha}\Big)\,g\Big(\frac{k-N}{(2N)^\alpha}\Big)\;,
\end{equation*}
for the discrete $L^2$ product, and
\begin{align*}
\nabla f(x) := f(x+(2N)^{-{\alpha}}) - f(x)\;, \quad \Delta f (x) := \nabla f(x) - \nabla f(x-(2N)^{-\alpha})\;,
\end{align*}
for the discrete gradient and Laplacian. Let us state a classical result of the theory of stochastic PDEs.

\begin{proposition}[Martingale problem]\label{Def:MgalePb}
	Let $(u(t,x),x\in\R,t\geq 0)$ be a continuous process such that $\E[\|u(0,\cdot)\|_\infty]<\infty$ and for all $\varphi\in\cC^\infty_c(\R)$, the processes $M(\varphi)$ and $L(\varphi)$ are continuous martingales where
	\begin{align*}
	M_t(\varphi) &= \langle u(t),\varphi\rangle - \langle u(0),\varphi \rangle - \frac{1}{2} \int_0^t \big\langle u(s), \partial^2_x \varphi + 4\sigma \partial_x\big(\varphi \partial_x \Sigma_\alpha\big) \big\rangle ds\;,\\
	L_t(\varphi) &= M_t(\varphi)^2 - t \langle \varphi, \varphi \big(1-(\partial_x \Sigma_\alpha)^2\big) \rangle\;.
	\end{align*}
	Then, $u$ solves (\ref{SHE3}) started from the initial profile $u(0,\cdot)$.
\end{proposition}

\begin{proof}[Proof of Theorem \ref{Th:Dynamic}]
We treat in details the case $\alpha < 1$. By Subsection \ref{Subsec:TightFluct}, we already know that $(u^N)_{N\geq 1}$ is a tight sequence in $\bbD([0,\infty),\cC(\R))$. Since the sizes of the jumps are vanishing with $N$, any limit point lies in $\bbC([0,\infty),\cC(\R))$. Let us consider an arbitrary convergent subsequence $(u^{N_i})_{i\geq 1}$ and let $u$ be its limit. We only need to check that $u$ satisfies the Martingale problem of Proposition \ref{Def:MgalePb}. Our starting point is the stochastic differential equations solved by the discrete process. Namely, for all $\varphi\in\cC^\infty_c(\R)$, we set
\begin{align*}
M^N_t(\varphi) :=& \langle u^N(t),\varphi\rangle_N - \langle u^N(0),\varphi \rangle_N - \frac{1}{2} \int_0^t \big\langle u^N(s), (2N)^{2\alpha}\Delta \varphi \big\rangle_N ds\\
&-(2N)^\frac{3\alpha}{2} \int_0^t \Big\langle \frac{1}{2}\Delta \Sigma^N_\alpha(\cdot)+ (2p_N-1)\tun_{\{\Delta S(s,\cdot)\ne 0\}}, \varphi \Big\rangle_N ds\;,
\end{align*}
where we have abbreviated $\Delta S(s(2N)^{2\alpha},N + \cdot(2N)^\alpha)$ into $\Delta S(s,\cdot)$ for simplicity. Then, $M^N(\varphi)$ is a martingale, as well as $L^N_t(\varphi):= M^N_t(\varphi)^2 - \langle M^N_\cdot(\varphi)\rangle_t$ where
\begin{align*}
\langle M^N_\cdot(\varphi)\rangle_t &= \int_0^t 4 \big\langle  p_N \tun_{\{\Delta S(s,\cdot) > 0\}} + (1-p_N) \tun_{\{\Delta S(s,\cdot) < 0\}} ,\varphi^2\big\rangle_N ds\\
&= \int_0^t 2\big\langle \tun_{\{\Delta S_{\cdot}(s)\ne 0\}} ,\varphi^2 \big\rangle_N ds + \cO(N^{-\alpha})\;.
\end{align*}
Notice the similarity between $M^N(\varphi), L^N(\varphi)$ and $M(\varphi), L(\varphi)$. In order to pass to the limit along the subsequence $N_i$, we need to deal with the indicators in the expressions above. To that end, we set $\Psi(\eta)=\eta(1)(1-\eta(2)) + (1-\eta(1))\eta(2)$, which is nothing else than the indicator of the event $\{\Delta S(1)\ne 0\}$, and we aim at applying the Boltzmann-Gibbs principe of Proposition \ref{Prop:BG}. A simple calculation yields
\begin{equation*}
\tilde{\Psi}_N = q^N_1(1-q^N_2) + (1-q^N_1)q^N_2\;,\;\; \tilde{\Psi}_N' = 1-2q^N_2\;,\;\; 2p_N-1 = \frac{2\sigma}{(2N)^\alpha}+\cO(N^{-2\alpha})\;.
\end{equation*}
By Proposition \ref{Prop:BG}, the error made upon replacing the indicators in $M^N$ and $L^N$ by $\tau_\cdot \tilde{\Psi}_N + 2\tau_\cdot\tilde{\Psi}_N'(\eta_s(\cdot+1)-q^N_{\cdot+1})$ vanishes in probability as $N\rightarrow\infty$. We are left with computing
\begin{align}
{}&(2N)^\frac{\alpha}{2} \sum_{k=1}^{2N} \Big(\frac{1}{2}\Delta \Sigma^N_\alpha\Big(\frac{k-N}{(2N)^\alpha}\Big)+ (2p_N-1)\tau_k \tilde{\Psi}_N\\
&\qquad\qquad+ 2(2p_N-1)\tau_k\tilde{\Psi}_N'(\eta_s(k+1)-q^N_{k+1}) \Big)\varphi\Big(\frac{k-N}{(2N)^\alpha}\Big)\;,
\end{align}
as well as
\begin{equation*}
\frac1{(2N)^\alpha}\sum_{k=1}^{2N} \Big(\tau_k \tilde{\Psi}_N + 2\tau_k\tilde{\Psi}_N'(\eta_s(k+1)-q^N_{k+1}) \Big)\varphi^2\Big(\frac{k-N}{(2N)^\alpha}\Big)\;.
\end{equation*}
We have the identities $L'(h^N_k) = 2q^N_k-1$ and
\begin{equation*}
\nabla u^N\Big(s,\frac{k-N}{(2N)^\alpha}\Big) = \frac{2(\eta_s(k+1)-q^N_{k+1})}{(2N)^{\alpha/2}}\;,\;\;\tau_k\tilde{\Psi}_N = \frac12 \big(1-L'(h^N_k)^2\big) + \cO(N^{-\alpha})\;,
\end{equation*}
uniformly over all $k$. By a simple integration by parts, we get
\begin{align*}
{}&2(2p_N-1) (2N)^{\frac{\alpha}{2}}\sum_{k=1}^{2N} \tau_k\tilde{\Psi}_N' (\eta_s(k+1)-q^N_{k+1})\varphi\Big(\frac{k-N}{(2N)^\alpha}\Big)\\
=& \frac{2\sigma}{(2N)^\alpha} \sum_{k=1}^{2N} u^N\Big(s,\frac{k-N}{(2N)^\alpha}\Big) (2N)^\alpha\nabla\Big(L'(h^N_k) \varphi\Big(\frac{k-N}{(2N)^\alpha}\Big) \Big) + \cO(N^{-\alpha})\;.
\end{align*}
Furthermore, the identity $L'' = 1-(L')^2$ together with the expression of $\Sigma^N_\alpha$ yield
\begin{equation*}
\frac{1}{2}\Delta \Sigma^N_\alpha\Big(\frac{k-N}{(2N)^\alpha}\Big) + (2p_N-1)\tau_k\tilde{\Psi}^N = \cO(N^{-2\alpha})\;,
\end{equation*}
uniformly over all $k$. Putting everything together, we deduce that $M^N_t(\varphi)$ and $L^N_t(\varphi)$ converge in probability to $M_t(\varphi)$ and $L_t(\varphi)$ along the convergent subsequence $u^{N_i}$. This completes the proof of Theorem \ref{Th:Dynamic}.
\end{proof}

\section{Hydrodynamic limit}\label{Section:CVEq}

\subsection{The replacement lemma}\label{Subsection:Replacement}

The goal of this subsection is to establish the so-called \textit{replacement lemma} that allows to replace averages of non-linear functionals applied to the particle system by non-linear functionals applied to averages of the particle system. This lemma will be needed in the proof of the hydrodynamic limit.\\

In this subsection, we work at the level of the particle system $\eta_t,t\geq 0$ where
\begin{equation*}
\eta_t(k) = \frac{1+X(t(2N)^{(1+\alpha)\wedge 2},k)}{2}\;.
\end{equation*}
Even though the model was defined on systems of $N$ particles, the dynamics still makes sense when the total number of particles is any integer between $0$ and $2N$.\\
We need to introduce some notations. Let $r\geq 1$ be an integer and $\Phi:\{0,1\}^r \rightarrow\R$. For all $\eta \in \{0,1\}^{2N}$ and as soon as $2N \geq r$, we define
\begin{equation*}
\Phi(\eta) := \Phi\big(\eta(1),\ldots,\eta(r)\big)\;.
\end{equation*}
We also introduce the expectation of $\Phi$ under a product of Bernoulli measures with parameter $a \in [0,1]$:
\begin{equation}\label{Eq:tildePhi}
\tilde{\Phi}(a) := \sum_{\eta \in \{0,1\}^r} \Phi(\eta) a^{\#\{i:\eta(i)=1\}} (1-a)^{\#\{i:\eta(i)=0\}}\;.
\end{equation}
We let $T_\ell(i):=\{i-\ell,i-\ell+1,\ldots,i+\ell\}$ be the box of size $2\ell+1$ around site $i$, and for any sequence $a(k),k\in\Z$, we define its average over $T_\ell(i)$ as follows:
\begin{equation*}
\ccM_{T_\ell(i)} a := \frac1{2\ell+1}\sum_{k=i-\ell}^{i+\ell} a(k)\;.
\end{equation*}
For every $k\in\Z$, we define the shift operator $\tau_k$ as follows
$$ \tau_k \eta (\ell) := \eta(\ell+k)\;,\quad \ell \in \{1,\ldots,2N\}\;,$$
where $\ell+k$ is taken modulo $2N$. We consider the sequence $\Phi(\eta)(k):= \Phi(\tau_k \eta)$ and the associated averages $\ccM_{T_\ell(i)} \Phi(\eta)$. In the sequel, we will need a ``replacement lemma" that bounds the following quantity
\begin{equation*}
V_{\ell}(\eta) = \Big| \ccM_{T_\ell(0)} \Phi(\eta) - \tilde{\Phi}\Big(\ccM_{T_\ell(0)} \eta\Big)\Big|\;.
\end{equation*}
Informally, if this quantity is small, this says that one can replace the average on a large box of some function $\Phi$ evaluated at the particle system $\eta$, by the expectation of $\Phi$ under a product of Bernoulli measures whose parameter equals the density of $\eta$ on this large box.\\

The replacement lemma works for all initial conditions when $\alpha \geq 1$. On the other hand, when $\alpha < 1$, we make the following assumption.
\begin{assumption}\label{Assumption:IC}
For all $N\geq 1$, the initial condition $\iota_N$ is a product measure on $\{0,1\}^{2N}$ of the form $\otimes_{k=1}^{2N} \mbox{Be}(f(k/2N))$, where $f:[0,1]\rightarrow[0,1]$ is assumed to be piecewise constant and does not depend on $N$.
\end{assumption}
We can probably relax this assumption, but it is sufficient for our purpose.

\begin{theorem}[Replacement lemma]\label{Th:Replacement}
Let $\alpha \in(0,\infty)$ and let $\iota_N$ be a measure on $\{0,1\}^{2N}$. For $\alpha \in (0,1)$, we suppose that Assumption \ref{Assumption:IC} is fulfilled. Then, for every $\delta > 0$, we have
\begin{equation}\label{Eq:Replacement}
\varlimsup_{\epsilon\downarrow 0}\varlimsup_{N\rightarrow\infty} \bbP^N_{\iota_N} \Big(\int_0^t \frac{1}{N}\sum_{k=1}^{2N} V_{\epsilon N}(\tau_k\eta_s) ds \geq \delta \Big) = 0\;.
\end{equation}
\end{theorem}
The proof of this theorem relies on the classical one-block and two-blocks estimates. First, let us introduce the Dirichlet form associated to our dynamics:
\begin{equation*}
D_N(f) = -\sum_{\eta} \sqrt{f}(\eta) \cL^N \sqrt{f}(\eta) \,\nu_N(\eta)\;,
\end{equation*}
where $f:\{0,1\}^{2N}\rightarrow\R_+$ and $\cL^N$ is the generator of our sped up process, that is
\begin{align*}
\cL^N g(\eta) = (2N)^{(1+\alpha)\wedge 2}\sum_{k=1}^{2N-1}\big(&g(\eta^{k,k+1})-g(\eta)\big) \big(p_N\, \eta(k+1)(1-\eta(k))\\
&+ (1-p_N)\,\eta(k)(1-\eta(k+1))\big)\;,
\end{align*}
where $\eta^{k,k+1}$ is obtained from $\eta$ by permuting the values at sites $k$ and $k+1$.\\
Notice that the reference measure in the Dirichlet form is taken to be the reversible measure $\nu_N$, which was defined in Section \ref{Section:InvMeas}. Recall that $\nu_N$ is supported by the whole set $\{0,1\}^{2N}$, but its restriction to the hyperplane with $N$ particles coincides with the measure $\mu_N$ up to a multiplicative constant.\\
In the statements of the lemmas below, the function $f$ will always be non-negative and such that $\nu_N[f]=1$.

\begin{lemma}[One-block estimate]\label{Lemma:OneBlock}
For any $\alpha>0$ and any $C>0$, we have
\begin{equation*}
\varlimsup_{\ell\rightarrow\infty}\varlimsup_{N\rightarrow\infty} \sup_{f: D_N(f) \leq C N^{(2-\alpha)\vee 1}} \frac1{N}\sum_{k=1}^{2N} \nu_N\Big[V_\ell(\tau_k\eta) f(\eta)\Big] = 0\;.
\end{equation*}
\end{lemma}
The proof is an adaptation of Kipnis, Olla and Varadhan~\cite{KOV}.
\begin{proof}
In this proof, $\xi$ always denotes an element of $\{0,1\}^{2\ell+1}$ and $\eta$ an element of $\{0,1\}^{2N}$. The identity $\eta_{|T_\ell(i)}=\xi$ will be an abusive notation for $\eta(i-\ell-1+j)=\xi(j)$ for all $j\in\{1,\ldots,2\ell+1\}$. Recall that $\Phi$ only depends on $r$ sites.\\
First, let us observe that we can restrict the sum over $k$ to $R_\ell^N := \{\ell+1,\ldots,2N-\ell\}$ since the remaining terms have a negligible contribution. Second, we have the following identity
\begin{equation}\label{Eq:DirichletOneBlock}
\frac1{N}\sum_{k\in R_\ell^N} \nu_N\Big[V_\ell(\tau_k\eta) f(\eta)\Big] = \frac1{N} \sum_{k\in R_\ell^N} \sum_{\xi} \sum_{\eta: \eta_{|T_\ell(k)} = \xi} V_\ell(\xi) f(\eta)\nu_N(\eta) + \cO(\ell^{-1})\;,
\end{equation}
uniformly over all densities $f$. Let $D^*$ denote the Dirichlet form of the symmetric simple exclusion process on $\{0,1\}^{2\ell+1}$, that is
\begin{equation*}
D^*(g) = \frac14 \sum_{\xi} 2^{-(2\ell+1)} \sum_{j=1}^{2\ell}\Big(\sqrt{g(\xi^{j,j+1})} - \sqrt{g(\xi)}\Big)^2\;.
\end{equation*}
Let us introduce
\begin{equation*}
f_\ell(\xi) := \frac{1}{\# R_\ell^N} \sum_{i\in R_\ell^N} \sum_{\eta: \eta_{|T_\ell(i)}=\xi} \nu_N(\eta) f(\eta)\;. 
\end{equation*}
Since $\nu_N[f]=1$, we immediately deduce that $\sum_\xi f_\ell(\xi)=1$. Recall the inequality
\begin{equation*}
\bigg(\sqrt{\sum_i a_i} - \sqrt{\sum_i b_i}\bigg)^2 \leq \sum_i\big(\sqrt{a_i} - \sqrt{b_i}\big)^2\;,
\end{equation*}
that holds for all summable sequences $a_i, b_i \geq 0$. Using this inequality, one gets the bound
\begin{equation*}
D^*(f_\ell) \lesssim \frac{\ell 2^{-2\ell}}{\# R_\ell^N (2N)^{2\wedge (1+\alpha)}} D_N(f)\;,
\end{equation*}
uniformly over all densities $f$, all $\ell\geq 1$ and all $N\geq 1$. Notice that $\# R_\ell^N= 2N-2\ell$. Combining the last bound with (\ref{Eq:DirichletOneBlock}), we deduce that we only need to show
\begin{equation}
\varlimsup_{\ell\rightarrow\infty} \varlimsup_{N\rightarrow\infty} \sup_{g:D^*(g) \leq \frac{C'\ell}{(2N)^{2(\alpha\wedge 1)}}} F(g)=0\;,
\end{equation}
where the supremum is taken over all $g:\{0,1\}^{2\ell+1} \rightarrow\R_+$ such that $\sum_\xi g(\xi)=1$, and where $F(g) := \sum_{\xi} V_\ell(\xi) g(\xi)$. By the lower semi-continuity of the Dirichlet form, $\{g:D^*(g) \leq \frac{C'\ell}{(2N)^{2(\alpha\wedge 1)}}\}$ is a closed subset of the compact set of all densities $g$ and is therefore a compact set. Let $g_N$ be an element for which $F$ reaches its maximum over this compact set. We claim that
\begin{equation*}
\varlimsup_{N\rightarrow\infty} F(g_N) \leq \sup_{g:D^*(g) = 0} F(g)\;.
\end{equation*}
Indeed, there exists a subsequence of $(g_N)_N$ whose image through $F$ converges to the l.h.s. One can extract another sub-subsequence that converges to some element $g_\infty$. Necessarily $g_\infty$ is a density and $D^*(g_\infty)=0$, thus yielding the claim.\\
Notice that $\{g: D^*(g)=0\}$ is the set formed by all convex combinations of the measures $\pi_{\ell,k}$, $k\in\{0,\ldots,2\ell+1\}$ where $\pi_{\ell,k}$ is the uniform measure on the subset of $\{0,1\}^{2\ell+1}$ with $k$ particles (which is irreducible for our dynamics). Henceforth, we have to show
\begin{equation*}
\varlimsup_{\ell\rightarrow\infty} \sup_{k=0,\ldots,2\ell+1} \sum_\xi V_{\ell}(\xi) \pi_{\ell,k}(\xi)=0\;.
\end{equation*}
This can be done using a Local Limit Theorem, see~\cite[Step 6 Chapter 5.4]{KipLan}.
\end{proof}

\begin{lemma}[Two-blocks estimate]\label{Lemma:TwoBlocks}
For any $\alpha \geq 1$ and any $C>0$, we have
\begin{align*}
\varlimsup_{\ell\rightarrow\infty}\varlimsup_{\epsilon\downarrow 0}\varlimsup_{N\rightarrow\infty} &\sup_{f:D_N(f)\leq C N} \frac{1}N\sum_{k=1}^{2N-1} \frac1{(2\epsilon N+1)^2}\\
&\times\nu_N\bigg[\sum_{j: |j-k| \leq \epsilon N}\sum_{j':|j'-k|\leq \epsilon N} \Big|\ccM_{T_\ell(j')}(\eta) - \ccM_{T_\ell(j)}(\eta)\Big| f(\eta)\bigg]=0\;.
\end{align*}
For $\alpha < 1$, if $\iota_N$ satisfies Assumption \ref{Assumption:IC}, we have for all $t,\delta >0$
\begin{align*}
\varlimsup_{\ell\rightarrow\infty}\varlimsup_{\epsilon\downarrow 0}\varlimsup_{N\rightarrow\infty} &\frac{1}N\sum_{k=1}^{2N-1} \frac1{(2\epsilon N+1)^2}\sum_{j: |j-k| \leq \epsilon N}\sum_{j':|j'-k|\leq \epsilon N} \\
&\times \int_0^t \P^N_{\iota_N}\Big(\big|\ccM_{T_\ell(j')}(\eta_s) - \ccM_{T_\ell(j)}(\eta_s)\big| \geq \delta\Big)ds=0\;.
\end{align*}
\end{lemma}
The proof in the case $\alpha \geq 1$ is due to Kipnis, Olla and Varadhan~\cite{KOV}.
\begin{proof}[Proof of Lemma \ref{Lemma:TwoBlocks}, $\alpha \ge 1$]
We can restrict the sum to all $k\in R^N_\epsilon$ where $R^N_\epsilon:=\{\lceil \epsilon N\rceil,\ldots,2N-\lceil \epsilon N\rceil\}$, since the contribution of the remaining terms is negligible. We can also restrict the sum over $j,j'$ to the set
\begin{equation*}
J(k):=\{(j,j'): |j-k|\leq \epsilon N, |j'-k|\leq \epsilon N, j' > j+2\ell\}\;.
\end{equation*}
Notice that $\# J(k)=\#J$ does not depend on $k$ and that it is of order $(\epsilon N)^2$ as long as $\ell$ is small compared to $\epsilon N$. Therefore, we have to control
\begin{equation*}
\frac{1}N\sum_{k\in R^N_\epsilon} \frac1{\#J}\,\nu_N\bigg[\sum_{(j,j')\in J(k)} \Big|\ccM_{T_\ell(j')}(\eta) - \ccM_{T_\ell(j)}(\eta)\Big| f(\eta)\bigg]\;.
\end{equation*}
From now on, $(\xi_1,\xi_2)$ will always denote an element of $\{0,1\}^{2\ell+1}\times\{0,1\}^{2\ell+1}$, and $\eta$ an element of $\{0,1\}^{2N}$. We set
\begin{align*}
D^1(g) &:=\frac14 \sum_{\xi_1,\xi_2}2^{-2(2\ell+1)} \sum_{n=1}^{2\ell} \big(\sqrt{g(\xi_1^{n,n+1},\xi_2)}-\sqrt{g(\xi_1,\xi_2)}\big)^2\;,\\
D^2(g) &:=\frac14 \sum_{\xi_1,\xi_2}2^{-2(2\ell+1)} \sum_{n=1}^{2\ell} \big(\sqrt{g(\xi_1,\xi_2^{n,n+1})}-\sqrt{g(\xi_1,\xi_2)}\big)^2\;,\\
D^\circ(g) &:=\frac14 \sum_{\xi_1,\xi_2}2^{-2(2\ell+1)} \sum_{n=1}^{2\ell} \big(\sqrt{g((\xi_1,\xi_2)^{\circ})}-\sqrt{g(\xi_1,\xi_2)}\big)^2\;,
\end{align*}
where $(\xi_1,\xi_2)^{\circ}$ is obtained from $(\xi_1,\xi_2)$ upon exchanging the values of $\xi_1(\ell+1)$ and $\xi_2(\ell+1)$. Let us now set
\begin{equation*}
f_\ell(\xi_1,\xi_2) := \sum_{k\in R^N_\epsilon} \frac1{\#R^N_\epsilon \#J} \sum_{(j,j')\in J(k)} \sum_{\substack{\eta:\\ \eta_{|T_\ell(j)} = \xi_1\\\eta_{|T_\ell(j')} = \xi_2}} f(\eta) \nu_N(\eta)\;.
\end{equation*}
Notice that $\sum_{\xi_1,\xi_2} f_\ell(\xi_1,\xi_2) = 1$. As in the proof of Lemma \ref{Lemma:OneBlock}, we get the bounds
\begin{equation*}
D^1(f_\ell) \lesssim \frac{\ell}{N^3} D_N(f)\;,\qquad D^2(f_\ell) \lesssim \frac{\ell}{N^3} D_N(f)\;,
\end{equation*}
uniformly over all $\ell$ and all densities $f$. On the other hand, we have the bound
\begin{align*}
D^\circ(f_\ell) \lesssim \sum_{\xi_1,\xi_2} \sum_{k\in R^N_\epsilon} \frac1{\#R^N_\epsilon \#J} \sum_{(j,j')\in J(k)} \sum_{\substack{\eta:\\ \eta_{|T_\ell(j)} = \xi_1\\\eta_{|T_\ell(j')} = \xi_2}} \nu_N(\eta) \Big(\sqrt{f(\eta^{j,j'})}-\sqrt{f(\eta)}\Big)^2\;.
\end{align*}
Observe that we have
\begin{equation*}
\eta^{j,j'} = \Big(\ldots \Big(\Big( \big(\ldots\big((\eta^{j,j+1})^{j+1,j+2}\big)\ldots\big)^{j'-1,j'}\Big)^{j'-2,j'-1}\Big)\ldots\Big)^{j,j+1}\;.
\end{equation*}
This induces a chain of configurations $\eta_0=\eta$, $\eta_1=\eta^{j,j+1}$, $\ldots$, $\eta_{2(j'-j)-1}=\eta^{j,j'}$. Then we write
\begin{equation*}
\Big(\sqrt{f(\eta^{j,j'})}-\sqrt{f(\eta)}\Big)^2 \leq (2(j'-j)-1) \sum_{m=1}^{2(j'-j)-1}\Big(\sqrt{f(\eta_m)} - \sqrt{f(\eta_{m-1})}\Big)^2\;.
\end{equation*}
A simple calculation then yields the following bound:
\begin{equation*}
D^\circ(f_\ell) \lesssim \frac{(\epsilon N)^2}{N} \frac{D_N(f)}{(2N)^2}\;.
\end{equation*}
By similar arguments as in the proof of Lemma \ref{Lemma:OneBlock}, we deduce that it suffices to show that
\begin{equation*}
\varlimsup_{\ell\rightarrow\infty} \sup_{D^1(g)=D^2(g)=D^\circ(g)=0} \sum_{\xi_1,\xi_2} \Big|\ccM_{T_\ell(0)} \xi_1 - \ccM_{T_\ell(0)} \xi_2\Big| g(\xi_1,\xi_2) = 0\;.
\end{equation*}
The set $\{g:D^1(g)=D^2(g)=D^\circ(g)=0\}$ is the set of all convex combinations of the measures $m_{\ell,k}$, where $m_{\ell,k}$ is the uniform measure on $\{0,1\}^{2\ell+1}\times\{0,1\}^{2\ell+1}$ with $k$ particles, $k=0,1,\ldots,4\ell+2$. Consequently, it suffices to show
$$ \varlimsup_{\ell\rightarrow\infty} \sup_{k=0,\ldots,4\ell+2} \sum_{\xi_1,\xi_2} \Big|\ccM_{T_\ell(0)} \xi_1 - \ccM_{T_\ell(0)} \xi_2\Big| m_{\ell,k}(\xi_1,\xi_2) = 0\;.$$
A computation shows that this quantity vanishes, thus concluding the proof.
\end{proof}
The proof in the case $\alpha < 1$ is due to Rezakhanlou, we only adapt the arguments in~\cite[Lemma 6.6]{Reza}.
\begin{proof}[Proof of Lemma \ref{Lemma:TwoBlocks}, $\alpha \in (0,1)$]
We rely on the coupling $(\eta_t^N,\zeta_t^N)$, $t\geq 0$, introduced in Subsection \ref{SectionHyperbo}: $\zeta^N$ is stationary with law $\otimes_{k=1}^{2N}\mbox{Be}(c)$, $\eta^N$ follows the same dynamics as before, $\eta_0$ has law $\iota_N$, $(\eta_0,\zeta_0)$ is ordered as follows:
$$ \iota_N[\eta_0(k)] \ge c \Leftrightarrow \eta_0(k) \ge \zeta_0(k)\;,$$
and the dynamics preserves the ordering (see Subsection \ref{SectionHyperbo} for more details).\\
Forthcoming Lemma \ref{Lemma:Coupling} shows that the number $n(t)$ of changes of sign of $k\mapsto \eta_t(k)-\zeta_t(k)$ is bounded by a constant $C>0$ for all $t\geq 0$ and all $N\geq 1$. We deduce that on the box $T_{2\epsilon N}(k)$ and for all $t\geq 0$, either $\eta_t \geq \zeta_t$ or $\eta_t \leq \zeta_t$ except for at most $2C\epsilon N$ integers $k$'s in $\{1,\ldots,2N\}$. Consequently, except for at most $2C\epsilon N$ integers $k$, we have for all $j$ such that $|j-k| \le \epsilon N$ and all $\ell \le \epsilon N$:
$$ \ccM_{T_\ell(j)}(\eta_t) \le \ccM_{T_\ell(j)}(\zeta_t)\;,\mbox{ and }\ccM_{T_{2\epsilon N}(j)}(\eta_t) \le \ccM_{T_{2\epsilon N}(k)}(\zeta_t)\;,$$
or
$$ \ccM_{T_\ell(j)}(\eta_t) \ge \ccM_{T_\ell(j)}(\zeta_t)\;,\mbox{ and }\ccM_{T_{2\epsilon N}(j)}(\eta_t) \ge \ccM_{T_{2\epsilon N}(k)}(\zeta_t)\;.$$
With a probability going to $1$ as $N\to\infty$, $\epsilon \downarrow 0$ and $\ell\to\infty$, we can replace the averages of $\zeta_t$ by the value $c$. Therefore, for any given $\delta > 0$:
\begin{align*}
\varliminf_{\ell \to\infty}\varliminf_{\epsilon\downarrow 0}\varliminf_{N\to\infty} &\frac1{N} \sum_{k=1}^{2N} \frac1{2\epsilon N+1} \sum_{|j-k| \le \epsilon N}\P_{\iota_N}^N\Big( \ccM_{T_\ell(j)}(\eta_t), \ccM_{T_{2\epsilon N}(j)}(\eta_t) \le c+\delta\;, \\
&\mbox{  or } \ccM_{T_\ell(j)}(\eta_t), \ccM_{T_{2\epsilon N}(j)}(\eta_t) \ge c-\delta\Big) = 1\;.
\end{align*}
Applying this reasoning simultaneously for all values $c\in (\delta \Z) \cap [0,1]$, we deduce that
\begin{align*}
\varliminf_{\ell \to\infty}\varliminf_{\epsilon\downarrow 0}\varliminf_{N\to\infty} &\frac1{N} \sum_{k=1}^{2N} \frac1{2\epsilon N+1} \sum_{|j-k| \le \epsilon N}\P_{\iota_N}^N\Big( |\ccM_{T_\ell(j)}(\eta_t)-\ccM_{T_{2\epsilon N}(j)}(\eta_t)| < 2 \delta\Big) = 1\;.
\end{align*}
Consequently, for any $t\ge 0$
\begin{align*}
\varlimsup_{\ell\rightarrow\infty}\varlimsup_{\epsilon\downarrow 0}\varlimsup_{N\rightarrow\infty} &\frac{1}N\sum_{k=1}^{2N-1} \frac1{(2\epsilon N+1)^2}\sum_{j: |j-k| \leq \epsilon N}\sum_{j':|j'-k|\leq \epsilon N} \\
&\times \P^N_{\iota_N}\Big(\big|\ccM_{T_\ell(j')}(\eta_t) - \ccM_{T_\ell(j)}(\eta_t)\big| \geq 4\delta\Big)=0\;.
\end{align*}
The Dominated Convergence Theorem completes the proof.
\end{proof}
We also need the following technical lemma.
\begin{lemma}\label{Lemma:Radon}
Let $G:\{0,1\}^{2N}\to\R_+$. For any $t > 0$ there exists $C>0$ such that for any initial probability measure $\iota_N$ on $\{0,1\}^{2N}$ we have
\begin{align*}
\bbP^N_{\iota_N} \Big(\int_0^t G(\eta_s) ds \geq \delta \Big) &\le \delta^{-1} t \sup_{f:D_N(f) \leq C N^{1\vee(2-\alpha)}} \sum_{\eta} \nu_N(\eta) G(\eta) f(\eta)\;,\label{Bound:Radon}
\end{align*}
where the supremum is taken over all $f:\{0,1\}^{2N}\rightarrow\R_+$ such that $\nu_N[f]=1$.
\end{lemma}
\begin{proof}
Denote by $P^N_t$ the semigroup associated to our discrete dynamics and by $f_t^N$ the density of the measure $\iota_N P^N_t$ w.r.t.~the measure $\nu_N$. In other words, $f_t^N$ is the density w.r.t.~$\nu_N$ of the law of our process starting from the initial distribution $\iota_N$. Since the dynamics is reversible w.r.t.~$\nu_N$, the operator $\cL^N$ is self-adjoint in $L^2(\nu_N)$ and we have:
$$ f_0^N = \frac{d \iota_N}{d\nu_N}\;,\quad \partial_t f_t^N = \cL^N f^N_t\;.$$
Let us collect a bound on $D_N(\frac1{t} \int_0^t f^N_s ds)$ following the lines of~\cite[Section 5.2]{KipLan}. To that end, we recall the definition of the entropy of some measure $\pi$ w.r.t.~to the measure $\nu_N$:
$$ H_N(\pi | \nu_N) := \nu_N\bigg[ \frac{d\pi}{d\nu_N} \log \frac{d\pi}{d\nu_N} \bigg] = \sum_\eta \nu_N(\eta)\frac{d\pi}{d\nu_N}(\eta) \log \frac{d\pi}{d\nu_N}(\eta)\;.$$
In the case where $\pi = \iota_N P^N_t$, we will write $H_N(f^N_t | \nu_N)$ for simplicity.\\
We have:
$$\partial_t H_N(f^N_t | \nu_N) = \sum_\eta \log(f^N_t)(\eta) \cL^N \big(f^N_t(\eta)\big) \nu_N(\eta) + \sum_\eta \cL^N \big(f^N_t(\eta)\big) \nu_N(\eta)\;.$$
Since $\nu_N$ is invariant under the dynamics the second term on the right vanishes and since $\cL^N$ is self-adjoint in $L^2(\nu_N)$ we can rewrite the first term as follows
$$= \sum_\eta \cL^N \big(\log(f^N_t)(\eta)\big) f^N_t(\eta) \nu_N(\eta)\;.$$
From the elementary inequality $\log x \le x - 1$ that holds for all $x\ge 0$, by setting $x=\sqrt{b/a}$ we deduce that $a \log b/a \le 2 \sqrt a(\sqrt b - \sqrt a)$ holds for all $a,b \ge 0$. Thus we get for every $k$
$$ f^N_t(\eta)(\log(f^N_t)(\eta^{k,k+1}) - \log(f^N_t)(\eta)) \le 2 \sqrt{f^N_t(\eta)}(\sqrt{f^N_t(\eta^{k,k+1})}-\sqrt{f^N_t(\eta)})\;,$$
and consequently
\begin{align*}
\sum_\eta \cL^N \big(\log(f^N_t)(\eta)\big) f^N_t(\eta) \nu_N(\eta) &\le 2\sum_\eta \sqrt{f^N_t(\eta)} \cL^N \Big(\sqrt{f^N_t(\eta)}\Big) \nu_N(\eta)\\
&= -2 D_N(f^N_t(\eta))\;.
\end{align*}
We thus get
$$ H_N(f^N_t | \nu_N) + 2 \int_0^t D_N(f^N_s(\eta)) \le H_N(f^N_0)\;.$$
The convexity of the Dirichlet form and the positivity of the entropy yield
$$ D_N(\frac1{t} \int_0^t f^N_s ds) \le \frac1{t} \int_0^t D_N(f^N_s) ds \le \frac{1}{2t}H_N(f^N_0) \;.$$
Let us estimate this last quantity. Using the inequality $x\log x \le 0$ that holds for all $x\le 1$ we have
\begin{align*}
H_N(f^N_0) &= H_N(\iota_N) = \iota_N\Big[\log \frac{d\iota_N}{d\nu_N} \Big]\\
&= \sum_{\eta} \iota_N(\eta)\Big( \log \iota_N(\eta) - \log \nu_N(\eta) \Big)\\
&\le \sum_{\eta} \iota_N(\eta) \log(1/\nu_N(\eta))\\
&\le \max_{\eta} \log(1/\nu_N(\eta))\;.
\end{align*}
From the definition of $\nu_N$ given in \eqref{Eq:DefnuN}, we deduce that
$$ 1/\nu_N(\eta) = 2^{2N} \Big(\frac{p_N}{1-p_N}\Big)^{-\frac{A(S)}{2}} e^{L_S(h^N)+\frac{2\sigma}{(2N)^\alpha}(N+\frac12)S(2N)}\;,$$
where $S$ is the height function associated to $\eta$. Maximising over $\eta$ this quantity, we find
$$ \max_{\eta} 1/\nu_N(\eta) \le 2^{2N} \Big(\frac{p_N}{1-p_N}\Big)^{(2N)^2} e^{L_S(h^N) + 2\sigma (2N)^{2-\alpha}}\;.$$
Using the computations made in the proof of Proposition \ref{Prop:Partition}, we deduce that there exists $C>0$ such that for all $N\ge 1$ we have
$$ H_N(f^N_0) \le C N^{1\vee(2-\alpha)}\;.$$
Let $G:\{0,1\}^{2N}\rightarrow\R_+$. For any measure $\iota_N$ we have
\begin{align*}
\bbP^N_{\iota_N} \Big(\int_0^t G(\eta_s) ds \geq \delta \Big) &\leq \delta^{-1} \bbE^N_{\iota_N} \bigg[\int_0^t G(\eta_s) ds \bigg]\\
&\le \delta^{-1} t \sum_{\eta} (\frac1{t} \int_0^t \iota_N P^N_s ds)(\eta) G(\eta)\\
&\le \delta^{-1} t \sup_{f:D_N(f) \leq \frac{C}{2t} N^{1\vee(2-\alpha)}} \sum_{\eta} \nu_N(\eta) G(\eta) f(\eta)\;,
\end{align*}
where the supremum is taken over all $f:\{0,1\}^{2N}\rightarrow\R_+$ such that $\nu_N[f]=1$.
\end{proof}
We now proceed to the proof of the Replacement Lemma.
\begin{proof}[Proof of Theorem \ref{Th:Replacement}]
Fix $\ell \ge 1$. Following the calculation performed on p.120 of~\cite{KOV}, we get
\begin{align*}
V_{\epsilon N}(\tau_k\eta) &= \Big|\frac1{2\epsilon N+1} \sum_{|j-k|\le \epsilon N} \tau_j \Phi(\eta) - \tilde{\Phi}\Big(\frac1{2\epsilon N+1} \sum_{|j'-k|\le \epsilon N} \eta(j')\Big) \Big|\\
&\le \frac1{2\epsilon N+1} \sum_{|j-k|\le \epsilon N} \Big| \frac1{2\ell + 1} \sum_{|n-j|\le \ell} \tau_n \Phi(\eta) - \tilde{\Phi}\Big(\frac1{2\epsilon N+1} \sum_{|j'-k|\le \epsilon N} \eta(j')\Big) \Big|\\
&+ \cO(\ell/N)
\end{align*}
where the last term is uniform over all $k$ and $\eta$. Then, we write
\begin{align*}
& \frac1{2\epsilon N+1} \sum_{|j-k|\le \epsilon N} \Big| \frac1{2\ell + 1} \sum_{|n-j|\le \ell} \tau_n \Phi(\eta) - \tilde{\Phi}\Big(\frac1{2\epsilon N+1} \sum_{|j'-k|\le \epsilon N} \eta(j')\Big) \Big|\\
&\le \frac1{2\epsilon N+1} \sum_{|j-k|\le \epsilon N} \Big| \frac1{2\ell + 1} \sum_{|n-j|\le \ell} \tau_n \Phi(\eta) - \tilde{\Phi}\Big(\frac1{2\ell+1} \sum_{|n-j|\le \ell} \eta(n)\Big) \Big|\\
&+ \frac1{2\epsilon N+1} \sum_{|j-k|\le \epsilon N} \Big|\tilde{\Phi}\Big(\frac1{2\ell+1} \sum_{|n-j|\le \ell} \eta(n)\Big) - \tilde{\Phi}\Big(\frac1{2\epsilon N+1} \sum_{|j'-k|\le \epsilon N} \eta(j')\Big)\Big|\;.
\end{align*}
The latter term can be bounded as follows
\begin{align*}
&\frac1{2\epsilon N+1} \sum_{|j-k|\le \epsilon N} \Big|\tilde{\Phi}\Big(\frac1{2\ell+1} \sum_{|n-j|\le \epsilon N} \eta(n)\Big) - \tilde{\Phi}\Big(\frac1{2\epsilon N+1} \sum_{|j'-k|\le \epsilon N} \eta(j')\Big)\Big|\\
&\le \frac{\|\tilde{\Phi}'\|_\infty}{2\epsilon N+1} \sum_{|j-k|\le \epsilon N} \Big|\frac1{2\ell+1} \sum_{|n-j|\le \ell} \eta(n) - \frac1{2\epsilon N+1} \sum_{|j'-k|\le \epsilon N} \eta(j')\Big|\\
&\le \frac{\|\tilde{\Phi}'\|_\infty}{(2\epsilon N+1)^2} \sum_{|j-k|\le \epsilon N}\sum_{|j'-k|\le \epsilon N} \Big|\frac1{2\ell+1} \sum_{|n-j|\le \ell} \eta(n) - \frac1{2\ell+1} \sum_{|n-j'|\le \ell} \eta(n)\Big|\\
&+ \cO(\ell/N)\;.
\end{align*}
Putting everything together, we showed that:
\begin{equation}\label{Eq:CalcKOV}\begin{split}
V_{\epsilon N}(\tau_k\eta) &\leq \frac{\|\tilde{\Phi}'\|_\infty}{(2\epsilon N +1)^2} \sum_{j: |j-k| \leq \epsilon N}\sum_{j':|j'-k|\leq \epsilon N} \Big|\ccM_{T_\ell(j')}(\eta) - \ccM_{T_\ell(j)}(\eta)\Big|\\
& + \frac{1}{2\epsilon N +1} \sum_{j:|j-k| \leq \epsilon N} V_{\ell}(\tau_j\eta) + \cO\Big(\frac{\ell}{N}\Big)\;,
\end{split}\end{equation}
where the $\cO\big(\frac{\ell}{N}\big)$ is uniform in $k$ and $\eta$, so that it has a negligible contribution in (\ref{Eq:Replacement}) when $N\to\infty$ and then $\ell \to \infty$.\\
To control the contribution of the second term of \eqref{Eq:CalcKOV}, we apply Lemma \ref{Lemma:Radon} with the map
$$ G(\eta) = \frac{1}{N}\sum_{k=1}^{2N} \frac{1}{2\epsilon N +1} \sum_{j:|j-k| \leq \epsilon N} V_{\ell}(\tau_j\eta)\;.$$
We obtain:
\begin{align*}
 {}&\bbP^N_{\iota_N} \Big(\int_0^t \frac{1}{N}\sum_{k=1}^{2N} \frac{1}{2\epsilon N +1} \sum_{j:|j-k| \leq \epsilon N} V_{\ell}(\tau_j\eta_s) ds \geq \delta \Big)\\
 &\leq \delta^{-1} t \sup_{f:D_N(f) \leq C N^{1\vee(2-\alpha)}} \nu_N \bigg[\frac{1}{N}\sum_{k=1}^{2N} \frac{1}{2\epsilon N +1} \sum_{j:|j-k| \leq \epsilon N} V_{\ell}(\tau_j\eta) f(\eta) \bigg]\;,
\end{align*}
so that Lemma \ref{Lemma:OneBlock} ensures that this term has a vanishing contribution as $N\rightarrow\infty$, $\epsilon\downarrow 0$ and then $\ell\rightarrow\infty$.\\
Similarly, for $\alpha \geq 1$ the contribution of the first term of (\ref{Eq:CalcKOV}) is handled by applying Lemma \ref{Lemma:TwoBlocks} combined with the bound of Lemma \ref{Lemma:Radon}.\\
For $\alpha < 1$, the contribution of the first term of (\ref{Eq:CalcKOV}) is dealt with as follows. Using the Markov inequality, we get
\begin{equation}\label{Eq:ExprTwoBlocks}\begin{split}
 {}&\bbP^N_{\iota_N} \Big(\int_0^t \frac{1}{N}\sum_{k=1}^{2N} \frac{\|\tilde{\Phi}'\|_\infty}{(2\epsilon N +1)^2} \sum_{\substack{j: |j-k| \leq \epsilon N\\j':|j'-k|\leq \epsilon N}} \Big|\ccM_{T_\ell(j')}(\eta_s) - \ccM_{T_\ell(j)}(\eta_s)\Big| ds \geq \delta \Big)\\
 &\leq \delta^{-1}  \frac{1}{N}\sum_{k=1}^{2N} \frac{\|\tilde{\Phi}'\|_\infty}{(2\epsilon N +1)^2} \sum_{\substack{j: |j-k| \leq \epsilon N\\j':|j'-k|\leq \epsilon N}} \int_0^t \bbE^N_{\iota_N} \Big[\big|\ccM_{T_\ell(j')}(\eta_s) - \ccM_{T_\ell(j)}(\eta_s)\big|\Big] ds\;.
 \end{split}\end{equation}
Then, for any $\kappa>0$ we write
 \begin{equation*}
\bbE^N_{\iota_N} \Big[\big|\ccM_{T_\ell(j')}(\eta_s) - \ccM_{T_\ell(j)}(\eta_s)\big|\Big] \leq \bbP^N_{\iota_N} \Big(\big|\ccM_{T_\ell(j')}(\eta_s) - \ccM_{T_\ell(j)}(\eta_s)\big| \geq \delta \kappa\Big)  + \delta \kappa\;, 
\end{equation*}
where we have used the fact that $\ccM_{T_\ell(j)}(\eta)$ belongs to $[0,1]$ for all $\ell,j,\eta$. By Lemma \ref{Lemma:TwoBlocks}, we deduce that (\ref{Eq:ExprTwoBlocks}) goes to $0$ as $N\rightarrow\infty$, $\epsilon\downarrow 0$ and $\ell\rightarrow\infty$. This concludes the proof.
\end{proof}

\subsection{Hydrodynamic limit: the parabolic case}\label{Section:HydroParabo}

The goal of this subsection is to prove Theorem \ref{Th:Hydro} for $\alpha\in [1,\infty)$. We start with the proof of tightness, and then we identify the limit. This second task is carried out separately according as $\alpha > 1$ or $\alpha =1$: this is because the limiting PDEs are different and a special treatment is necessary in the second case. We write $\P^N$ for the law of the process involved in the statements. Recall that we write $m^N(t,k)$ instead of $m^N(t,k/2N)$ for simplicity.\\

To prove tightness of the sequence $m^N$ in the Skorohod space $\bbD([0,\infty),\cC([0,1]))$, it suffices to show that the sequence $m^N(t=0,\cdot)$ is tight in $\cC([0,1])$, and that we have for any $T>0$
\begin{equation}\label{Eq:TightnessCriterion}
	\varlimsup_{h\downarrow 0} \varlimsup_{N\rightarrow\infty} \E^N\Big[ \sup_{t,s \leq T, |t-s|\leq h}\| m^N(t,\cdot)-m^N(s,\cdot)\|_\infty\Big]=0\;.
\end{equation}
The former is actually a hypothesis of our theorem. To prove the latter, we introduce a piecewise linear time interpolation of $m^N$, namely we set $t_N:=\lfloor t(2N)^{2}\rfloor$ and
\begin{align*}
	\bar{m}^N(t,\cdot) := &\big(t_N+1-t(2N)^{2}\big) m^N\Big(\frac{t_N}{(2N)^{2}},\cdot\Big)\\
	&+ \big(t(2N)^{2}-t_N\big) m^N\Big(\frac{t_N+1}{(2N)^{2}},\cdot\Big)\;.
\end{align*}
This process is continuous in time so that one can hope that it is (uniformly in $N$) H\"older continuous in space-time. If such a property holds true, then we get tightness of $m^N$ if we are able to control the distance between $m^N$ and $\bar{m}^N$: this is the content of the next lemma.

\begin{lemma}\label{Lemmambarm}
For all $T>0$, we have
\begin{equation*}
\lim_{N\rightarrow\infty}\E^N\Big[\sup_{t\in[0,T]}\|m^N(t,\cdot)-\bar{m}^N(t,\cdot) \|_\infty\Big] = 0 \;.
\end{equation*}
\end{lemma}
The proof of this lemma is almost the same as the proof of Lemma \ref{Lemma:uubar}, so we omit it. This result ensures that it is actually sufficient to show (\ref{Eq:TightnessCriterion}) with $m^N$ replaced by $\bar{m}^N$ in order to get tightness.\\

The following proposition ensures that $\bar{m}^N$ satisfies (\ref{Eq:TightnessCriterion}).

\begin{proposition}\label{Prop:Tightness}
For any $T>0$, there exists $\delta> 0$ such that
\begin{equation*}
\sup_{N\geq 1} \bbE^N\bigg[ \sup_{0\leq s < t \leq T} \frac{\|\bar{m}^N(t,\cdot)-\bar{m}^N(s,\cdot)\|_\infty}{|t-s|^{\delta}}\bigg] < \infty\;.
\end{equation*}
\end{proposition}

\noindent Before we proceed to the proof of this proposition, we need to collect a few preliminary results. The stochastic differential equations solved by the discrete process $m^N$ are given by
\begin{align*}
dm^N(t,\ell) &= \frac{(2N)^{2}}{2} \Delta m^N(t,\ell)dt\\
&+ (2N)(2p_N-1) \tun_{\{\Delta S(t(2N)^{2},\ell)\ne 0\}}dt + dM^N(t,\ell)\;,
\end{align*}
where $M^N$ is a martingale with bracket given by
\begin{equation*}
d\langle M^N(\cdot,\ell)\rangle_t = 4\Big(p_N\tun_{\{\Delta S(t(2N)^2,\ell) > 0\}} + (1-p_N)\tun_{\{\Delta S(t(2N)^2,\ell) < 0\}}\Big)dt\;.
\end{equation*}
If we let $p^N_t(k,\ell)$ be the fundamental solution of the discrete heat equation:
\begin{align*}
	\begin{cases}
	\partial_t p^N_t(k,\ell) = \frac{(2N)^2}{2} \Delta p^N_t(k,\ell)\;,\\
	p^N_0(k,\ell) = \delta_k(\ell)\;,\\
	p^N_t(k,0) = p^N_t(k,2N) = 0\;,\end{cases}
\end{align*}
then it is simple to check that we have
\begin{equation}\label{Eq:MildHydroHeat}\begin{split}
m^N(t,\ell)&=\sum_k p^N_t(k,\ell)m^N(0,k)+ N^t_t(\ell)\\
&+ (2N)(2p_N-1) \int_0^t \sum_k p^N_{t-s}(k,\ell)\tun_{\{\Delta S(s(2N)^2,k)\ne 0\}} ds \;,
\end{split}\end{equation}
where $N^t_s(\ell)$ is the martingale defined by
\begin{equation*}
N^t_s(\ell) := \int_0^s \sum_k p^N_{t-r}(k,\ell) dM^N(r,k)\;,\quad s\in[0,t]\;.
\end{equation*}
\begin{lemma}\label{Lemma:IncrTime}
For all $\delta \in (0, \frac{1}{2})$, all $T>0$ and all $p\geq 1$, we have
\begin{equation}
	\E^N\Big[|m^N(t',x)-m^N(t,x)|^p \Big]^{\frac{1}{p}} \lesssim |t'-t|^\delta + \frac{1}{\sqrt{2N}} \;,
\end{equation}
uniformly over all $t',t\in [0,T]$, all $x \in [0,1]$ and all $N\geq 1$.
\end{lemma}
Observe that the term $1/\sqrt{2N}$ reflects the discontinuous nature of the process $m^N$.
\begin{proof}
Let $t' > t$. Given the expression (\ref{Eq:MildHydroHeat}), the increment $m^N(t',\ell)-m^N(t,\ell)$ can be written as the sum of three terms: the contribution of the initial condition, of the asymmetry and of the martingale terms. We bound separately the $p$-th moments of these three terms. First, we let $\bar{p}^N$ be the fundamental solution of the discrete heat equation on the whole line $\Z$:
\begin{align*}
	\begin{cases}\partial_t \bar{p}^N_t(\ell) = \frac{(2N)^2}{2} \Delta \bar{p}^N_t(\ell)\;,\\
	\bar{p}^N_0(\ell) = \delta_0(\ell)\;,\end{cases}
\end{align*}
Contrary to $p^N$, $\bar{p}^N$ is translation invariant. Let us also extend $m^N$ into a function on the whole line $\Z$: we simply consider the $4N$-periodic, odd function that coincides with $m^N$ on $[0,2N]$. By symmetry for every $\ell \in \{0,\ldots,2N\}$ and all $t\ge 0$ we have
$$ \sum_{k\in\Z} \bar{p}^N_{t}(\ell-k) m^N(0,k) = \sum_{k=1}^{2N-1} p^N_t(k,\ell) m^N(0,k)\;.$$
Therefore,
\begin{align*}
{}&\sum_{k=1}^{2N-1} \big( p^N_{t'}(k,\ell) - p^N_t(k,\ell) \big) m^N(0,k)\\
&= \sum_{k\in\Z} \big( \bar{p}^N_{t'}(\ell-k) - \bar{p}^N_t(\ell-k) \big) m^N(0,k)\\
&= \sum_{k\in\Z} \bar{p}^N_t(k) \sum_{j\in\Z} \bar{p}^N_{t'-t}(j) \big(m^N(0,\ell-k-j) - m^N(0,\ell-k)\big)\;.
\end{align*}
At this point we use the $1$-Lipschitz regularity of the initial condition and a simple bound on the heat kernel (see Lemma \ref{Lemma:BoundHeatKernelZ} for a proof) to get
$$ |\sum_{j\in\Z} \bar{p}^N_{t'-t}(j) \big(m^N(0,\ell-k-j) - m^N(0,\ell-k)\big)| \le \sum_{j\in\Z} \bar{p}^N_{t'-t}(j)\frac{|j|}{2N} \lesssim \sqrt{t'-t}\;.$$
Consequently
\begin{equation*}
\Big|\sum_{k=1}^{2N-1} \big( p^N_{t'}(k,\ell) - p^N_t(k,\ell) \big) m^N(0,k)\Big| \lesssim \sqrt{t'-t}\;,
\end{equation*}
uniformly over all $\ell\in\{1,\ldots,2N-1\}$, all $t\leq t' \in [0,T]$ and all $N\geq 1$.\\
We turn to the contribution of the asymmetry. Using the estimate
$$ \big| p^N_{t'-s}(k,\ell) - p^N_{t-s}(k,\ell) \big| \lesssim \frac{1}{(2N) \sqrt{t-s}}\Big(\frac{t'-t}{t-s}\Big)^{\delta}\;,$$
whose proof is presented in Lemma \ref{Lemma:BoundHeatKernel}, we get the following almost sure bound
\begin{align*}
{}&\Big| \int_0^{t'} \sum_k p^N_{t'-s}(k,\ell) \tun_{\{\Delta S(s(2N)^2,k)\ne 0\}} ds - \int_0^{t} \sum_k p^N_{t-s}(k,\ell) \tun_{\{\Delta S(s(2N)^2,k)\ne 0\}} ds\Big|\\
&\leq \int_0^t \sum_k \big| p^N_{t'-s}(k,\ell) - p^N_{t-s}(k,\ell) \big| ds + \int_t^{t'} \sum_k p^N_{t'-s}(k,\ell) ds\\
&\lesssim (t'-t)^\delta\;,
\end{align*}
uniformly over all $\ell \in \{1,\ldots,2N-1\}$, all $t\le t'\in[0,T]$ and all $N\geq 1$. Finally, we treat the martingale term: since this term does not have contribution in the hydrodynamic limit, we can use the rough inequality $|N^t_t(\ell)-N^{t'}_{t'}(\ell)| \le |N^t_t(\ell)|+|N^{t'}_{t'}(\ell)|$ and we bound separately the two corresponding terms. By symmetry, it suffices to bound $|N^t_t(\ell)|$. Since $p^N_{t-r}(k,\ell) \lesssim N^{-1} (t-r)^{-1/2}$, see Lemma \ref{Lemma:BoundHeatKernel}, we have the almost sure bound
\begin{equation*}
\big\langle N^{t}_\cdot(\ell) \big\rangle_{t} \leq 4 \int_0^{t}\sum_k p^N_{t-r}(k,\ell)^2 dr\lesssim \frac{1}{2N}\;,
\end{equation*}
uniformly over all $t\in[0,T]$. Since the jumps of the martingale $N^t_\cdot(\ell)$ are all of size at most $1/N$, we apply the Burkholder-Davis-Gundy inequality (\ref{Eq:BDG3}) and get the bound
$$ \E^N\Big[|N^t_t(\ell)|^p \Big]^{\frac{1}{p}} \lesssim \frac1{\sqrt{2N}}\;,$$
as required. This concludes the proof.
\end{proof}

\begin{proof}[Proof of Proposition \ref{Prop:Tightness}]
Fix $T>0$. Recall the definition of $\bar{m}^N$. Arguing differently according to the relative values of $|t'-t|$ and $(2N)^{-2}$ (similarly as what we did in the proof of Lemma \ref{Lemma:HolderFluct}), one deduces from Lemma \ref{Lemma:IncrTime} that there exists $\delta >0$ such that for any $p\geq 1$
\begin{equation*}
	\E^N\Big[|\bar{m}^N(t',x)-\bar{m}^N(t,x)|^p \Big]^{\frac{1}{p}} \lesssim |t'-t|^\delta \;,
\end{equation*}
uniformly over all $t',t\in [0,T]$, all $x \in [0,1]$ and all $N\geq 1$. Using the $1$-Lipschitz regularity in space of $m^N$ and the definition of $\bar{m}^N$, we also get
\begin{align*}
	{}&\E^N\Big[\big|\bar{m}^N(t,x)-\bar{m}^N(t,y)\big|^p\Big]^{\frac{1}{p}}\\
	&\leq \sum_{j=0}^1 \E^N\Big[\Big|m^N\Big(\frac{t_N+j}{(2N)^{2}},x\Big)-m^N\Big(\frac{t_N+j}{(2N)^{2}},y\Big)\Big|^p\Big]^{\frac{1}{p}}\\
	&\lesssim |x-y|\;,
\end{align*}
uniformly over all $x,y\in [0,1]$, all $t \in [0,T]$ and all $N\geq 1$. Combining these two bounds, we obtain for all $p\geq 1$,
\begin{equation*}
	\E^N\big[|\bar{m}^N(t',x)-\bar{m}^N(t,y)|^p\big]^{\frac{1}{p}} \lesssim \big(|t'-t|+|x-y|\big)^\delta\;,
\end{equation*}
uniformly over the same set of parameters. Kolmogorov's Continuity Theorem then ensures that $\bar{m}^N$ admits a modification satisfying the bound stated in Proposition \ref{Prop:Tightness} uniformly in $N\geq 1$ for some $\delta > 0$. Since $\bar{m}^N$ is already continuous, it coincides with its modification $\P^N$-a.s., thus concluding the proof.
\end{proof}

We now proceed to the proof of Theorem \ref{Th:Hydro}: we argue differently in the cases $\alpha\in(1,\infty)$ and $\alpha=1$. In both cases, we set
\begin{equation}\label{Eq:InnerProduct}
\langle f,g \rangle_N = \frac1{2N}\sum_{k=1}^{2N}f\Big(\frac{k}{2N}\Big)g\Big(\frac{k}{2N}\Big)\;.
\end{equation}

\begin{proof}[Proof of Theorem \ref{Th:Hydro}, $\alpha \in (1,\infty)$]
We already know that the sequence $m^N, N\geq 1$ is tight. Let $m$ be the limit of a convergent subsequence. To conclude the proof, we only need to show that for any $\varphi\in\cC^2([0,1])$ such that $\varphi(0)=\varphi(1)=0$, we have
\begin{equation}\label{Eq:m}
\langle m(t),\varphi \rangle = \langle m(0),\varphi \rangle + \frac{1}{2} \int_0^t \langle m(s),\varphi'' \rangle ds\;.
\end{equation}
This characterises the unique weak solution of the PDE (\ref{PDEHeat}).\\
The definition of our dynamics implies that for all $\varphi\in\cC^2([0,1])$ such that $\varphi(0)=\varphi(1)=0$, we have
\begin{equation}\label{Eq:mN}\begin{split}
\langle m^N(t),\varphi\rangle_N =& \langle m^N(0),\varphi \rangle_N + \frac{1}{2} \int_0^t \big\langle m^N(s), (2N)^2 \Delta \varphi \big\rangle_N ds\\
&+ \cO(N^{1-\alpha}) + M^N_t(\varphi)\;,
\end{split}\end{equation}
where $M^N(\varphi)$ is a martingale with bracket
\begin{align*}
\langle M^N(\varphi) \rangle_t &= \int_0^t\frac{4}{2N} \big\langle \varphi^2, p_N \tun_{\{\Delta S(s(2N)^2,\cdot)> 0\}} + (1-p_N) \tun_{\{\Delta S(s(2N)^2,\cdot)< 0\}} \big\rangle_N ds\\
&\leq \frac{4t\|\varphi\|_\infty^2}{2N}\;.
\end{align*}
The jumps of $M^N(\varphi)$ are almost surely bounded by a term of order $N^{-1}$, uniformly over all $N\geq 1$. Then, by the Burkholder-Davis-Gundy inequality (\ref{Eq:BDG3}) we get
\begin{equation*}
	\E^N\Big[\sup_{t\leq T} \big|M^N_t(\varphi)\big|^2\Big]^{\frac{1}{2}} \lesssim \frac1{\sqrt{N}} + \frac1{N}\;,
\end{equation*}
uniformly over all $N\geq 1$, so that $M^N(\varphi)$ vanishes in probability as $N\rightarrow\infty$. Then classical arguments ensure that, along a convergent subsequence of $m^N$, we can pass to the limit on (\ref{Eq:mN}) and get (\ref{Eq:m}), thus concluding the proof.
\end{proof}

\begin{proof}[Proof of Theorem \ref{Th:Hydro}, $\alpha = 1$]
In that case, we characterise the limit via the Hopf-Cole transform $\xi(t,x) = \exp(-2\sigma m(t,x)+2\sigma^2 t)$ that maps, formally, the PDE (\ref{PDEHC}) into
\begin{equation}\label{PDEHC2}
\begin{cases}
\partial_t \xi = \frac{1}{2} \partial^2_x \xi\;,\quad x\in [0,1]\;,\quad t>0\;,\\
\xi(t,0)=\xi(t,1)=e^{2\sigma^2 t}\;,\quad \xi(0,\cdot) = e^{-2\sigma m(0,\cdot)}\;.
\end{cases}
\end{equation}
This equation admits a unique weak solution in the space of continuous space-time functions, and it is well-known that the unique weak solution of (\ref{PDEHC}) coincides with the latter solution upon reverse Hopf-Cole transform.\\
A famous result due to G\"artner~\cite{Gartner88} shows that a similar transform, performed at the level of the exclusion process, linearises the drift of the stochastic differential equations solved by our discrete process. Namely, if one sets
\begin{equation*}
\gamma_N = \frac{2\sigma}{2N}\;,\quad c_N = \frac{(2N)^2}{e^{\gamma_N}+e^{-\gamma_N}} \;,\quad \lambda_N = c_N (e^{\gamma_N} - 2 + e^{-\gamma_N})\;,
\end{equation*}
and
\begin{equation*}
\xi^N(t,x) := e^{-\gamma_N S(t(2N)^2,2N x) + \lambda_N t} \;,\quad x\in [0,1]\;,\quad t\geq 0\;,
\end{equation*}
then, using the abusive notation $\xi^N(t,k)$ for $\xi^N(t,x)$ when $x=k/2N$, we have
\begin{equation}
\begin{cases}
d\xi^N(t,k) = c_N \Delta \xi^N(t,k)dt + d\tilde{M}^N(t,k)\;,\\
\xi^N(t,0)=\xi^N(t,1) = e^{\lambda_N t}\;,\end{cases}
\end{equation}
where $\Delta$ is the discrete Laplacian and $\tilde{M}^N(t,k)$ is a martingale with quadratic variation given by
\begin{equation}\label{Eq:MgaleCH}\begin{split}
\langle \tilde{M}^N(\cdot,k) \rangle_t = (2N)^2\int_0^t \xi^N(s,k)^2\Big(&\big(e^{-2\gamma_N} - 1 \big)^2 \tun_{\{\Delta S(s(2N)^2,k) > 0\}} p_N\\
+& \big(e^{2\gamma_N} - 1 \big)^2 \tun_{\{\Delta S(s(2N)^2,k) < 0\}}(1-p_N)\Big) ds\;.
\end{split}\end{equation}
The tightness of $m^N$ implies the tightness of $\xi^N$. It only remains to identify the limit. To that end, we observe that for all $\varphi\in\cC^2([0,1])$ such that $\varphi(0)=\varphi(1)=0$, we have
\begin{align*}
\langle \xi^N(t),\varphi\rangle_N &= \langle \xi^N(0),\varphi\rangle_N + R^N_t(\varphi)\\
&+c_N \int_0^t \Big(\big\langle \xi^N(s), \Delta \varphi \big\rangle_N + \frac1{2N}e^{\lambda_Ns}\Big(\varphi\Big(\frac1{2N}\Big)+\varphi\Big(\frac{2N-1}{2N}\Big)\Big)\Big) ds\;,
\end{align*}
where
\begin{equation*}
R^N_t(\varphi) = \int_0^t \frac{1}{2N}\sum_{k=1}^{2N-1} \varphi\Big(\frac{k}{2N}\Big) d\tilde{M}^N(s,k)\;.
\end{equation*}
It is elementary to check that there exists $C>0$ such that $|\xi^N(t,k)| \leq C$ for all $t$ in a compact set of $\R_+$, all $k\in\{1,\ldots,2N-1\}$ and all $N\geq 1$. Consequently there exists $C'>0$ such that $\langle \tilde{M}^N(\cdot,k) \rangle_t \leq C' t$ uniformly over the same set of parameters. Moreover, the jumps of this martingale are uniformly bounded by some constant on the same set of parameters. Then, a simple calculation based on the Burkholder-Davis-Gundy inequality (\ref{Eq:BDG3}) shows that $R^N(\varphi)$ converges to $0$ uniformly on compact sets, as $N\rightarrow\infty$. Hence, any limit $\xi$ of a convergent subsequence of $\xi^N$ satisfies
\begin{align*}
\langle \xi(t),\varphi\rangle &= \langle \xi(0),\varphi\rangle + \frac12 \int_0^t \Big(\big\langle \xi(s), \varphi'' \big\rangle + e^{2\sigma^2 s}(\varphi'(0)-\varphi'(1))\Big) ds\;,
\end{align*}
for all $\varphi$ as above, and therefore coincides with the unique weak solution of (\ref{PDEHC2}), thus concluding the proof.
\end{proof}

\subsection{Hydrodynamic limit: the hyperbolic case}\label{SectionHyperbo}
\textit{This subsection is taken from~\cite{LabbeKPZ}}.\\

For simplicity, we take $\sigma = 1$ in this whole subsection. The general case $\sigma > 0$ can be obtained \textit{mutatis mutandis}.\\

Let us present the outline of this technical section which is split into four parts. We start with a short subsection on the notion of solution that we consider: we show the equivalence between the solution with appropriate Dirichlet boundary conditions and the solution with zero-flux boundary conditions; so that in the rest of the proof we rely on the former notion of solution.\\
Second we establish tightness of the sequences of processes at stake: this is rather elementary since we are looking at the hydrodynamic scale.\\
Third, we assume that the convergence of the density of particles holds when we start from ``simple" initial conditions, that is, given by a product of Bernoulli r.v.~with densities that are piecewise constant. Then we show how to go from simple initial conditions to general initial conditions. To do so, the idea is to bound from above and below the given initial condition by some simple initial conditions, then to run three instances of the particle systems under a monotone coupling (that is, a coupling that preserves the order on the height functions) and finally to use the continuity in the initial condition of the solution map associated to the PDE.\\
Fourth, we prove the convergence of the density of particles when we start from ``simple" initial conditions. The identification of the limit in that case then follows the arguments presented in~\cite{Reza}: we show that the entropy inequalities are satisfied at the microscopic level, see Lemma \ref{Lemma:MicroIneq}, and then we show in Lemma \ref{Lemma:MacroIneq} that they can be transferred to the macroscopic level, using in particular the Replacement Lemma established in Theorem \ref{Th:Replacement} which relies on our assumption on the initial condition to be simple.

\subsubsection{Notion of solution}

Recall the notation introduced in Subsection \ref{Subsection:Replacement}. Let us present the notion of solution that we consider for the Burgers equation with zero-flux boundary condition. This material is taken from~\cite[Def.4]{BFK07}.

\begin{definition}\label{Def:EntropySolutionZeroFlux}
Let $\eta_0\in L^\infty(0,1)$. We say that $\eta\in L^\infty\big((0,\infty)\times (0,1)\big)$ is an entropy solution of (\ref{PDEBurgersDensity}) if:\begin{enumerate}
\item For all $c\in[0,1]$ and all $\varphi\in\cC^\infty_c\big((0,\infty)\times(0,1),\R_+\big)$, we have
\begin{align*}
{}\int_0^\infty \int_0^1 \Big(&\big|\eta(t,x)-c\big|\partial_t \varphi(t,x)- 2\sgn(\eta(t,x)-c)\\
&\times\big((\eta(t,x)(1-\eta(t,x)) - c(1-c)\big)\partial_x\varphi(t,x)\Big) dx\,dt \geq 0\;,
\end{align*}
\item We have $\esslim_{t\downarrow 0} \int_0^1 \big| \eta(t,x)-\eta_0(x)\big| dx = 0$,
\item We have $\eta(t,x)(1-\eta(t,x)) = 0$ for almost all $t>0$ and all $x\in\{0,1\}$.
\end{enumerate}
\end{definition}
Let us mention that the first condition is sufficient to ensure that $\eta$ has a trace at the boundaries so that the third condition is meaningful. B\"urger, Frid and Karlsen~\cite[Sect. 4 and 5]{BFK07} show existence and uniqueness of entropy solutions with zero-flux boundary condition.

Let us now introduce the Burgers equation with some appropriate Dirichlet boundary conditions:
\begin{equation}\label{PDEBurgersDirichlet}
	\begin{cases}\partial_t \eta = 2\partial_x\big(\eta(1-\eta)\big) \;,\\
	\eta(t,0) = 1\;,\quad \eta(t,1)=0\;,\\
	\eta(0,\cdot) = \eta_0(\cdot)\;.\end{cases}
\end{equation}
The precise definition of the entropy solution of (\ref{PDEBurgersDirichlet}) is the following.
\begin{definition}\label{Def:EntropySolutionDirichlet}
Let $\eta_0\in L^\infty(0,1)$. We say that $\eta\in L^\infty\big((0,\infty)\times (0,1)\big)$ is an entropy solution of (\ref{PDEBurgersDirichlet}) if it satisfies conditions 1. and 2. from Definition \ref{Def:EntropySolutionZeroFlux} together with the so-called BLN conditions
\begin{equation}\label{Eq:BLN}\begin{split}
\sgn(\eta(t,0)-1)\big(\eta(t,0)(1-\eta(t,0)) - c(1-c)\big) &\geq 0\;,\quad \forall c\in [\eta(t,0),1]\;,\\
\sgn(\eta(t,1)-0)\big(\eta(t,1)(1-\eta(t,1)) - c(1-c)\big) &\leq 0\;,\quad \forall c\in [0,\eta(t,1)]\;,
\end{split}\end{equation}
for almost all $t>0$.
\end{definition}

Here again, there is existence and uniqueness of entropy solutions of (\ref{PDEBurgersDirichlet}), see for instance~\cite[Sect. 2.7 and 2.8]{Ruzicka}.

\begin{proposition}\label{Prop:PDE}
The entropy solutions of (\ref{PDEBurgersDensity}) and (\ref{PDEBurgersDirichlet}) coincide.
\end{proposition}
\begin{proof}
Both solutions exist and are unique. Let us show that the solution of (\ref{PDEBurgersDirichlet}) satisfies the conditions of Definition \ref{Def:EntropySolutionZeroFlux}: actually, the two first conditions are automatically satisfied, so we focus on the third one. If $\eta(t,0)\in (0,1)$ then pick $c=1$ and observe that
$$ \eta(t,0)(1-\eta(t,0)) > c(1-c)\;,$$
so that
$$ \sgn(\eta(t,0)-1)\big(\eta(t,0)(1-\eta(t,0)) - c(1-c)\big) < 0\;.$$
This raises a contradiction so that $\eta(t,0)$ needs to be in $\{0,1\}$. A similar argument shows that $\eta(t,1)$ needs to be in $\{0,1\}$.
\end{proof}

As a consequence, we can choose the formulation (\ref{PDEBurgersDirichlet}) in the proof of our convergence result. Let us finally collect some properties of the solutions that we will use later on.

\begin{proposition}\label{Prop:EntropySolution}
Let $\eta_0\in L^\infty(0,1)$. A function $\eta\in L^\infty\big((0,\infty)\times (0,1)\big)$ is the entropy solution of (\ref{PDEBurgersDirichlet}) if and only if for all $c\in[0,1]$ and all $\varphi\in\cC^\infty_c\big([0,\infty)\times[0,1],\R_+\big)$ we have
\begin{equation}\label{Eq:EntropyCond}\begin{split}
{}&\int_0^\infty \int_0^1 \big((\eta(t,x)-c)^{\pm}\partial_t \varphi(t,x) + h^{\pm}(\eta(t,x),c) \partial_x\varphi(t,x)\big) dx\,dt\\
&+ \int_0^1 (\eta_0(x)-c)^\pm \varphi(0,x) dx+ 2 \int_0^\infty \Big((1-c)^\pm \varphi(t,0) + (0-c)^\pm \varphi(t,1)\Big) dt \geq 0\;,
\end{split}\end{equation}
where $(x)^\pm$ denotes the positive/negative part of $x\in\R$, $\sgn^\pm(x)=\pm\tun_{(0,\infty)}(\pm x)$ and $h^\pm(\eta,c):=-2 \sgn^\pm(\eta-c) \big(\eta(1-\eta) - c(1-c)\big)$.\\
Furthermore for any $t>0$, the map $\eta_0\mapsto \eta(t)$ is $1$-Lipschitz in $L^1(0,1)$.
\end{proposition}
\begin{proof}
The notion of solution defined by (\ref{Eq:EntropyCond}) is introduced in~\cite[Def. 1]{Vovelle} and it is shown therein that it coincides with another notion of solution, originally due to Otto, which is based on boundary entropy-entropy flux pairs. It is then shown in~\cite[Th 7.31]{Ruzicka} that the latter notion of solution is equivalent with the notion of solution of Definition \ref{Def:EntropySolutionDirichlet}. This completes the proof of the first part of the statement. The Lipschitz continuity in $L^1(0,1)$ is proved in~\cite[Th. 3]{BFK07} for the Burgers equation with zero-flux boundary conditions. By Proposition \ref{Prop:PDE}, it also holds for (\ref{PDEBurgersDirichlet}), thus concluding the proof.
\end{proof}

\subsubsection{Tightness}

We let $\cM$ be the space of measures on $[0,1]$ with total mass at most $1$, endowed with the topology of weak convergence. Recall the process $\rho^N$ defined in (\ref{Eq:DefrhoN}) and that $m^N$ is the rescaled height process. These two processes are related through:
$$ \rho^N\big(\big[0,\frac{k}{2N}\big]\big) = \frac12 \Big(m^N\big(\frac{k}{2N}\big) + \frac{k}{2N}\Big)\;,\quad k\in\{0,\ldots,2N\}\;.$$

\begin{proposition}\label{Prop:TightnessHyperbo}
Let $\iota_N$ be any probability measure on $\{0,1\}^{2N}$. The sequence of processes $(\rho^N_t, t\geq 0)$, starting from $\iota_N$, is tight in the space $\bbD([0,\infty),\cM)$. Furthermore, the associated sequence of processes $(m^N(t,x),t\geq 0,x\in[0,1])$ is tight in $\bbD([0,\infty),\cC([0,1]))$.
\end{proposition}

\noindent Note that for a generic measure $\iota_N$ on $\{0,1\}^{2N}$, $m^N(t,1)$ is not necessarily equal to $0$.

\begin{proof}
Let $\varphi\in\cC^2([0,1])$. It suffices to show that $\langle \rho^N_0, \varphi\rangle$ is tight in $\R$, and that for all $T>0$
\begin{equation}\label{Eq:TightnessHydro}
\varlimsup_{h\downarrow 0} \varlimsup_{N\rightarrow\infty} \E^N_{\iota_N}\Big[ \sup_{s,t\leq T, |t-s|\leq h}| \langle \rho^N_t-\rho^N_s, \varphi\rangle |\Big]=0\;.
\end{equation}
The former is immediate since $|\langle \rho^N_0, \varphi\rangle| \leq \|\varphi\|_\infty$. Regarding the latter, we let $L^N$ be the generator of our sped-up process and we write
\begin{align*}
\langle \rho^N_t-\rho^N_s, \varphi\rangle = \frac1{2N} \int_s^t \sum_{k=1}^{2N} \varphi(k) L^N \eta_r^N(k) dr + M^N_{s,t}(\varphi)\;,
\end{align*}
where $M^N_{s,t}(\varphi)$ is a martingale. Its bracket can be bounded almost surely as follows
\begin{equation*}
\langle M^N_{s,\cdot}(\varphi) \rangle_t \leq \int_s^t \frac1{(2N)^2} \sum_{k=1}^{2N-1}\big(\nabla\varphi(k)\big)^2 (2N)^{1+\alpha} dr \lesssim \frac{t-s}{(2N)^{2-\alpha}}\;.
\end{equation*}
Since the jumps of this martingale are bounded by a term of order $\|\varphi'\|_\infty / (2N)^2$, and since
$$\E^N_{\iota_N}\Big[ \sup_{t\in [s,s+h]\cap [0,T]}|M^N_{s,t}(\varphi)|\Big]\leq \E^N_{\iota_N}\Big[ \sup_{t\in [s,s+h]\cap [0,T]}|M^N_{s,t}(\varphi)|^2\Big]^\frac12\;,$$
the BDG inequality (\ref{Eq:BDG3}) ensures that we have
\begin{equation}\label{Eq:BdSupMgale}
\varlimsup_{N\rightarrow\infty} \sup_{s\in [0,T]} \E^N_{\iota_N}\Big[ \sup_{t\in [s,s+h]\cap [0,T]}|M^N_{s,t}(\varphi)|\Big]= 0\;.
\end{equation}
This being given, we observe that $M^N_{s,t}(\varphi)=M^N_{r,t}(\varphi) - M^N_{r,s}(\varphi)$ where $r$ is taken to be the largest element in $\{0,h,2h,\ldots, \lfloor\frac{T}{h}\rfloor h\}$ which is below $s$. Therefore, we have
\begin{align*}
\E^N_{\iota_N}\Big[ \sup_{0\le s\le t\le T, |t-s|\leq h}|M^N_{s,t}(\varphi)|\Big] &\le 2\, \E^N_{\iota_N}\Big[\sum_{r=0,h,\ldots,\lfloor\frac{T}{h}\rfloor h} \sup_{t\in [r,r+2h]\cap [0,T]} |M^N_{r,t}(\varphi)|\Big]\\
&\le 2 \sum_{r=0,h,\ldots,\lfloor\frac{T}{h}\rfloor h} \E^N_{\iota_N}\Big[ \sup_{t\in [r,r+2h]\cap [0,T]} |M^N_{r,t}(\varphi)|\Big]\\
&\le \frac{C}{h} \sup_{r\in[0,T]} \E^N_{\iota_N}\Big[ \sup_{t\in [r,r+2h]\cap [0,T]} |M^N_{r,t}(\varphi)|\Big]\;,
\end{align*}
for some $C>0$. Combining this with \eqref{Eq:BdSupMgale} we deduce that
\begin{equation}\label{Eq:TightnessHydro1}
\varlimsup_{h\downarrow 0} \varlimsup_{N\rightarrow\infty} \E^N_{\iota_N}\Big[ \sup_{0\le s\le t\le T, |t-s|\leq h}|M^N_{s,t}(\varphi)|\Big]=0\;.
\end{equation}
Let us bound the term involving the generator. Decomposing the jump rates into the symmetric part (of intensity $1-p_N$) and the totally asymmetric part (of intensity $2p_N-1$), we find
\begin{align*}
\frac1{2N} \sum_{k=1}^{2N} \varphi(k) L^N \eta(k)&=
-(2N)^\alpha (1-p_N) \sum_{k=1}^{2N-1} \nabla\eta(k) \nabla\varphi(k)\\
&\quad- (2N)^\alpha(2p_N-1) \sum_{k=1}^{2N-1} \eta(k+1)(1-\eta(k)) \nabla \varphi(k)\;.
\end{align*}
A simple integration by parts shows that the first term on the right is bounded by a term of order $N^{\alpha-1}$ while the second term is of order $1$. Consequently
\begin{equation}\label{Eq:TightnessHydro2}
\E^N_{\iota_N}\Big[\sup_{s,t \leq T, |t-s|\leq h} \Big|\frac1{2N} \int_s^t \sum_{k=1}^{2N} \varphi(k) L^N \eta_r(k) dr \Big|\Big] \lesssim h\;,
\end{equation}
uniformly over all $N\geq 1$ and all $h>0$. The l.h.s.~vanishes as $N\rightarrow\infty$ and $h\downarrow 0$. Combining (\ref{Eq:TightnessHydro1}) and (\ref{Eq:TightnessHydro2}), (\ref{Eq:TightnessHydro}) follows.\\
We turn to the tightness of the interface $m^N$. First, the profile $m^N(t,\cdot)$ is $1$-Lipschitz for all $t\geq 0$ and all $N\geq 1$. Second, we claim that for some $\beta\in(\alpha,1)$ 
\begin{equation}\label{Eq:IncrHyperbo}
\E^N_{\iota_N} \Big[ |m^N(t,k)-m^N(s,k)|^p \Big]^{\frac1{p}} \lesssim |t-s| + \frac1{N^{1-\beta}}\;,
\end{equation}
uniformly over all $0 \leq s \leq t \leq T$, all $k\in\{1,\ldots,2N\}$ and all $N\geq 1$. This being given, the arguments for proving tightness are classical: one introduces a piecewise linear time-interpolation $\bar{m}^N$ of $m^N$ and shows tightness for this process, and then one shows that the difference between $\bar{m}^N$ and $m^N$ is uniformly small. We are left with the proof of (\ref{Eq:IncrHyperbo}). Let $\psi:\R\rightarrow\R_+$ be a non-increasing, smooth function such that $\psi(x)=1$ for all $x\leq 0$ and $\psi(x)=0$ for all $x\geq 1$. Fix $\beta \in (\alpha,1)$. For any given $k\in\{1,\ldots,2N\}$, we define $\varphi^N_k:\{0,\ldots,2N\}\rightarrow\R$ by setting $\varphi^N_k(\ell) = \psi\big((\ell-k)/(2N)^\beta\big)$. Then, we observe that
\begin{equation*}
\frac1{2N}\sum_{\ell=1}^{2N} \big(2\eta_t(\ell)-1\big) \varphi^N_k(\ell) = m^N(t,k) + \cO(N^{\beta-1})\;,
\end{equation*}
uniformly over all $k\in\{1,\ldots,2N\}$ and all $t\geq 0$. Then, similar computations to those made in the first part of the proof show that
\begin{equation*}
\E^N_{\iota_N}\Big[\Big|\frac1{2N}\sum_{\ell=1}^{2N}\big(\eta_t(\ell)-\eta_s(\ell)\big) \varphi^N_k(\ell)\Big|^p\Big]^{\frac1{p}} \lesssim (t-s) + \sqrt{\frac{t-s}{N^{1+\beta-\alpha}}} + \frac1{N^{1+\beta}}\;,
\end{equation*}
uniformly over all $k$, all $0\leq s \leq t \leq T$ and all $N\geq 1$. This yields (\ref{Eq:IncrHyperbo}).
\end{proof}

\subsubsection{Identification of the limit}

The main step in the proof of Theorem \ref{Th:Hydro} is to prove the convergence of the density of particles, starting from a product measure $\iota_N$ satisfying Assumption \ref{Assumption:IC}. Notice that if the process starts from $\iota_N$, then $\rho^N_0$ converges to the deterministic limit $\rho_0(dx)=f(x)dx$.

\begin{theorem}\label{Th:HydroProd}
Under Assumption \ref{Assumption:IC}, the process $\rho^N$ converges in distribution in the Skorohod space $\bbD\big([0,\infty),\cM\big)$ to the deterministic process $(\eta(t,x)dx,t\geq 0)$, where $\eta$ is the entropy solution of (\ref{PDEBurgersDirichlet}) starting from $\eta_0=f$.
\end{theorem}

\noindent Given this result, the proof of the hydrodynamic limit is derived as follows.
\begin{proof}[Proof of Theorem \ref{Th:Hydro}]
Let $\iota_N$ be as in Theorem \ref{Th:HydroProd}. We know that $\rho^N$ converges to $\rho$, where $\rho(t,dx) =\eta(t,x)dx$ and $\eta$ is the entropy solution of (\ref{PDEBurgersDirichlet}) starting from $\eta_0=f$. Let us show that $m^N$ converges to the integrated solution associated to $\rho$, namely
$$ (t,x) \mapsto \int_0^x (2\eta(t,y)-1) dy\;.$$
Once this will be proved, we will have completed the proof of our theorem when the initial condition satisfies Assumption \ref{Assumption:IC}.

Let $m$ be the limit point of some convergent subsequence $m^{N_i}$. By Skorohod's representation theorem, we can assume that $(\rho^{N_i},m^{N_i})$ converges almost surely to $(\rho,m)$. Recall that $\rho$ is of the form $\rho(t,x)=\eta(t,x)dx$. Our goal is to show that $m(t,x) = \int_0^x (2\eta(t,y) - 1)dy$ for all $t,x$.\\
Fix $x_0\in (0,1)$. Take $\varphi_p$ be a $\cC^\infty$ function that approximates the indicator of $[0,x_0]$. More precisely, we assume that for any $\delta > 0$ we have
\begin{equation}\label{Eq:phip}
\big\| \varphi_p - \tun_{[0,x_0]} \big\|_{L^1(0,1)} \rightarrow 0\;,\quad \sup_{f\in\cC^\delta([0,1])} \frac{\big| \langle f, \delta_{x_0} + \partial_x \varphi_p \rangle \big|}{\| f \|_{\cC^\delta}} \rightarrow 0\;,
\end{equation}
as $p\rightarrow\infty$. Such a function exists, take for instance $\varphi_p(\cdot) = 1- \int_{-\infty}^\cdot P_{1/p}(y-x_0) dy$ where $P_t$ is the heat kernel on $\R$ at time $t$. Considering a smooth approximation of the indicator of $[0,x_0]$ is convenient as we would like to pass to the limit on $\langle \rho^N_t, \tun_{[0,x_0]}\rangle$ but we established convergence in the topology of weak convergence of probability measures, and the indicator is not continuous.\\
If we set $I(t,x_0)= m(t,x_0) - \int_0^{x_0} (2\eta(t,y) - 1)dy$ for some $x_0\in (0,1)$ and some $t>0$, then $| I(t,x_0) |$ is bounded by
\begin{align*}
{}&\big\| m(t)- m^{N_i}(t)\big\|_\infty+\big| \langle m^{N_i}(t),\delta_{x_0} + \partial_x \varphi_p \rangle \big|+ \big|\langle m^{N_i}(t),\partial_x\varphi_p\rangle+\langle2\rho^{N_i}_t-1,\varphi_p\rangle\big|\\
&+ 2\big| \langle \rho^{N_i}_t-\rho_t,\varphi_p \rangle \big|+ \big| \langle 2\rho(t)-1,\varphi_p - \tun_{[0,x_0]}\rangle\big|\;.
\end{align*}
Recall that $m^N$ is $1$-Lipschitz in space, so that the second term vanishes as $p\rightarrow\infty$ by (\ref{Eq:phip}). A discrete integration by parts shows that the third term vanishes as $N_i$ goes to $\infty$. The first and fourth terms vanish as $N_i\rightarrow\infty$ by the convergences of $m^{N_i}$ and $\rho^{N_i}$, and the last term is dealt with using (\ref{Eq:phip}). Choosing $p$ and then $N_i$ large enough, we deduce that $|I(t,x_0)|$ is almost surely as small as desired. This identifies completely the limit $m$ of any convergent subsequence, under $\P^N_{\iota_N}$.

\medskip

We are left with the extension of this convergence result to an arbitrary initial condition. Assume that $m^N(0,\cdot)$ converges to some profile $m_0(\cdot)$ for the supremum norm. Since each $m^N(0,\cdot)$ is $1$-Lipschitz, so is $m_0$.\\
The idea is to squeeze $m_0$ in between two elements $m_0^{\epsilon,+}$ and $m_0^{\epsilon,-}$ that are the scaling limits of initial conditions for which the result has already been proved, then to consider a monotone coupling of three instances of our height processes starting from approximations of these three initial conditions and then to combine this monotonicity with the continuity of the solution map of our PDE to deduce the convergence result.\\
More precisely for every $\epsilon > 0$, one can find two profiles $m_0^{\epsilon,+}$ and $m_0^{\epsilon,-}$ which are $1$-Lipschitz, piecewise affine, start from $0$ and are such that:
\begin{itemize}
\item $m_0^{\epsilon,-}$ stays below $m_0$: $m_0-\epsilon \leq m_0^{\epsilon,-} \leq (m_0-\frac{\epsilon}{4})\vee(-x)\vee(x-1)\;,$
\item $m_0^{\epsilon,+}$ stays above $m_0$: $(m_0+\frac{\epsilon}{4})\wedge x \wedge(1-x) \leq m_0^{\epsilon,+} \leq m_0 + \epsilon\;,$
\item $\| \rho_0^{\epsilon,\pm} - \rho_0\|_{L^1} \rightarrow 0$ as $\epsilon \downarrow 0$, where $\rho_0^{\epsilon,\pm}(\cdot) = \big(\partial_x m_0^{\epsilon,\pm}(\cdot) +1\big)/2$.
\end{itemize}
The proof of the existence of such profiles is postponed below. Now consider a coupling $(m^{N,\epsilon,-},m^N,m^{N,\epsilon,+})$ of three instances of our height process which preserves the order of the interfaces and is such that $m^{N,\epsilon,\pm}(0,\cdot)$ is the height function associated with the particle density distributed as
$$\otimes_{k=1}^{2N} \mbox{Be}\big(\rho^{\epsilon,\pm}_0(k/2N)\big)\;.$$
We draw independently these two sets of Bernoulli r.v. It is simple to check that
$$ \P\big(m^{N,\epsilon,-}(0,\cdot) \leq m^N(0,\cdot) \leq m^{N,\epsilon,+}(0,\cdot)\big) \to 1\;,$$
as $N\rightarrow\infty$. By the order preserving property of the coupling, if these inequalities are satisfied at time $0$ they remain true at all times. Our convergence result applies to $m^{N,\epsilon,\pm}$ and, consequently, any limit point of the tight sequence $m^N$ is squeezed in $[m^{\epsilon,-},m^{\epsilon,+}]$ where $m^{\epsilon,\pm}$ is the integrated entropy solution of (\ref{PDEBurgersDirichlet}) starting from $m_0^{\epsilon,\pm}$. By the second part of Proposition \ref{Prop:EntropySolution}, we deduce that $m^{\epsilon,\pm}$ converge, as $\epsilon\downarrow 0$, to the integrated entropy solution of (\ref{PDEBurgersDirichlet}) starting from $m_0$, thus concluding the proof.\\

Let us briefly explain how one can construct $m_0^{\epsilon,-}$, the construction of $m_0^{\epsilon,+}$ being similar. For simplicity, we let $V(x) := (-x)\vee(x-1)$. Let $n\ge 1$ be given. We subdivide $[0,1]$ into three sets:
$$I:=\{x: m_0(x) > V(x) + \epsilon/2\}\;,\quad J_1:= [0,1/2]\backslash I\;,\quad J_2 := (1/2,1] \backslash I\;.$$
On $J_1$ and $J_2$, we set $m_0^{\epsilon,-}(x) = V(x)$. On $I\cap \{k/n: k=0,1,\ldots,n\}$, we set $m_0^{\epsilon,-}(x) =m_0(x) - \frac{\epsilon}{2}$. Then, we extend $m_0^{\epsilon,-}$ to the rest of $I$ by affine interpolation. The fact that $m_0$ is $1$-Lipschitz ensures that $m_0^{\epsilon,-}$ is also $1$-Lipschitz. If $n$ is large enough compared to $1/\epsilon$ we get the inequalities $m_0-\epsilon \leq m_0^{\epsilon,-} \leq (m_0-\frac{\epsilon}{4})\vee(-x)\vee(x-1)$. Regarding the convergence in $L^1$ of the density, we observe that $ \rho^{\epsilon,-}_0(x) = 0$ on $J_1$, $\rho^{\epsilon,-}_0(x) = 1$ on $J_2$ and
$$  \rho^{\epsilon,-}_0(x) = \frac1{|I(x)|} \int_{I(x)} \rho_0(u) du\;,\quad x\in I\;,$$
where $I(x) = I \cap [k/n,(k+1)/n)$ and $k$ is the integer part of $nx$. From there, we deduce that $\| \rho_0^{\epsilon,-} - \rho_0\|_{L^1(J_1)}=\int_{J_1} \rho_0(x)dx$. At this point, we observe that $J_1$ is an interval starting at $0$ and that, at the end point $x$ of $J_1$ we have $m_0(x) \le V(x) - \epsilon/2$. Hence we have $\int_{J_1} \rho_0(x)dx \le \epsilon/4$ and $\| \rho_0^{\epsilon,-} - \rho_0\|_{L^1(J_1)} \le \epsilon/4$. Similarly, $\| \rho_0^{\epsilon,-} - \rho_0\|_{L^1(J_2)}=\int_{J_2} (1-\rho_0(x))dx$ which is also smaller than $\epsilon/4$. Finally, $\| \rho_0^{\epsilon,-} - \rho_0\|_{L^1(I)}$ goes to $0$ as $n\rightarrow\infty$: indeed, the almost everywhere differentiability of $x\mapsto \int_0^x \rho_0(u) du$ ensures that $\rho_0^{\epsilon,-}(x)$ goes to $\rho_0(x)$ for almost all $x\in I$ (see for instance~\cite[Chap. 8]{Rudin}), so that the dominated convergence theorem yields the asserted convergence.
\end{proof}

\subsubsection{Proof of the convergence starting from simple initial conditions}

To prove Theorem \ref{Th:HydroProd}, we need to show that the limit of any convergent subsequence of $\rho^N$ is of the form $\rho(t,dx)=\eta(t,x)dx$ and that $\eta$ satisfies the entropy inequalities of Proposition \ref{Prop:EntropySolution}. To make appear the constant $c$ in these inequalities, the usual trick is to define a coupling of the particle system $\eta^N$ with another particle system $\zeta^N$ which is stationary with density $c$ so that, at large scales, one can replace the averages of $\zeta^N$ by $c$. Such a coupling has been defined by Rezakhanlou~\cite{Reza} in the case of the infinite lattice $\Z$. The specificity of the present setting comes from the boundary conditions of our system: one needs to choose carefully the flux of particles at $1$ and $2N$ for $\zeta^N$.

The precise definition of our coupling goes as follows. We set
\begin{equation*}
p(1)=1-p_N\;,\quad p(-1)=p_N\;,\quad \mbox{and}\quad p(k)=0 \quad \forall k\ne \{-1,1\}\;,
\end{equation*}
as well as $b(a,a')=a(1-a')$. We denote by $\eta^{k,\ell}$ the particle configuration obtained from $\eta$ by permuting the values $\eta(k)$ and $\eta(\ell)$. We also denote by $\eta \pm \delta_k$ the particle configuration which coincides with $\eta$ everywhere except at site $k$ where the occupation is set to $\eta(k) \pm 1$. Then, we define
\begin{align*}
\tilde{\cL}^{\mbox{\tiny bulk}}f(\eta,\zeta) &= (2N)^{1+\alpha}\sum_{k,\ell =1}^{2N} p(\ell-k)\times\\
&\Big[\big( b(\eta(k),\eta(\ell))\wedge b(\zeta(k),\zeta(\ell)) \big)\big( f(\eta^{k,\ell},\zeta^{k,\ell}) - f(\eta,\zeta)\big)\\
& + \big(b(\eta(k),\eta(\ell))- b(\eta(k),\eta(\ell))\wedge b(\zeta(k),\zeta(\ell)) \big)\big( f(\eta^{k,\ell},\zeta) - f(\eta,\zeta)\big)\\
& + \big( b(\zeta(k),\zeta(\ell))- b(\eta(k),\eta(\ell))\wedge b(\zeta(k),\zeta(\ell)) \big)\big( f(\eta,\zeta^{k,\ell}) - f(\eta,\zeta)\big) \Big]\;,
\end{align*}
and
\begin{align*}
\tilde{\cL}^{\mbox{\tiny bdry}} f(\eta,\zeta) &= (2N)^{1+\alpha}(2p_N-1) (1-c) \zeta(1) \big( f(\eta,\zeta-\delta_1) - f(\eta,\zeta)\big)\\
& + (2N)^{1+\alpha}(2p_N-1) c (1-\zeta(2N)) \big( f(\eta,\zeta+\delta_{2N}) - f(\eta,\zeta)\big)\;.
\end{align*}
We consider the stochastic process $(\eta^N_t,\zeta^N_t), t\geq 0$ associated to the generator $\tilde{\cL}=\tilde{\cL}^{\mbox{\tiny bulk}}+\tilde{\cL}^{\mbox{\tiny bdry}}$. From now on, we will always assume that $\eta^N_0$ has law $\iota_N$, where $\iota_N$ satisfies Assumption \ref{Assumption:IC}, and that $\zeta^N_0$ is distributed as a product of Bernoulli measures with parameter $c$. Furthermore, we will always assume that the coupling at time $0$ is such that 
\begin{equation*}
\sgn(\eta^N_0(k)-\zeta^N_0(k)) = \sgn(f(k/2N)-c)\;,\quad \forall k\in\{1,\ldots,2N\}\;,
\end{equation*}
where $f$ is the macroscopic density profile of Assumption \ref{Assumption:IC}. Such a coupling can be constructed by considering i.i.d.~r.v.~$U_1,\ldots,U_{2N}$ uniformly distributed over $[0,1]$, and by setting $\eta^N_0(k)$ (resp.~$\zeta^N_0(k)$) to $1$ if $U_k \le f(k/2N)$ (resp.~$U_k \le c$). We let $\tilde{\P}^N_{\iota_N,c}$ be the law of the process $(\eta^N,\zeta^N)$.
\begin{remark}
The process $\eta^N$ follows the dynamics of the WASEP with zero-flux boundary conditions. The process $\zeta^N$ follows the dynamics of the WASEP with some open boundary conditions chosen in a such a way that the process is stationary with density $c$. Actually, we prescribe the minimal jump rates at the boundary for the process to be stationary with density $c$: there is neither entering flux at $1$ nor exiting flux at $2N$. This choice is convenient for establishing the entropy inequalities. Let us also mention that the coupling is such that the order of $\zeta^N$ and $\eta^N$ is preserved. More precisely, if in both particle systems there is a particle which can attempt a jump from $k$ to $\ell$, then the jump times are simultaneous.
\end{remark}
It will actually be important to track the sign changes in the pair $(\eta^N,\zeta^N)$. To that end, we let $F_{k,\ell}(\eta,\zeta) = 1$ if $\eta(k)\geq \zeta(k)$ and $\eta(\ell)\geq \zeta(\ell)$; and $F_{k,\ell}(\eta,\zeta) = 0$ otherwise. We say that a subset $C$ of consecutive integers in $\{1,\ldots,2N\}$ is a cluster with constant sign if for all $k,\ell \in C$ we have $F_{k,\ell}(\eta,\zeta) =1$, or for all $k,\ell \in C$ we have $F_{k,\ell}(\zeta,\eta) = 1$. For a given configuration $(\eta,\zeta)$, we let $n$ be the minimal number of clusters needed to cover $\{1,\ldots,2N\}$: we will call $n$ the number of sign changes. There is not necessarily a unique choice of covering into $n$ clusters. Let $C(i),i\leq n$ be any such covering and let $1=k_1 <  k_2 < \ldots k_n < k_{n+1} = 2N+1$ be the integers such that $C(i)=\{k_i,k_{i+1}-1\}$.

\begin{lemma}\label{Lemma:Coupling}
Under $\tilde{\P}^N_{\iota_N,c}$, the process $\eta^N$ has law $\P^N_{\iota_N}$ while the process $\zeta^N$ is stationary with law $\otimes_{k=1}^{2N} \mbox{Be}(c)$. Furthermore, the number of sign changes $n(t)$ is smaller than $n(0)+3$ at all time $t\geq 0$.
\end{lemma}
\begin{proof}
It is simple to check the assertion on the laws of the marginals $\eta^N$ and $\zeta^N$. Regarding the number of sign changes, the key observation is the following. In the bulk $\{2,\ldots,2N-1\}$, to create a new sign change we need to have two consecutive sites $k,\ell$ such that $\eta^N(k)=\zeta^N(k)=1$, $\eta^N(\ell)=\zeta^N(\ell)=0$ and we need to let one particle jump from $k$ to $\ell$, but not both. However, our coupling does never allow such a jump. Therefore, the number of sign changes can only increase at the boundaries due to the interaction of $\zeta^N$ with the reservoirs: this can create at most $2$ new sign changes, thus concluding the proof.
\end{proof}

Assumption \ref{Assumption:IC} ensures the existence of a constant $C>0$ such that $n(0) < C$ almost surely for all $N\geq 1$. We now derive the entropy inequalities at the microscopic level. Recall that $\tau_k$ stands for the shift operator with periodic boundary conditions, and let $\langle u,v\rangle_N = (2N)^{-1} \sum_{k=1}^{2N} u(k/2N)v(k/2N)$ denote the discrete $L^2$ product.
\begin{lemma}[Microscopic inequalities]\label{Lemma:MicroIneq}
Let $\iota_N$ be a measure on $\{0,1\}^{2N}$ satisfying Assumption \ref{Assumption:IC}. For all $\varphi\in \cC^\infty_c([0,\infty)\times[0,1],\R_+)$, all $\delta > 0$ and all $c\in [0,1]$, we have $\lim_{N\rightarrow\infty} \tilde{\P}^N_{\iota_N,c}(\cI^N(\varphi) \geq -\delta) = 1$ where
\begin{align*}
\cI^N(\varphi) &:=\int_0^\infty\!\!\! \bigg(\Big\langle \partial_s \varphi(s,\cdot) , \big( \eta^N_s(\cdot) - \zeta^N_s(\cdot) \big)^\pm \Big\rangle_N \!\!\!+ \Big\langle \partial_x \varphi(s,\cdot),H^\pm\big(\tau_\cdot \eta^N_s,\tau_\cdot \zeta^N_s \big) \Big\rangle_N \\
&+ 2 \Big( (1-c)^\pm \varphi(s,0) + (0-c)^\pm \varphi(s,1) \Big)\bigg) ds\\
&+ \Big\langle \varphi(0,\cdot) , \big( \eta^N_0(\cdot) - \zeta^N_0(\cdot) \big)^\pm \Big\rangle_N\;,
\end{align*}
where $H^+(\eta,\zeta) = -2 \big( b(\eta(1),\eta(0))-b(\zeta(1),\zeta(0))\big) F_{1,0}(\eta,\zeta)$ and $H^-(\eta,\zeta)=H^+(\zeta,\eta)$.
\end{lemma}
This is an adaptation of Theorem 3.1 in~\cite{Reza}.
\begin{proof}
We define
\begin{align*}
B_t &=\int_0^t \Big( \big\langle \partial_s\varphi(s,\cdot) , \big(\eta_s(\cdot)-\zeta_s(\cdot)\big)^\pm \big\rangle_N + \tilde{\cL}\big\langle \varphi(s,\cdot) , \big(\eta_s(\cdot)-\zeta_s(\cdot)\big)^\pm \big\rangle_N \Big)ds\\
&\quad+ \Big\langle \varphi(0,\cdot) , \big( \eta^N_0(\cdot) - \zeta^N_0(\cdot) \big)^\pm \Big\rangle_N\;.
\end{align*}
We have the identity
\begin{equation}\label{Eq:Dynkin}
\Big\langle \varphi(t,\cdot) , \big( \eta^N_t(\cdot) - \zeta^N_t(\cdot) \big)^\pm \Big\rangle_N = B_t + M_t\;,
\end{equation}
where $M$ is a mean zero martingale. Since $\varphi$ has compact support, the l.h.s.~vanishes for $t$ large enough. Below, we work at an arbitrary time $s$ so we drop the subscript $s$ in the calculations. Moreover, we write $\varphi(k)$ instead of $\varphi(k/2N)$ to simplify notations. We treat separately the boundary part and the bulk part of the generator. Regarding the former, we have
\begin{align*}
{}&\tilde{\cL}^{\mbox{\tiny bdry}}\big\langle \varphi(\cdot) , \big(\eta(\cdot)-\zeta(\cdot)\big)^+ \big\rangle_N\\
&= (2N)^\alpha (2p_N-1) \Big(\varphi(1)\eta(1)\zeta(1)(1-c) - \varphi(2N)\eta(2N) (1-\zeta(2N)) c\Big)\\
&\leq 2 \varphi(0) (1-c) + \cO(N^{-\alpha})\;,
\end{align*}
since $\varphi$ is non-negative and $2p_N-1 \sim 2(2N)^{-\alpha}$. Similarly, we find
\begin{equation*}
\tilde{\cL}^{\mbox{\tiny bdry}}\big\langle \varphi(\cdot) , \big(\eta(\cdot)-\zeta(\cdot)\big)^- \big\rangle_N\leq 2 \varphi(2N) (0-c)^- + \cO(N^{-\alpha})\;.
\end{equation*}

We turn to the bulk part of the generator. Recall the map $F_{k,\ell}(\eta,\zeta)$, and set $G_{k,\ell}(\eta,\zeta) = 1 - F_{k,\ell}(\eta,\zeta)F_{k,\ell}(\zeta,\eta)$. By checking all the possible cases, one easily gets the following identity
\begin{align*}
\tilde{\cL}^{\mbox{\tiny bulk}}\big(\eta(k)-\zeta(k)\big)^+ &= (2N)^{1+\alpha}\sum_\ell \bigg[ \Big(p(\ell-k)\big(b(\zeta(k),\zeta(\ell))-b(\eta(k),\eta(\ell))\big)\\
&\qquad- p(k-\ell)\big(b(\zeta(\ell),\zeta(k))-b(\eta(\ell),\eta(k))\big) \Big) F_{k,\ell}(\eta,\zeta)\\
&-\Big(p(\ell-k)b(\eta(k),\eta(\ell)) + p(k-\ell)b(\zeta(\ell),\zeta(k))\Big)G_{k,\ell}(\eta,\zeta) \bigg]\;.
\end{align*}
Since $\eta$ and $\zeta$ play symmetric r\^oles in $\tilde{\cL}^{\mbox{\tiny bulk}}$, we find a similar identity for $\tilde{\cL}^{\mbox{\tiny bulk}}\big(\eta(k)-\zeta(k)\big)^-$. Notice that the term on the third line is non-positive, so we will drop it in the inequalities below. We thus get
\begin{equation*}
\tilde{\cL}^{\mbox{\tiny bulk}}\big\langle \varphi(\cdot) , \big(\eta(\cdot)-\zeta(\cdot)\big)^\pm \big\rangle_N \leq (2N)^{\alpha} \sum_{\substack{k,\ell= 1\\ \ell = k \pm 1}}^{2N} p(\ell-k) \big(\varphi(k)-\varphi(\ell)\big) I_{k,\ell}^\pm(\eta,\zeta)\;,
\end{equation*}
where
\begin{equation*}
I_{k,\ell}^+(\eta,\zeta)= \big(b(\zeta(k),\zeta(\ell))-b(\eta(k),\eta(\ell))\big) F_{k,\ell}(\eta,\zeta)\;,\quad I_{k,\ell}^-(\eta,\zeta)= I_{k,\ell}^+(\zeta,\eta) \;.
\end{equation*}
Up to now, we essentially followed the calculations made in the first step of the proof of~\cite[Thm 3.1]{Reza}. At this point, we argue differently: we decompose $p(\pm 1)$ into the symmetric part $1-p_N$, which is of order $1/2$, and the asymmetric part which is either $0$ or $2p_N-1\sim 2 (2N)^{-\alpha}$.\\
We start with the contribution of the symmetric part. Recall the definition of the number of sign changes $n$ and of the integers $k_1 < \ldots < k_{n+1}$. Using a discrete integration by parts, one easily deduces that for all $i \leq n$
\begin{align*}
\sum_{\substack{k,\ell= k_i\\ \ell = k \pm 1}}^{k_{i+1}-1} \big(\varphi(k)-\varphi(\ell)\big) I_{k,\ell}^\pm(\eta,\zeta) &= \sum_{k=k_{i}}^{k_{i+1}-2}\big(\eta(k)-\zeta(k)\big)^\pm \Delta \varphi(k)\\
&- \big(\eta(k_{i+1}-1)-\zeta(k_{i+1}-1)\big)^\pm \nabla \varphi(k_{i+1}-2)\\
&+ \big(\eta(k_{i})-\zeta(k_{i})\big)^\pm \nabla \varphi(k_{i}-1)\;.
\end{align*}
Since $n(s)$ is bounded uniformly over all $N\geq 1$ and all $s\geq 0$, we deduce that the boundary terms arising at the second and third lines yield a negligible contribution. Thus we find
\begin{align*}
(2N)^\alpha \sum_{\substack{k,\ell= 1\\ \ell = k \pm 1}}^{2N} (1-p_N)  \big(\varphi(k)-\varphi(\ell)\big) I_{k,\ell}^\pm(\eta,\zeta) = \cO\Big(\frac1{N^{1-\alpha}}\Big)\;.
\end{align*}
Regarding the asymmetric part $p(\pm 1)-(1-p_N)$, a simple calculation yields the identity
\begin{align*}
{}&(2N)^{\alpha} \sum_{\substack{k,\ell= 1\\ \ell = k \pm 1}}^{2N} \big(p(\ell-k)-1+p_N\big) \big(\varphi(k)-\varphi(\ell)\big) I_{k,\ell}^\pm(\eta,\zeta)\\
&= \frac1{2N}\sum_{k=1}^{2N-1} \partial_x \varphi(k) \tau_k H^\pm(\eta,\zeta) + \cO(N^{-\alpha})\;,
\end{align*}
uniformly over all $N\geq 1$. Therefore
\begin{align*}
\tilde{\cL}^{\mbox{\tiny bulk}}\big\langle \varphi(\cdot) , \big(\eta(\cdot)-\zeta(\cdot)\big)^\pm \big\rangle_N \leq \frac1{2N}\sum_{k=1}^{2N-1} \partial_x\varphi(k) \tau_k H^\pm(\eta,\zeta) + \cO\Big(\frac1{N^{\alpha\wedge(1-\alpha)}}\Big)\;.
\end{align*}
Putting together the two contributions of the generator, we get
\begin{align*}
B_t &\leq \int_0^t \Big( \big\langle  \partial_s\varphi(s,\cdot) , \big(\eta^N_s(\cdot)-\zeta^N_s(\cdot)\big)^\pm \big\rangle_N + \big\langle \partial_x \varphi(s,\cdot) , \tau_\cdot H^\pm(\eta^N_s,\zeta^N_s) \big\rangle_N  \\
&\quad+ 2t \big( (1-c)^\pm \varphi(s,0) + (0-c)^\pm \varphi(s,1) \big) \Big)ds\\
&\quad+ \Big\langle \varphi(0,\cdot) , \big( \eta^N_0(\cdot) - \zeta^N_0(\cdot) \big)^\pm \Big\rangle_N +\cO\Big(\frac1{N^{\alpha\wedge(1-\alpha)}}\Big)\;.
\end{align*}
Recall the equation (\ref{Eq:Dynkin}). A simple calculation shows that $\tilde{\E}^N_{\iota_N,c} \langle M \rangle_t \lesssim \frac{1}{N^{1-\alpha}}$ uniformly over all $N\geq 1$ and all $t\geq 0$. Moreover, the jumps of $M$ are almost surely bounded by a term of order $N^{-1}$. Applying the BDG inequality (\ref{Eq:BDG3}), we deduce that
\begin{equation*}
\tilde{\E}^N_{\iota_N,c} \Big[\sup_{s\leq t} M_s^2\Big]^\frac12 \lesssim \frac{1}{N^{\frac{1-\alpha}{2}}}\;,
\end{equation*}
uniformly over all $N\geq 1$ and all $t\geq 0$. Since $\varphi$ has compact support, $B_t = -M_t$ for $t$ large enough. The assertion of the lemma then easily follows.
\end{proof}
Recall that $\ccM_{T_\ell(u)} \eta$ is the average of $\eta$ on the box $T_\ell(u)$ for any $u\in \{1,\ldots,2N\}$.
\begin{lemma}[Macroscopic inequalities]\label{Lemma:MacroIneq}
Let $\iota_N$ be a measure on $\{0,1\}^{2N}$ satisfying Assumption \ref{Assumption:IC}. For all $\varphi\in \cC^\infty_c([0,\infty)\times[0,1],\R_+)$, all $\delta > 0$ and all $c\in [0,1]$, we have $\lim_{\epsilon \downarrow 0} \varliminf_{N\rightarrow\infty} \bbP^N_{\iota_N}(\cJ^N(\varphi) \geq -\delta) = 1$ where
\begin{equation}\label{Eq:ClaimHydro}\begin{split}
\cJ^N(\varphi) &:= \int_0^\infty \bigg(\Big\langle \partial_s \varphi(s,\cdot) , \Big( \ccM_{T_{\epsilon N}(\cdot)}(\eta^N_s) - c \Big)^\pm \Big\rangle_N\\
&+ \Big\langle \partial_x \varphi(s,\cdot),h^\pm\Big(\ccM_{T_{\epsilon N}(\cdot)}(\eta^N_s),c\Big) \Big\rangle_N \\
&+ 2 \Big( (1-c)^\pm \varphi(s,0) + (0-c)^\pm \varphi(s,1) \Big)\bigg) ds\\
&+ \Big\langle \varphi(0,\cdot) , \Big( \ccM_{T_{\epsilon N}(\cdot)}(\eta^N_0) - c \Big)^\pm \Big\rangle_N\;.
\end{split}\end{equation}
\end{lemma}
\begin{proof}
Since at any time $s\geq 0$, $\zeta^N(s,\cdot)$ is distributed according to a product of Bernoulli measures with parameter $c$, we deduce that
\begin{equation*}
\lim_{\epsilon \downarrow 0} \lim_{N\rightarrow\infty} \tilde{\E}^N_{\iota_N,c}\Big[\frac1{2N} \sum_{u=1}^{2N} \Big| \ccM_{T_{\epsilon N}(u)}(\zeta^N_s) - c\Big|\Big] = 0\;.
\end{equation*}
and consequently, by Fubini's Theorem and stationarity, we have
\begin{equation*}
\lim_{\epsilon \downarrow 0} \lim_{N\rightarrow\infty} \tilde{\E}^N_{\iota_N,c}\Big[\int_0^t \frac1{2N} \sum_{u=1}^{2N} \Big| \ccM_{T_{\epsilon N}(u)}(\zeta^N_s) - c\Big| ds \Big] = 0\;.
\end{equation*}
Now we observe that for all $\epsilon >0$, we have $\tilde{\P}^N_{\iota_N,c}$ almost surely
\begin{equation*}
\Big\langle \varphi(0,\cdot) , \big( \eta^N_0(\cdot) - \zeta^N_0(\cdot) \big)^\pm \Big\rangle_N = \Big\langle \varphi(0,\cdot) , \ccM_{T_{\epsilon N}(\cdot)}\big( \eta^N_0 - \zeta^N_0 \big)^\pm \Big\rangle_N + \cO(\epsilon)\;.
\end{equation*}
Recall the coupling we chose for $(\eta^N_0(\cdot),\zeta^N_0(\cdot))$. Since $\tilde{\P}^N_{\iota_N,c}$ almost surely the number of sign changes $n(0)$ is bounded by some constant $C>0$ uniformly over all $N\geq 1$, we deduce using the previous identity that
\begin{equation*}
\Big\langle \varphi(0,\cdot) , \big( \eta^N_0(\cdot) - \zeta^N_0(\cdot) \big)^\pm \Big\rangle_N = \Big\langle \varphi(0,\cdot) , \big( \ccM_{T_{\epsilon N}(\cdot)}\eta^N_0 - \ccM_{T_{\epsilon N}(\cdot)}\zeta^N_0 \big)^\pm \Big\rangle_N + \cO(\epsilon)\;.
\end{equation*}
Therefore, by Lemma \ref{Lemma:MicroIneq}, we deduce that the statement of the lemma follows if we can show that for all $\delta > 0$
\begin{equation}\label{Eq:ClaimHydro2}\begin{split}
\varlimsup_{\epsilon \downarrow 0} \varlimsup_{N\rightarrow\infty} \tilde{\P}^N_{\iota_N,c}&\bigg( \int_0^t \frac1{2N} \sum_{u=1}^{2N} \Big| \ccM_{T_{\epsilon N}(u)} \big( \eta^N_s - \zeta^N_s \big)^\pm\\
&\qquad\qquad - \big( \ccM_{T_{\epsilon N}(u)}(\eta^N_s-\zeta^N_s) \big)^\pm \Big| ds > \delta \bigg) = 0\;,\\
\varlimsup_{\epsilon \downarrow 0} \varlimsup_{N\rightarrow\infty} \tilde{\P}^N_{\iota_N,c}&\bigg( \int_0^t \frac1{2N} \sum_{u=1}^{2N} \Big| \ccM_{T_{\epsilon N}(u)} H^\pm(\eta^N_s,\zeta^N_s) \\
&\qquad\qquad- h^\pm\Big( \ccM_{T_{\epsilon N}(u)}(\eta^N_s), \ccM_{T_{\epsilon N}(u)}(\zeta^N_s)\Big) \Big| ds > \delta \bigg) = 0\;.
\end{split}\end{equation}
We restrict ourselves to proving the second identity, since the first is simpler. Let $\cN^+_s$, resp. $\cN^-_s$, be the set of $u\in\{1,\ldots,2N\}$ such that $\eta_s \geq \zeta_s$, resp. $\zeta_s \geq \eta_s$, on the whole box $T_{\epsilon N}(u)$. By Lemma \ref{Lemma:Coupling}, $2N-\#\cN^+_s-\#\cN^-_s$ is of order $\epsilon N$ uniformly over all $s$, all $N\geq 1$ and all $\epsilon$. Therefore, we can neglect the contribution of all $u \notin \cN^+_s \cup \cN^-_s$. If we define $\Phi(\eta)=-2 \eta(1)(1-\eta(0))$ and if we let $\tilde{\Phi}(a)$ be as in (\ref{Eq:tildePhi}) below, then for all $u\in \cN^+_s$ we have 
\begin{equation*}
\ccM_{T_{\epsilon N}(u)} H^-(\eta^N_s,\zeta^N_s) - h^-\Big( \ccM_{T_{\epsilon N}(u)}(\eta^N_s), \ccM_{T_{\epsilon N}(u)}(\zeta^N_s)\Big) = 0\;,
\end{equation*}
as well as
\begin{align*}
{}&\ccM_{T_{\epsilon N}(u)} H^+(\eta^N_s,\zeta^N_s) - h^+\Big( \ccM_{T_{\epsilon N}(u)}(\eta^N_s), \ccM_{T_{\epsilon N}(u)}(\zeta^N_s)\Big)\\
= \;&\ccM_{T_{\epsilon N}(u)} \Phi(\eta^N_s) - \tilde{\Phi}\Big(\ccM_{T_{\epsilon N}(u)} \eta^N_s\Big)- \ccM_{T_{\epsilon N}(u)} \Phi(\zeta^N_s) + \tilde{\Phi}\Big(\ccM_{T_{\epsilon N}(u)} \zeta^N_s\Big)\;.
\end{align*}
Similar identities hold for every $u\in \cN^-_s$. We deduce that the second identity of (\ref{Eq:ClaimHydro2}) follows if we can show that for all $\delta >0$
\begin{align*}
\varlimsup_{\epsilon\downarrow 0} \varlimsup_{N\rightarrow\infty} \P^N_{\iota_N}\bigg( \int_0^t \frac1{2N} \sum_{u=1}^{2N} \Big| \ccM_{T_{\epsilon N}(u)} \Phi(\eta^N_s) - \tilde{\Phi}\Big(\ccM_{T_{\epsilon N}(u)}\eta_s\Big) \Big| ds > \delta \bigg) &= 0\;,\\
\varlimsup_{\epsilon\downarrow 0} \varlimsup_{N\rightarrow\infty} \tilde{\E}^N_{\iota_N,c}\bigg[ \int_0^t \frac1{2N} \sum_{u=1}^{2N} \Big| \ccM_{T_{\epsilon N}(u)}\Phi(\zeta^N_s) - \tilde{\Phi}\Big(\ccM_{T_{\epsilon N}(u)}\zeta^N_s\Big) \Big| ds \bigg] &= 0\;.
\end{align*}
The first convergence is ensured by Theorem \ref{Th:Replacement}, while the second follows from the stationarity of $\zeta^N$ and the Ergodic Theorem. This completes the proof of the lemma.
\end{proof}
\begin{proof}[Proof of Theorem \ref{Th:HydroProd}]
For any given $\epsilon > 0$, we have
\begin{equation}\label{Eq:AverageEta}\begin{split}
\ccM_{T_{2\epsilon N}(k)}(\eta_s) &= \frac1{2\epsilon} \rho^N\Big(s,\Big[\frac{k}{2N}-\epsilon,\frac{k}{2N}+\epsilon\Big]\Big)\\
&= \frac1{2\epsilon} \rho^N\Big(s,[x-\epsilon,x+\epsilon]\Big) + \cO(N^{-1})\;,
\end{split}\end{equation}
uniformly over all $k\in\{1,\ldots,2N-1\}$, all $x\in \Big[\frac{k}{2N},\frac{k+1}{2N}\Big]$ and all $N\geq 1$. Notice that the $\cO(N^{-1})$ depends on $\epsilon$. For all $\rho\in \bbD\big([0,\infty),\cM([0,1])\big)$, we set
\begin{align*}
V_c(\epsilon,\rho) &:= \int_0^\infty \bigg(\Big\langle \partial_s \varphi(s,\cdot) , \Big( \frac1{2\epsilon} \rho\Big(s,\Big[\cdot-\epsilon,\cdot+\epsilon\Big]\Big) - c \Big)^\pm \Big\rangle \\
&\quad+ \Big\langle \partial_x \varphi(s,\cdot),h^\pm\Big(\frac1{2\epsilon} \rho\Big(s,\Big[\cdot-\epsilon,\cdot+\epsilon\Big]\Big),c\Big) \Big\rangle \\
&\quad+ 2 \Big( (1-c)^\pm \varphi(s,0) + (0-c)^\pm \varphi(s,1) \Big)\bigg) ds\\
&\quad+\Big \langle \varphi(0,\cdot), \Big( \frac1{2\epsilon} \rho\Big(0,\Big[\cdot-\epsilon,\cdot+\epsilon\Big]\Big) - c \Big)^\pm \Big\rangle\;.
\end{align*}
Combining (\ref{Eq:AverageEta}), (\ref{Eq:ClaimHydro}) and the continuity of the maps $h^\pm(\cdot,c)$ and $(\cdot)^\pm$, we deduce that for any $\delta > 0$, we have
\begin{align*}
\lim_{\epsilon \downarrow 0} \varliminf_{N\rightarrow\infty} \bbP^N_{\iota_N}&\big( V_c(\epsilon,\rho^N) \geq -\delta\big) = 1\;.
\end{align*}
At this point, we observe that for all $\varphi\in\cC([0,1],\R_+)$ we have
\begin{equation*}
\langle \rho^N(t),\varphi \rangle \leq \frac1{2N} \sum_{k=1}^{2N} \varphi(k/2N)\;,
\end{equation*}
so that a simple argument ensures that for every limit point $\rho$ of $\rho^N$ and for all $t\geq 0$, the measure $\rho(t,dx)$ is absolutely continuous with respect to the Lebesgue measure, and its density is bounded by $1$. Therefore, any limit point is of the form $\rho(t,dx)=\eta(t,x)dx$ with $\eta\in L^\infty\big([0,\infty)\times(0,1)\big)$. Let $\P$ be the law of the limit of a convergent subsequence $\rho^{N_i}$. Since $\rho \mapsto V_c(\epsilon,\rho)$ is a $\P$-a.s.~continuous map on $\bbD\big([0,\infty),\cM([0,1])\big)$, we have for all $\epsilon > 0$
\begin{equation*}
\varlimsup_{i\rightarrow\infty} \bbP^{N_i}_{\iota_{N_i}}\big( V_c(\epsilon,\rho^{N_i}) \geq -\delta\big) \leq \bbP\big( V_c(\epsilon,\rho) \geq -\delta\big)\;.
\end{equation*}
For any $\rho$ of the form $\rho(t,dx) = \eta(t,x)dx$, we set
\begin{align*}
V_c(\rho) &:= \int_0^\infty \bigg(\Big\langle \partial_s \varphi(s,\cdot) , \Big( \eta(s,\cdot) - c \Big)^\pm \Big\rangle + \Big\langle \partial_x \varphi(s,\cdot),h^\pm\Big(\eta(s,\cdot),c\Big) \Big\rangle \\
&\quad+ 2 \Big( (1-c)^\pm \varphi(s,0) + (0-c)^\pm \varphi(s,1) \Big)\bigg) ds + \Big\langle \varphi(0,\cdot), (\eta_0-c)^\pm \Big\rangle \;,
\end{align*}
and we observe that by Lebesgue Differentiation Theorem, we have $\P$-a.s.~$V_c(\rho) = \lim_{\epsilon\downarrow 0} V_c(\epsilon,\rho)$. Therefore,
\begin{align*}
\P\big(V_c(\rho) \geq -\delta \big) &= \P\big(\lim_{\epsilon\downarrow 0} V_c(\epsilon,\rho) \geq -\delta \big)\\
&\geq \E\big[\varlimsup_{\epsilon\downarrow 0} \tun_{\{V_c(\epsilon,\rho) \geq -\delta/2\}}\big]\\
&\geq \varlimsup_{\epsilon\downarrow 0}\E\big[ \tun_{\{V_c(\epsilon,\rho) \geq -\delta/2\}}\big]\\
&\geq \varlimsup_{\epsilon\downarrow 0}\varlimsup_{i\rightarrow\infty} \bbP^{N_i}_{\iota_{N_i}}\big( V_c(\epsilon,\rho^{N_i}) \geq -\delta/2\big) =1\;,
\end{align*}
so the process $\big(\eta(t,x),t\geq 0, x\in (0,1)\big)$ under $\P$ coincides with the unique entropy solution of (\ref{PDEBurgersDirichlet}), thus concluding the proof.
\end{proof}

\section{KPZ fluctuations}\label{Section:KPZ}
\textit{This section is taken from~\cite{LabbeKPZ}, with more details at some places.}\\

For simplicity, we take $\sigma = 1$ in this whole section. The general case $\sigma > 0$ can be obtained \textit{mutatis mutandis}.\\

To prove Theorem \ref{Th:KPZ}, we follow the method of Bertini and Giacomin~\cite{BG97}. Due to our boundary conditions, there are two important steps that need some specific arguments: first the bound on the moments of the discrete process, see Proposition \ref{Prop:BoundMomentsKPZ}, second the bound on the error terms arising in the identification of the limit, see Proposition \ref{Prop:DelicateKPZ}. In order to simplify the notations, we will regularly use the microscopic variables $k,\ell \in\{1,\ldots,2N-1\}$ in rescaled quantities: for instance $h^N(t,\ell)$ stands for $h^N(t,x)$ with $x=(\ell-N)/(2N)^{2\alpha}$. As usual, we let $\cF_t,t\geq 0$ be the natural filtration associated with the process $(\xi^N(t),t\geq 0)$. We also introduce the notation
\begin{equation*}
\nabla^+ f(\ell):=f(\ell+1)-f(\ell)\;,\quad \nabla^- f(\ell) := f(\ell)-f(\ell-1)\;.
\end{equation*}

\subsection{The discrete Hopf-Cole transform}

The proof relies on the discrete Hopf-Cole transform, which was originally introduced by G\"artner~\cite{Gartner88} in the context of the WASEP on the line. Recall that
\begin{equation*}
h^N(t,x) := \gamma_N S\big(t(2N)^{4\alpha},N + x(2N)^{2\alpha}\big) - \lambda_N t \;,
\end{equation*}
and
\begin{equation*}
	\gamma_N := \frac12 \log \frac{p_N}{1-p_N}\;,\quad c_N := \frac{(2N)^{4\alpha}}{e^{\gamma_N} + e^{-\gamma_N}} \;,\quad  \lambda_N := c_N( e^{\gamma_N} -2 + e^{-\gamma_N})\;.
\end{equation*}
The discrete Hopf-Cole transform consists in setting $\xi^N(t,x) := \exp(-h^N(t,x))$. It allows to counterbalance the asymmetry of the drift of the WASEP: while the drift in the evolution equations of $h^N$ was given by a Laplacian plus a gradient, the drift in the evolution equations of $\xi^N$ is given by the Laplacian. Let us present the details here.\\
We decompose $\xi^N$ into a drift part $D^N(t)$ and a martingale part $M^N(t)$:
$$ d\xi^N(t,\ell) = D^N(t,\ell) dt + dM^N(t,\ell)\;,$$
and we aim at identifying the expressions of these two terms. Recall the dynamics of $S$: the process $S(t,\ell)$ makes a jump of size $2$ at rate $p_N$ if $\Delta S(t,\ell) =2$ and of size $-2$ at rate $1-p_N$ if $\Delta S(t,\ell) = -2$. Hence we have
\begin{equation*}
D^N(t,\ell) = \begin{cases}
\xi^N(t,\ell)\big((e^{2\gamma_N}-1)(1-p_N)(2N)^{4\alpha} + \lambda_N\big) &\mbox{ if } \Delta S(t,\ell) = -2\;,\\
\xi^N(t,\ell)\big((e^{-2\gamma_N}-1)p_N(2N)^{4\alpha} + \lambda_N\big) &\mbox{ if } \Delta S(t,\ell) = 2\;,\\
\xi^N(t,\ell) \lambda_N &\mbox{ if } \Delta S(t,\ell) = 0\;,\\
\end{cases}
\end{equation*}
as well as $d\langle M^N(\cdot,k),M^N(\cdot,\ell)\rangle_t = 0$ whenever $k\ne \ell$ and
\begin{equation*}
\frac{d\langle M^N(\cdot,\ell)\rangle_t}{dt} = \begin{cases}
\xi^N(t,\ell)^2 (e^{2\gamma_N}-1)^2 (1-p_N)(2N)^{4\alpha} &\mbox{ if } \Delta S(t,\ell) = -2\;,\\
\xi^N(t,\ell)^2 (e^{-2\gamma_N}-1)^2 p_N(2N)^{4\alpha} &\mbox{ if } \Delta S(t,\ell) = 2\;,\\
0 &\mbox{ if } \Delta S(t,\ell) = 0\;.\\
\end{cases}
\end{equation*}
On the other hand, we have the following array
\begin{equation*}
\begin{array}{l | l | l |}
& \Delta \xi^N(t,\ell)  & \nabla^+ \xi^N(t,\ell)\cdot \nabla^- \xi^N(t,\ell) \\\hline
\Delta S(t,\ell) = -2 & \xi^N(t,\ell)(2e^{\gamma_N}-2) & \xi^N(t,\ell)^2(e^{\gamma_N}-1)(1-e^{\gamma_N})\\
\Delta S(t,\ell) = 2 & \xi^N(t,\ell)(2e^{-\gamma_N}-2) & \xi^N(t,\ell)^2(e^{-\gamma_N}-1)(1-e^{-\gamma_N})\\
\Delta S(t,\ell) = 0 & \xi^N(t,\ell)(e^{\gamma_N}-2+e^{-\gamma_N}) & \xi^N(t,\ell)^2(e^{\gamma_N}-1)(1-e^{-\gamma_N})
\end{array}\end{equation*}
Putting everything together, we deduce that the stochastic differential equations solved by $\xi^N$ are given by
\begin{align}\label{Eq:SDEKPZ}
	\begin{cases}
	d\xi^N(t,\ell) = c_N \Delta \xi^N(t,\ell)dt + dM^N(t,\ell)\;,\qquad  \ell\in\{1,\ldots,2N-1\}\;,\\
	\xi^N(t,0) = \xi^N(t,2N) = e^{\lambda_N t}\;,\\
	\xi^N(0,\cdot) = e^{-h^N(0,\cdot)} \;,\end{cases}
\end{align}
where the bracket of $M^N$ is given by $d\langle M^N(\cdot,k),M^N(\cdot,\ell)\rangle_t = 0$ whenever $k\ne \ell$, and
\begin{equation}\label{Eq:Bracket}\begin{split}
	d \langle M^N(\cdot,k)\rangle_t &= \lambda_N\Big(\xi^N(t,k)\Delta\xi^N(t,k)+2\xi^N(t,k)^2 \Big)dt\\
	&\quad- (2N)^{4\alpha} \nabla^+\xi^N(t,k) \nabla^-\xi^N(t,k)dt\;.
\end{split}
\end{equation}
Observe that
\begin{equation*}
\big| d \langle M^N(\cdot,k)\rangle_t \big| \lesssim \xi^N(t,k)^2 (2N)^{2\alpha} dt\;,
\end{equation*}
uniformly over all $t\geq 0$, all $k$ and all $N\geq 1$.\\

The goal of the next paragraph is to express the process $\xi^N$ as the mild solution of \eqref{Eq:SDEKPZ} and to split this expression into two terms: one coming from the initial condition and another one from the stochastic oscillations. We let $p^N_t(k,\ell)$ be the discrete heat kernel on $\{0,\ldots,2N\}$ sped up by $2c_N$ and endowed with homogeneous Dirichlet boundary conditions:
\begin{align*}
	\begin{cases}
	\partial_t p^N_t(k,\ell) = c_N \Delta p^N_t(k,\ell)\;,\\
	p^N_0(k,\ell) = \delta_k(\ell)\;,\\
	p^N_t(k,0) = p^N_t(k,2N) = 0\;,\end{cases}
\end{align*}
For any $t\ge 0$, we introduce the martingale $[0,t]\ni r \mapsto N^t_r(\ell)$ by setting
\begin{equation}\label{Eq:DefNt}
	N^t_r(\ell) = \int_0^r \sum_{k=1}^{2N-1} p^N_{t-s}(k,\ell) dM^N(s,k)\;.
\end{equation}
We also define the process $I^N$ as the solution of
\begin{align*}
	\begin{cases}\partial_t I^N(t,\ell) = c_N \Delta I^N(t,\ell)\;,\\
	I^N(t,0) = I^N(t,2N) = e^{\lambda_N t}\;,\\
	I^N(0,\ell) = \xi^N(0,\ell)\;.\end{cases}
\end{align*}
Then, standard arguments ensure that
\begin{equation}\label{Eq:DiscreteHC}
	\xi^N(t,\ell) = I^N(t,\ell) + N^t_t(\ell)\;,\quad \forall \ell\;.
\end{equation}

\subsection{Preliminary bounds}

The hydrodynamic limit of Theorem \ref{Th:Hydro}, upon discrete Hopf-Cole transform, is given by
$$ 1 \vee \exp\Big(\lambda_N t - \gamma_N\big(\ell\wedge (2N-\ell)\big)\Big)\;.$$
At any time $t$, the set of points $\ell$ for which this expression is equal to $1$ coincides with the window
$$ B_0^N(t) := [\frac{\lambda_N}{\gamma_N}t,2N-\frac{\lambda_N}{\gamma_N}t] \subset[0,2N]\;.$$
Within this window, the density of particles is approximately $1/2$. On the left of this window, the density is approximately $1$, and on the right it is approximately $0$. For technical reasons, it is convenient to introduce an $\epsilon$-approximation of this window by setting:
\begin{equation*}
B^N_\epsilon(t) := \Big[\frac{\lambda_N}{\gamma_N} t + \epsilon N, 2N-\frac{\lambda_N}{\gamma_N} t - \epsilon N \Big]\;,\quad t\in [0,T)\;.
\end{equation*}

The term $I^N$ coming from the initial condition remains close to the hydrodynamic limit, while the fluctuations are given by the martingale term. For convenience, we set
\begin{equation*}
	b^N(t,\ell) := 2 + \exp\Big(\lambda_N t - \gamma_N \big(\ell \wedge (2N - \ell)\big)\Big)\;.
\end{equation*}
and we introduce the discrete heat kernel on the whole line $\Z$:
\begin{align*}
	\begin{cases}\partial_t \bar{p}^N_t(\ell) = c_N \Delta \bar{p}^N_t(\ell)\;,\\
	\bar{p}^N_0(\ell) = \delta_0(\ell)\;.\end{cases}
\end{align*}

\begin{proposition}\label{Prop:IC}
Let $K$ be a compact subset of $[0,T)$ and fix $\epsilon >0$. Uniformly over all $t\in K$, we have
\begin{enumerate}
\item $I^N(t,\ell) \lesssim b^N(t,\ell)$ for all $\ell\in\{1,\ldots,2N\}$,
\item $|\nabla^\pm I^N(t,\ell)| \lesssim t^{-\frac12} N^{-3\alpha}$ uniformly over all $\ell \in B^N_\epsilon(t)$,
\item $|I^N(t,\ell) - I^N(t',\ell)| \lesssim N^{-\alpha}$ uniformly over all $\ell\in B^N_\epsilon(t')$ and all $t<t' \in K$,
\item $|I^N(t,\ell) - I^N(t,\ell')| \lesssim N^{-\alpha}$ uniformly over all $\ell,\ell' \in B^N_\epsilon(t)$.
\end{enumerate}
\end{proposition}
\begin{proof}
It will be convenient to split $I^N$ into two terms
\begin{equation}\label{Eq:DecompoIN}
I^N(t,\ell) = \xi^{N,\circ}(t,\ell) + \xi^{N,\times}(t,\ell)\;.
\end{equation}
where
\begin{align*}
	\begin{cases}\partial_t\xi^{N,\circ}(t,\ell) = c_N \Delta \xi^{N,\circ}(t,\ell)\;,\\
	\xi^{N,\circ}(t,0) = \xi^{N,\circ}(t,2N) = e^{\lambda_N t}\;,\\
	\xi^{N,\circ}(0,\ell) = 1\;,\end{cases}
	\mbox{ and }
	\begin{cases}\partial_t\xi^{N,\times}(t,\ell) = c_N \Delta \xi^{N,\times}(t,\ell)\;,\\
	\xi^{N,\times}(t,0) = \xi^{N,\times}(t,2N) = 0\;,\\
	\xi^{N,\times}(0,\ell) = 1\;.\end{cases}
\end{align*}
Notice that $\xi^{N,\circ}$ is an approximation of $I^N$ where the initial condition is set to $1$ for all $\ell$, while $\xi^{N,\times}$ is the error made under this approximation. Observe that
$$\xi^{N,\times}(t,\ell) = \sum_{k=1}^{2N-1} p^N_t(k,\ell)\big(\xi^N(0,k)-1\big)\;.$$

The four bounds of the statement will be obtained separately for $\xi^{N,\circ}$ and $\xi^{N,\times}$.\\
Since our initial condition is flat, it is immediate to check that
\begin{equation*}
\big|\xi^{N,\times}(t,\ell)\big| \lesssim N^{-\alpha} \ll b^N(t,\ell)\;,
\end{equation*}
which immediately yields the bounds 1., 3. and 4. of the statement for $\xi^{N,\times}$. Recall that $\bar{p}^N$ is the discrete heat kernel on $\Z$. Using Lemmas \ref{Lemma:ExpoDecay} and \ref{Lemma:DecaySeriesKernel}, we get
\begin{equation*}
\nabla^\pm \xi^{N,\times}(t,\ell) = \sum_{k\in B^N_{\epsilon/2}(0)}\!\! \nabla^+ \bar{p}^N_t(\ell-k)(\xi^{N}(0,k)-1)+ \cO(N^{1-\alpha} e^{-\delta N^{2\alpha}})\;,
\end{equation*}
uniformly over all $\ell\in B^N_\epsilon(t)$, all $t\in K$ and all $N\geq 1$. Then, we write
\begin{equation*}
\sum_{k\in B^N_{\epsilon/2}(0)} |\nabla^+ \bar{p}^N_t(\ell-k)| = -\bar{p}^N_t(\ell-i_- -1) + 2\bar{p}^N_t(0) - \bar{p}^N_t(\ell-i_+)\;,
\end{equation*}
where $i_{\pm}$ are the first and last integers in $B^N_{\epsilon/2}(0)$. Using Lemma \ref{Lemma:BoundHeatKernelZ} and our choice of initial condition, we deduce that
\begin{equation*}
\Big|\sum_{k\in B^N_{\epsilon/2}(0)} \nabla^+ \bar{p}^N_t(\ell-k)(\xi^{N}(0,k)-1)\Big| \lesssim 1 \wedge \frac{1}{\sqrt{t}(2N)^{3\alpha}}\;,
\end{equation*}
uniformly over the same set of parameters. The same applies to $\nabla^-$, thus concluding the proof of the bound 2. for $\xi^{N,\times}$.\\
To establish the required bounds on $\xi^{N,\circ}$, we first show that there exists $\delta > 0$ such that
\begin{equation}\label{Eq:BdExpo}
|\xi^{N,\circ}(t,\ell) - 1| \lesssim \exp(-\delta N^{2\alpha})\;,
\end{equation}
uniformly over all $t\in K$, all $\ell\in B^N_\epsilon(t)$ and all $N\geq 1$. Since
\begin{equation*}
\xi^{N,\circ}(t,\ell) = 1 + \lambda_N \int_0^t \Big(1-\sum_{k=1}^{2N-1} p^N_{t-s}(k,\ell) \Big) e^{\lambda_N s} ds\;,
\end{equation*}
the bound will be ensured if we are able to show that there exists $\delta > 0$ such that
\begin{equation}\label{Eq:BoundTailDiscreteKernel}
\Big(1-\sum_{k=1}^{2N-1} p^N_{t-s}(k,\ell) \Big) e^{\lambda_N s} \lesssim e^{-\delta N^{2\alpha}}\;,
\end{equation}
uniformly over all $s\in[0,t]$, all $t\in K$ and all $\ell \in B^N_\epsilon(t)$. The proof of this estimate on the heat kernel is provided in Appendix \ref{Appendix:Kernel}. This yields (\ref{Eq:BdExpo}), and therefore concludes the proof of the bounds 2., 3. and 4. of the statement for $\xi^{N,\circ}$.\\
Using the estimate on $\xi^{N,\circ}(t,N)-1$ obtained above, we deduce that for $N$ large enough, $b^N$ solves
\begin{align*}
	\begin{cases}
	\partial_t b^N(t,\ell) = c_N \Delta b^N(t,\ell)\;,\quad \ell\in \{1,\ldots,N-1\}\;,\\
	b^N(t,0) \geq \xi^{N,\circ}(t,0)\;,\quad b^N(t,N) \geq \xi^{N,\circ}(t,N)\;,\\
	b^N(0,k) \geq \xi^{N,\circ}(0,k)\;.\end{cases}
\end{align*}
By the maximum principle, one deduces that $b^N(t,\ell) \geq \xi^{N,\circ}(t,\ell)$ for all $t\in K$ and all $\ell\in\{0,\ldots,N\}$. By symmetry, this inequality also holds for $\ell \in \{N,\ldots,2N\}$.
\end{proof}

To alleviate the notation, we define
\begin{equation}\label{Def:qN}
	q^N_{s,t}(k,\ell) = p^N_{t-s}(k,\ell) b^N(s,k)\;.
\end{equation}
We now have all the ingredients at hand to bound the moments of $\xi^N$.

\begin{proposition}\label{Prop:BoundMomentsKPZ}
For all $n\geq 1$ and all compact set $K \subset [0,T)$, we have
\begin{equation*}
	\sup_{N\geq 1} \sup_{\ell\in\{1,\ldots,2N-1\}} \sup_{t\in K} \E\Big[ \Big(\frac{\xi^N(t,\ell)}{b^N(t,\ell)}\Big)^n \Big] < \infty\;.
\end{equation*}
\end{proposition}
Since $b^N$ is of order $1$ inside $B^N_\epsilon(t)$, this ensures that the moments are themselves of order $1$ in these windows.

\begin{proof}
We fix the compact set $K$ until the end of the proof. Using the expression (\ref{Eq:DiscreteHC}) and Proposition \ref{Prop:IC}, we deduce that
\begin{equation}\label{Eq:ExpresMoments}
	\E\bigg[ \bigg(\frac{\xi^N(t,\ell)}{b^N(t,\ell)}\bigg)^{2n} \bigg]^\frac{1}{2n} \lesssim 1 + \E\bigg[ \bigg(\frac{N^t_t(\ell)}{b^N(t,\ell)}\bigg)^{2n} \bigg]^\frac{1}{2n}\;.
\end{equation}
We set $D^t_r := \big[ N^t_\cdot \big]_r - \langle N^t_\cdot \rangle_r$ and we refer to Appendix \ref{Appendix:Mgale} for the notations. By the BDG inequality (\ref{Eq:BDG2}), we obtain
\begin{equation}\label{Eq:BDGMoments}
	\E \Big[ \big(N^t_t(\ell)\big)^{2n}\Big] \lesssim \E \Big[ \big\langle N^t_\cdot(\ell)\big\rangle_t^n\Big] + \E \Big[ \big[ D^t_\cdot(\ell)\big]_t^\frac{n}{2}\Big]\;,
\end{equation}
uniformly over all $\ell \in \{1,\ldots,2N-1\}$, all $t \geq 0$, and all $N\geq 1$. Let
\begin{equation*}
	g^N_n(s) := \sup_{k\in\{1,\ldots,2N-1\}} \E\bigg[ \Big(\frac{\xi^N(s,k)}{b^N(s,k)}\Big)^{2n} \bigg]\;.
\end{equation*}
We claim that
\begin{align}
	\E \Big[ \big\langle N^t_\cdot(\ell)\big\rangle_t^n\Big] &\lesssim b^N(t,\ell)^{2n}\int_0^t  \frac{g_n^N(s)}{\sqrt{t-s}} ds\;,\label{Eq:BoundBracket}\\
	\E \Big[ \big[ D^t_\cdot(\ell)\big]_t^\frac{n}{2}\Big] &\lesssim b^N(t,\ell)^{2n}\Big(1 + \int_0^t  \frac{g_n^N(s)}{\sqrt{t-s}} ds\Big)\;,\quad\label{Eq:BoundQuadVar}
\end{align}
uniformly over all $\ell\in\{1,\ldots,2N-1\}$, all $N\geq 1$ and all $t\in K$. We postpone the proof of these two bounds. Combining these two bounds with (\ref{Eq:ExpresMoments}) and (\ref{Eq:BDGMoments}), we obtain the following closed inequality
\begin{equation*}
	g^N_n(t) \lesssim 1 + \int_0^t \frac{g^N_n(s)}{\sqrt{t-s}} ds\;,
\end{equation*}
uniformly over all $N\geq 1$ and all $t\in K$. By a generalised Gr\"onwall's inequality, see for instance~\cite[Lemma 6 p.33]{Haraux}, we deduce that $g^N_n(t)$ is uniformly bounded over all $N\geq 1$ and all $t\in K$.\\
We are left with establishing (\ref{Eq:BoundBracket}) and (\ref{Eq:BoundQuadVar}). Using (\ref{Eq:Bracket}), we obtain the almost sure bound
\begin{equation*}
	\big\langle N^t_\cdot(\ell)\big\rangle_t \lesssim (2N)^{2\alpha}\int_0^t \sum_{k} p^N_{t-s}(k,\ell)^2 \xi^N(s,k)^2 ds\;,
\end{equation*}
uniformly over all $N\geq 1$, $t\geq 0$ and $\ell\in\{1,\ldots,2N-1\}$. Recall the function $q^N$ from (\ref{Def:qN}). Using H\"older's inequality at the second line, we find
\begin{align*}
	\E \bigg[ \Big(\frac{\big\langle N^t_\cdot(\ell)\big\rangle_t}{b^N(t,\ell)^2}\Big)^n\bigg] &\lesssim \int\limits_{s_1,\ldots,s_n=0}^t \sum_{k_1,\ldots,k_n} \E\bigg[\prod_{i=1}^n (2N)^{2\alpha} \Big(\frac{q^N_{s_i,t}(k_i,\ell)}{b^N(t,\ell)}\Big)^2 \Big(\frac{\xi^N(s_i,k_i)}{b^N(s_i,k_i)}\Big)^2 \bigg] ds_i\\
	&\lesssim \int\limits_{s_1,\ldots,s_n=0}^t \sum_{k_1,\ldots,k_n} \prod_{i=1}^n (2N)^{2\alpha} \Big(\frac{q^N_{s_i,t}(k_i,\ell)}{b^N(t,\ell)}\Big)^2  g_n^N(s_i)^{\frac{1}{n}} ds_i\\
	&\lesssim \Big(\int\limits_{s=0}^t \sum_{k} (2N)^{2\alpha} \Big(\frac{q^N_{s,t}(k,\ell)}{b^N(t,\ell)}\Big)^2  g_n^N(s)^{\frac{1}{n}} ds \Big)^n\;.
\end{align*}
By the first inequality of Lemma \ref{Lemma:HeatKernel} we bound $\sum_k q^N_{s,t}(k,\ell) / b^N(t,\ell)$ by a term of order $1$, and by the second inequality of the same lemma we bound
$$ (2N)^{2\alpha} \sup_k \frac{q^N_{s,t}(k,\ell)}{b^N(t,\ell)}\;,$$
by a term of order $1/\sqrt{t-s}$. Using Jensen's inequality at the second step, we thus get
\begin{align*}
	\E \bigg[ \Big(\frac{\big\langle N^t_\cdot(\ell)\big\rangle_t}{b^N(t,\ell)^2}\Big)^n\bigg] &\lesssim \Big(\int\limits_{s=0}^t \frac{g_n^N(s)^{\frac{1}{n}}}{\sqrt{t-s}} ds \Big)^n\lesssim \int_{0}^t \frac{g_n(s)}{\sqrt{t-s}} ds\;,
\end{align*}
uniformly over all $N\geq 1$, all $t\in K$ and all $\ell\in\{1,\ldots,2N-1\}$, thus yielding (\ref{Eq:BoundBracket}).\\
We turn to the quadratic variation. Let $J_k$ be the set of jump times of $\xi^N(\cdot,k)$. We start with the following simple bound
\begin{align*}
	\big[ D^t_\cdot(\ell)\big]_t &= \sum_{\tau \leq t} \sum_{k} p^N_{t-\tau}(k,\ell)^4\big(\xi^N(\tau,k)-\xi^N(\tau-,k)\big)^4\\
	&\lesssim \gamma_N^4 \sum_{k}\sum_{\tau \leq t; \tau \in J_k} q^N_{\tau,t}(k,\ell)^4\Big(\frac{\xi^N(\tau,k)}{b^N(\tau,k)}\Big)^4\;,
\end{align*}
uniformly over all $N\geq 1$, all $t\geq 0$ and all $\ell\in\{1,\ldots,2N-1\}$. We set $t_i:= i(2N)^{-4\alpha}$ for all $i\in\N$ and we let $\cI_i:=[t_i,t_{i+1})$. Then, by Minkowski's inequality we have
\begin{equation*}
	\E \Big[ \big[ D^t_\cdot(\ell)\big]_t^{\frac{n}{2}}\Big]^\frac{2}{n} \lesssim \gamma_N^4\!\!\!\sum_{i=0}^{\lfloor t(2N)^{4\alpha}\rfloor}\!\!\!\sum_{k}\sup_{s\in\cI_i, s<t} q^N_{s,t}(k,\ell)^4\E \Big[ \Big( \sum_{\tau \in \cI_i \cap J_k} \Big(\frac{\xi^N(\tau,k)}{b^N(\tau,k)}\Big)^4\Big)^\frac{n}{2}\Big]^{\frac{2}{n}}\;,
\end{equation*}
Let $Q(k,r,s)$ be the number of jumps of the process $\xi^N(\cdot,k)$ on the time interval $[r,s]$. We have the following almost sure bound
\begin{equation*}
\xi^N(\tau,k) \leq \xi^N(s,k) e^{2(2N)^{-4\alpha} \lambda_N + 2\gamma_N Q(k,s,t_{i+1})}\;,
\end{equation*}
uniformly over all $s\in \cI_{i-1}$, all $\tau \in \cI_i$, all $k\in\{1,\ldots,2N-1\}$ and all $i\geq 1$. Consequently we get
\begin{equation*}
	\sum_{\tau \in \cI_i\cap J_k}\Big(\frac{\xi^N(\tau,k)}{b^N(\tau,k)}\Big)^4 \lesssim (2N)^{4\alpha}\int_{t_{i-1}}^{t_i} \Big(\frac{\xi^N(s,k)}{b^N(s,k)}\Big)^4 Q(k,s,t_{i+1})\, e^{8\gamma_N Q(k,s,t_{i+1})} ds\;,
\end{equation*}
uniformly over all $N\geq 1$, all $i\geq 1$ and all $k\in\{1,\ldots,2N-1\}$. Since $(Q(k,s,t), t\geq s)$ is, conditionally given $\cF_s$, stochastically bounded by a Poisson process with rate $(2N)^{4\alpha}$, we deduce that there exists $C >0$ such that almost surely
\begin{equation*}
	\sup_{N\geq 1} \sup_{i\geq 1} \sup_{s\in\cI_{i-1}} \E\Big[Q(k,s,t_{i+1})^{\frac{n}{2}}e^{4n\gamma_N Q(k,s,t_{i+1})}\,\Big|\,\cF_s\Big] < C\;.
\end{equation*}
Then, we get
\begin{align*}
	{}&\E \Big[ \Big( \sum_{\tau \in \cI_i\cap J_k} \Big(\frac{\xi^N(\tau,k)}{b^N(\tau,k)}\Big)^4\Big)^\frac{n}{2}\Big]^{\frac{2}{n}}\\
	&\lesssim (2N)^{4\alpha}\int_{t_{i-1}}^{t_i} \E \Big[ \Big(\Big(\frac{\xi^N(s,k)}{b^N(s,k)}\Big)^4 Q(k,s,t_{i+1})\, e^{8\gamma_N Q(k,s,t_{i+1})}\Big)^\frac{n}{2}\Big]^{\frac{2}{n}} ds\\
	&\lesssim C (2N)^{4\alpha}\int_{t_{i-1}}^{t_i} g^N_n(s)^{\frac{2}{n}} ds\;,
\end{align*}
uniformly over all $N\geq 1$, all $i\geq 1$ and all $k$. On the other hand, when $i=0$ we have the following bound
\begin{equation*}
	\E \Big[ \Big( \sum_{\tau \in \cI_0\cap J_k} \Big(\frac{\xi^N(\tau,k)}{b^N(\tau,k)}\Big)^4\Big)^\frac{n}{2}\Big]^{\frac{2}{n}} \lesssim \Big(\frac{\xi^N(0,k)}{b^N(0,k)}\Big)^4 \E\Big[Q(k,0,t_{1})^{\frac{n}{2}}e^{2n\gamma_N Q(k,0,t_{1})}\Big]\lesssim 1\;,
\end{equation*}
uniformly over all $k$ and all $N\geq 1$.\\
Observe that
\begin{equation*}
p^N_{t-s}(k,\ell) = e^{-2c_N(t-s)} \sum_{n\geq 0} \frac{(2c_N(t-s))^n}{n!} \mathbf{p}_n(k,\ell)\;,
\end{equation*}
where $\mathbf{p}_n(k,\ell)$ is the probability that a discrete-time random walk, killed upon hitting $0$ and $2N$ and started from $k$, reaches $\ell$ after $n$ steps. Therefore, we easily deduce that $\sup_{s\in\cI_i, s<t}q^N_{s,t}(\ell,k) \lesssim q^N_{t_i,t}(\ell,k)$. Using the two bounds of Lemma \ref{Lemma:HeatKernel}, we get
\begin{align*}
	\sum_{k}\sup_{s\in\cI_i,s<t} q^N_{s,t}(k,\ell)^4 &\lesssim \sum_{k} q^N_{t_i,t}(k,\ell)^4 \lesssim \sup_{k} q^N_{t_i,t}(k,\ell)^3 \sum_{k} q^N_{t_i,t}(k,\ell)\\
	&\lesssim b^N(t,\ell)^4 \Big(1\wedge \frac{1}{\sqrt{t-t_i}\,(2N)^{2\alpha}}\Big)\;,
\end{align*}
	uniformly over all $N\geq 1$ and $i\geq 0$. Putting everything together, we obtain
	\begin{align*}
		\E \Big[ \big[ D^t_\cdot(k)\big]_t^{\frac{n}{2}}\Big]^\frac{2}{n} &\lesssim b^N(t,\ell)^4\Big( 1 + \int_0^t \frac{g^N_n(s)^\frac{2}{n}}{\sqrt{t-s}\,(2N)^{2\alpha}} ds \Big)\;,
	\end{align*}
	and the required bound follows by Jensen's inequality, thus concluding the proof.
\end{proof}

\subsection{Tightness}

The following two lemmas control the moments of the space and time increments of the process.

\begin{lemma}\label{Lemma:IncrSpaceKPZ}
Fix $\epsilon > 0$, $\beta\in (0,1/2)$ and a compact set $K \subset [0,T)$. For any $n\geq 1$, we have
\begin{equation*}
	\E\Big[ \big|\xi^N(t,\ell')-\xi^N(t,\ell)\big|^{2n} \Big]^{\frac{1}{2n}} \lesssim \Big|\frac{\ell-\ell'}{(2N)^{2\alpha}} \Big|^{\beta}\;,
\end{equation*}
uniformly over all $t\in K$, all $\ell,\ell'\in B_\epsilon^N(t)$ and all $N\geq 1$.
\end{lemma}
\begin{proof}
The expression (\ref{Eq:DiscreteHC}) yields two terms for $\xi^N(t,\ell)-\xi^N(t,\ell')$. By Proposition \ref{Prop:IC}, the first term can be bounded by a term of order $N^{-\alpha}$ which is negligible compared to $(|\ell-\ell'|/(2N)^{2\alpha})^{\beta}$ whenever $\ell\ne\ell'$. Therefore, to complete the proof of the lemma, we only need to establish the appropriate bound for the $2n$-th moment of $R^t_t(\ell,\ell')$, where we have introduced the martingale
\begin{equation*}
R^t_s(\ell,\ell') := \int_0^r \sum_{k=1}^{2N-1} \big(p^N_{t-s}(k,\ell)-p^N_{t-s}(k,\ell')\big) dM^N(s,k)\;,\quad r\in [0,t]\;.
\end{equation*}
Let $D^t_s(\ell,\ell')=[R^t_\cdot(\ell,\ell')]_s-\langle R^t_\cdot(\ell,\ell')\rangle_s$. We claim that we have
\begin{equation}\label{Eq:BoundBracketSpace}
	\E \Big[ \big\langle R^t_\cdot(\ell,\ell')\big\rangle_t^n\Big]^\frac{1}{n} \lesssim \Big|\frac{\ell-\ell'}{(2N)^{2\alpha}} \Big|^{2\beta}\;,\quad	\E \Big[ \big[ D^t_\cdot(\ell,\ell')\big]_t^\frac{n}{2}\Big]^\frac{2}{n} \lesssim (2N)^{-4\alpha}\;,		
	\end{equation}
uniformly over all $t\in K$, all $\ell,\ell' \in B_\epsilon^N(t)$ and all $N\geq 1$. These two inequalities, together with the BDG inequality (\ref{Eq:BDG2}) yield the desired bound on the $2n$-th moment of $R^t_s(\ell,\ell')$, thus concluding the proof. We are left with the proof of these inequalities. As in the proof of Proposition \ref{Prop:BoundMomentsKPZ}, we observe that
\begin{align*}
	\E \Big[ \big[ D^t_\cdot(\ell,\ell')\big]_t^\frac{n}{2}\Big]^\frac{2}{n} &\lesssim \gamma_N^4 \sum_{i=0}^{\lfloor t(2N)^{4\alpha}\rfloor} \sum_{k} \sup_{s\in\cI_i,s<t} \big(q^N_{s,t}(k,\ell)-q^N_{s,t}(k,\ell') \big)^4\\
	&\qquad\times\E \Big[ \Big( \sum_{\tau \in \cI_i} \Big(\frac{\xi^N(\tau,k)}{b^N(\tau,k)}\Big)^4\Big)^\frac{n}{2}\Big]^{\frac{2}{n}}\;.
\end{align*}
The arguments in that proof ensure that the expectation in the r.h.s.~is uniformly bounded over all $i$, all $k$ and all $N\geq 1$. On the other hand, $\sup_{s\in\cI_i} \big(q^N_{s,t}(k,\ell)-q^N_{s,t}(k,\ell') \big)^4 \lesssim q^N_{t_i,t}(k,\ell)^4+q^N_{t_i,t}(k,\ell')^4$, so that Lemma \ref{Lemma:HeatKernel} immediately yields
\begin{equation*}
	\sum_{k} \sup_{s\in\cI_i} \big(q^N_{s,t}(k,\ell)-q^N_{s,t}(k,\ell') \big)^4 \lesssim 1\wedge \Big(\frac{1}{\sqrt{t-t_i}(2N)^{2\alpha}}\Big)^{3}\;,
\end{equation*}
since $b^N(t,\ell)$ is of order $1$ in $B^N_\epsilon(t)$. Hence, we get
\begin{align*}
	\E \Big[ \big[ D^t_\cdot(\ell,\ell')\big]_t^\frac{n}{2}\Big]^\frac{2}{n} &\lesssim \gamma_N^4 \sum_{i=0}^{\lfloor t(2N)^{4\alpha}\rfloor}1\wedge \Big(\frac{1}{\sqrt{t-t_i}(2N)^{2\alpha}}\Big)^{3}\lesssim \gamma_N^4\;,
\end{align*}
uniformly over all $t\in K$, all $\ell,\ell' \in B_\epsilon^N(t)$ and all $N\geq 1$. This yields the second bound of (\ref{Eq:BoundBracketSpace}). Regarding the first bound, we notice that we only have to consider the cases where $\ell\ne \ell'$. Then, we have the following almost sure bound
\begin{equation*}
	\big\langle R^t_\cdot(\ell,\ell')\big\rangle_t \lesssim \int_0^t \sum_{k=1}^{2N-1} \big(p^N_{t-s}(k,\ell)-p^N_{t-s}(k,\ell')\big)^2(2N)^{2\alpha} \xi^N(s,k)^2 ds\;.
\end{equation*}
We argue differently according as $k$ belongs to $B^N_{\epsilon/2}(s)$ or not. Using Lemma \ref{Lemma:ExpoDecay}, we deduce that
\begin{align*}
\int_0^t \sum_{k\notin B^N_{\epsilon/2}(s)} \big(q^N_{s,t}(k,\ell)-q^N_{s,t}(k,\ell')\big)^2(2N)^{2\alpha} \E\Big[ \Big(\frac{\xi^N(s,k)}{b^N(s,k)}\Big)^{2n}\Big]^{\frac1{n}} ds\lesssim N^{1+2\alpha} e^{-\delta N^{2\alpha}}\;,
\end{align*}
uniformly over all $\ell \in B_\epsilon(t)$, all $t\in K$ and all $N\geq 1$. This yields a bound of the desired order whenever $\ell\ne \ell'$. On the other hand, using Lemma \ref{Lemma:BoundHeatKernel} the contribution of the remaining $k$'s can be bounded as follows
\begin{align*}
	{}&\int_0^t \sum_{k\in B_{\epsilon/2}(s)} \big(p^N_{t-s}(k,\ell)-p^N_{t-s}(k,\ell')\big)^2(2N)^{2\alpha} \E\big[ \xi^N(s,k)^{2n}\big]^{\frac1{n}} ds\\
	&\lesssim \int_0^t \frac{ds}{(t-s)^{\frac{1}{2}}}\,\Big|\frac{\ell-\ell'}{(2N)^{2\alpha}} \Big|^{2\beta}\;,
\end{align*}
since $b^N(s,k)$ is of order $1$ in $B^N_{\epsilon/2}(s)$, thus concluding the proof.
\end{proof}

\begin{lemma}\label{Lemma:IncrTimeKPZ}
Fix $\epsilon > 0$, $\beta\in(0,1/4)$ and a compact set $K \subset [0,T)$. For any $n\geq 1$, we have
\begin{equation*}
	\E\Big[ \big|\xi^N(t',\ell)-\xi^N(t,\ell)\big|^{2n} \Big]^{\frac{1}{2n}} \lesssim |t'-t|^\beta + \frac{1}{(2N)^{\alpha}}\;,
\end{equation*}
uniformly over all $N\geq 1$, all $t<t'\in K$ and all $\ell\in B_\epsilon^N(t')$.
\end{lemma}
\begin{proof}
Using (\ref{Eq:DiscreteHC}), we can write $\xi^N(t',\ell)-\xi^N(t,\ell)$ as the sum of two terms. Proposition \ref{Prop:IC} ensures that the first term is bounded by a term of order $N^{-\alpha}$ as required. Therefore, we only need to find the appropriate bound for the $2n$-th moment of $N^{t'}_{t'}(\ell)-N^t_t(\ell)$. To that end, we bound separately the $2n$-th moments of $A^{t,t'}_\delta$ and $B^{t,t'}_t$, where we have set $\delta = t'-t$ and introduced the martingales $A^{t,t'}_u := N^{t'}_{t+u}(\ell)-N^{t'}_{t}(\ell)$, $u\leq \delta$ and $B^{t,t'}_s := N^{t'}_{s}(\ell)-N^{t}_{s}(\ell)$, $s\leq t$.\\
Recall that $b^N(t',\ell)$ is of order $1$ in $B^N_\epsilon(t')$. Since
\begin{equation*}
	A^{t,t'}_u(\ell) = \int_t^{t+u} \sum_k p^N_{t'-r}(k,\ell) dM^N(r,k)\;,
\end{equation*}
a simple computation, using Proposition \ref{Prop:BoundMomentsKPZ} and Lemma \ref{Lemma:HeatKernel}, shows that
\begin{align*}
	\E \Big[ \big\langle A^{t,t'}_\cdot(\ell)\big\rangle_\delta^n\Big]^\frac{1}{n} &\lesssim (2N)^{2\alpha}\int_t^{t'} \sum_k q^N_{r,t'}(k,\ell)^2 \E \Big[ \Big(\frac{\xi^N(r,k)}{b^N(r,k)}\Big)^{2n}\Big]^\frac{1}{n} dr\\
	&\lesssim (2N)^{2\alpha}\int_t^{t'} \frac{1}{\sqrt{t'-r}\,(2N)^{2\alpha}} dr \lesssim \sqrt{\delta}\;,
\end{align*}
uniformly over all $t < t' \in K$, all $\ell \in B_\epsilon^N(t')$ and all $N\geq 1$. Then, we set $D^{t,t'}_u(\ell) := \big[ A^{t,t'}_\cdot(\ell)\big]_u-\langle A^{t,t'}_\cdot(\ell)\big\rangle_u$. Let $J_k$ be the set of jump times of $\xi^N(\cdot,k)$. We have the almost sure bound
\begin{align*}
	\big[ D^{t,t'}_\cdot(\ell)\big]_\delta \lesssim \gamma_N^4\sum_k  \sum_{\tau\in (t,t']\cap J_k}  q^N_{\tau,t'}(k,\ell)^4 \Big(\frac{\xi^N(\tau,k)}{b^N(\tau,k)}\Big)^4\;,
\end{align*}
uniformly over all the parameters. Thus, the same computation as in the proof of Lemma \ref{Lemma:IncrSpaceKPZ} ensures that
\begin{equation*}
	\E \Big[ \big[ D^{t,t'}_\cdot(\ell)\big]_\delta^\frac{n}{2}\Big]^\frac{2}{n} \lesssim \gamma_N^4\;,
\end{equation*}
uniformly over all $t < t' \in K$, all $\ell \in B_\epsilon^N(t')$ and all $N\geq 1$. Thus, by (\ref{Eq:BDG2}), we deduce that
\begin{equation*}
	\E \big[|N^{t'}_{t'}(\ell)-N^{t'}_{t}(\ell)|^{2n}\big]^{\frac{1}{2n}} \lesssim |t'-t|^{\frac{1}{4}} + \frac{1}{(2N)^\alpha}\;,
\end{equation*}
uniformly over the same set of parameters.\\
We turn our attention to $B^{t,t'}$. First, we have the identity
\begin{equation*}
	B^{t,t'}_s(\ell) = \int_0^s \sum_k \big(p^N_{t'-r}(k,\ell)-p^N_{t-r}(k,\ell) \big) dM^N(r,k)\;,\quad \forall s\leq t\;,
\end{equation*}
so that
\begin{equation*}
	\E \Big[ \big\langle B^{t,t'}_\cdot(\ell)\big\rangle_\delta^n\Big]^\frac{1}{n} \lesssim (2N)^{2\alpha} \int_0^t \sum_{k=1}^{2N-1} \big(q^N_{r,t'}(k,\ell)-q^N_{r,t}(k,\ell)\big)^2 dr\;.
\end{equation*}
At this point, we argue differently according as $k$ belongs to $B^N_{\epsilon/2}(r)$ or not. Using Lemma \ref{Lemma:ExpoDecay}, we have
\begin{equation*}
	(2N)^{2\alpha} \int_0^t \sum_{k\notin B^N_{\epsilon/2}(r)} \big(q^N_{r,t'}(k,\ell)-q^N_{r,t}(k,\ell)\big)^2 dr \lesssim N^{1+2\alpha} e^{-\delta N^{2\alpha}} \lesssim N^{-2\alpha}\;,
\end{equation*}
uniformly over all $\ell \in B^N_\epsilon(t')$, all $t<t'\in K$ and all $N\geq 1$. On the other hand, using Lemma \ref{Lemma:BoundHeatKernel}, we get for all $\beta \in (0,1/4)$
\begin{equation*}
(2N)^{2\alpha} \int_0^t \sum_{k\in B^N_{\epsilon/2}(r)} \big(p^N_{r,t'}(k,\ell)-p^N_{r,t}(k,\ell)\big)^2 dr
	\lesssim \int_0^t \frac{|t'-t|^{2\beta}}{(t-r)^{\frac{1}{2}+2\beta}} dr\lesssim (t'-t)^{2\beta}\;,
\end{equation*}
uniformly over all $t < t' \in K$, all $\ell \in B^N_\epsilon(t')$ and all $N\geq 1$. Furthermore, we set $E^{t,t'}_s := \big[B^{t,t'}_\cdot\big]_s - \langle B^{t,t'}_\cdot\rangle_s$, $s\leq t$ and we have the almost sure bound
\begin{align*}
	\big[ E^{t,t'}_\cdot(\ell)\big]_t \lesssim \gamma_N^4 \sum_{\tau\in (0,t]} \sum_k \big(q^N_{\tau,t'}(k,\ell)-q^N_{\tau,t}(k,\ell)\big)^4 \Big(\frac{\xi^N(\tau,k)}{b^N(\tau,k)}\Big)^4\;,
\end{align*}
	uniformly over all $0\leq t < t'$, all $\ell$ and all $N\geq 1$. This being given, we apply the same arguments as in the proof of Lemma \ref{Lemma:IncrSpaceKPZ} to get
\begin{equation*}
	\E \Big[ \big[ E^{t,t'}_\cdot(\ell)\big]_t^\frac{n}{2}\Big]^\frac{2}{n} \lesssim \gamma_N^4\;,
\end{equation*}
uniformly over the same set of parameters. Using (\ref{Eq:BDG2}), we deduce that
\begin{equation*}
	\E \big[|N^{t'}_{t}(\ell)-N^{t}_{t}(\ell)|^{2n}\big]^{\frac{1}{2n}} \lesssim |t'-t|^{\beta} + \frac{1}{(2N)^\alpha}\;,
\end{equation*}
uniformly over the same set of parameters, thus concluding the proof.
\end{proof}

\begin{proposition}
Fix $t_0 \in [0,T)$. The sequence $\xi^N$ is tight in $\bbD([0,t_0],\cC(\R))$, and any limit is continuous in time.
\end{proposition}
\begin{proof}
One introduces a piecewise linear time-interpolation $\bar{\xi}^N$ of our process $\xi^N$, namely we set $t_N :=\lfloor t(2N)^{4\alpha}\rfloor$ and
\begin{equation*}
	\bar{\xi}^N (t,\cdot) := \big(t_N+1-t(2N)^{4\alpha}\big)\xi^N\Big(\frac{t_N}{(2N)^{4\alpha}},\cdot\Big) + \big(t(2N)^{4\alpha}-t_N\big)\xi^N\Big(\frac{t_N+1}{(2N)^{4\alpha}},\cdot\Big)\;.
\end{equation*}
Using Lemmas \ref{Lemma:IncrSpaceKPZ} and \ref{Lemma:IncrTimeKPZ}, it is simple to show that the space-time H\"older semi-norm of $\bar{\xi}^N$ on compact sets of $[0,T)\times\R$ has finite moments of any order, uniformly over all $N\geq 1$. Additionally, the proof of Lemma 4.7 in~\cite{BG97} carries through, and ensures that $\xi^N-\bar{\xi}^N$ converges to $0$ uniformly over compact sets of $[0,T)\times\R$ in probability. All these arguments provide the required control on the space-time increments of $\xi^N$ to ensure its tightness, following the calculation below Proposition 4.9 in~\cite{BG97}.
\end{proof}

\subsection{The key lemma}

We use the notation $\langle f,g \rangle$ to denote the inner product of $f$ and $g$ in $L^2(\R)$. Similarly, for all maps $f,g:[0,2N] \rightarrow \R$, we set
\begin{equation*}
\langle f , g \rangle_N := \frac{1}{(2N)^{2\alpha}} \sum_{k=1}^{2N-1} f(k)g(k)\;.
\end{equation*}
Notice that the scaling here is different from the one used in the notation \eqref{Eq:InnerProduct}. This is because we only look at a window of order $(2N)^\alpha$ in space for the KPZ fluctuations, while in the hydrodynamic limit we were considering the whole lattice of size $2N$.

To conclude the proof of Theorem \ref{Th:KPZ}, it suffices to show that any limit point $\xi$ of a convergent subsequence of $\xi^N$ satisfies the following martingale problem (see Proposition 4.11 in~\cite{BG97}).

\begin{definition}[Martingale problem]
Let $(\xi(t,x),t\in [0,T), x\in \R)$ be a continuous process satisfying the following two conditions. Let $t_0\in [0,T)$. First, there exists $a > 0$ such that
\begin{equation*}
\sup_{t\leq t_0}\sup_{x\in \R} e^{-a|x|} \E\big[\xi(t,x)^2\big] < \infty\;.
\end{equation*}
Second, for all $\varphi\in\cC^\infty_c(\R)$, the processes
\begin{align*}
M(t,\varphi) &:= \langle \xi(t),\varphi\rangle - \langle \xi(0),\varphi\rangle - \frac12 \int_0^t \langle \xi(s),\varphi''\rangle ds\;,\\
L(t,\varphi) &:= M(t,\varphi)^2 - 4 \int_0^t \langle \xi(s)^2,\varphi^2\rangle ds\;,
\end{align*}
are local martingales on $[0,t_0]$. Then, $\xi$ is a solution of (\ref{mSHE}) on $[0,T)$.
\end{definition}
The first condition is a simple consequence of Proposition \ref{Prop:BoundMomentsKPZ}. To prove that the second condition is satisfied, we introduce the discrete analogues of the above processes. For all $\varphi\in\cC^\infty_c(\R)$, the processes
\begin{align*}
M^N(t,\varphi) = \langle \xi^N(t),\varphi\rangle_N - \langle \xi^N(0),\varphi\rangle_N - \frac12 \int_0^t \langle \xi^N(s),(2N)^{4\alpha} \Delta \varphi\rangle_N ds\;,\\
L^N(t,\varphi) = M^N(t,\varphi)^2 - \frac{2\lambda_N}{(2N)^{2\alpha}} \int_0^t \langle \xi^N(s)^2,\varphi^2\rangle_N ds + R^N_1(t,\varphi) + R^N_2(t,\varphi)\;,
\end{align*}
are martingales, where
\begin{align*}
R^N_1(t,\varphi) &:= -\frac{\lambda_N}{(2N)^{2\alpha}}\int_0^t \langle \xi^N(s) \Delta \xi^N(s) , \varphi^2 \rangle_N ds\;,\\
R^N_2(t,\varphi) &:= (2N)^{2\alpha} \int_0^t \langle \nabla^+\xi^N(s) \nabla^- \xi^N(s) , \varphi^2 \rangle_N ds\;.
\end{align*}
If we show that $R^N_1(t,\varphi)$ and $R^N_2(t,\varphi)$ vanish in probability when $N\rightarrow\infty$, then passing to the limit on a convergent subsequence, we easily deduce that the martingale problem above is satisfied. Below, we will be working on $[N-A(2N)^{2\alpha},N+A(2N)^{2\alpha}]$ where $A$ is a large enough value such that $[-A,A]$ contains the support of $\varphi$. The moments of $\xi^N$ on this interval are of order $1$ thanks to Proposition \ref{Prop:BoundMomentsKPZ}. Since $|\Delta \xi^N| \lesssim \gamma_N \xi^N$, we have
\begin{equation*}
\E\big[ |R^N_1(t,\varphi)| \big] \lesssim \gamma_N \int_0^t \frac{1}{(2N)^{2\alpha}} \sum_k \varphi^2\Big(\frac{k-N}{(2N)^{2\alpha}}\Big) ds \lesssim \gamma_N\;,
\end{equation*}
so that $R^N_1(t,\varphi)$ converges to $0$ in probability as $N\rightarrow\infty$. To prove that $R^N_2$ converges to $0$ in probability, it suffices to apply the following delicate estimate which is the analogue of Lemma 4.8 in~\cite{BG97}.

\begin{proposition}\label{Prop:DelicateKPZ}
There exists $\kappa > 0$ such that for all $A>0$, we have
\begin{equation*}
\E\Big[\big|\E\big[\nabla^+ \xi^N(t,\ell) \nabla^- \xi^N(t,\ell) \,|\,\cF_s\big]\big|\Big] \lesssim \frac{1}{(2N)^{2\alpha+\kappa}\sqrt{t-s}}\;,
\end{equation*}
uniformly over all $\ell\in[N-A(2N)^{2\alpha},N+A(2N)^{2\alpha}]$, all $s<t$ in a compact set of $[N^{-\alpha},T)$ and all $N\geq 1$.
\end{proposition}

\noindent To prove this proposition, we need to collect some preliminary results. Recall the decomposition (\ref{Eq:DiscreteHC}). If we set
\begin{equation*}
K^N_{t-r}(k,\ell) := \nabla^+ p^N_{t-r}(k,\ell) \nabla^- p^N_{t-r}(k,\ell)\;,
\end{equation*}
(here the gradients act on the variable $\ell$), then using the martingale property of $N^t_\cdot(\ell)$ we obtain for all $s\leq t$
\begin{align*}
\E\Big[\nabla^+ \xi^N(t,\ell) \nabla^- \xi^N(t,\ell)  \, | \, \cF_s \Big] &= \big(\nabla^+ I^{N}(t,\ell) + \nabla^+ N^t_s(\ell)\big)\big(\nabla^- I^{N}(t,\ell) + \nabla^- N^t_s(\ell)\big)\\
&\quad+ \E\bigg[\int_{s}^t \sum_{k=1}^{2N-1} K^N_{t-r}(k,\ell) \,d\langle M^N(\cdot,k)\rangle_r\, \Big| \, \cF_s\bigg]\;.
\end{align*}
Set
\begin{equation*}
f^N_s(t,\ell) := \E\bigg[\Big|\E\Big[\nabla^+\xi^N(t,\ell) \nabla^-\xi^N(t,\ell)\,\big|\, \cF_s\Big]\Big|\bigg]\;.
\end{equation*}
Fix $\epsilon > 0$. Using the expression of the bracket (\ref{Eq:Bracket}) of $M^N$, we get
\begin{equation}\label{Eq:fN}
f^N_s(t,\ell) \leq D^N_s(t,\ell) + \int_s^t \sum_{k\in B^N_{\epsilon/2}(r)} (2N)^{4\alpha} |K^N_{t-r}(k,\ell)| f^N_s(r,k) dr\;,
\end{equation}
where $D^N_s(t,\ell) = D^{N,1}_s(t,\ell) + D^{N,2}_s(t,\ell) + D^{N,3}_s(t,\ell)$ with
\begin{align*}
D^{N,1}_s(t,\ell) &:= \E\Big[\big|\big(\nabla^+ I^{N}(t,\ell) + \nabla^+ N^t_s(\ell)\big)\big(\nabla^- I^{N}(t,\ell) + \nabla^- N^t_s(\ell)\big) \big|\Big]\;,\\
D^{N,2}_s(t,\ell) &:= \int_s^t \sum_{k\notin B^N_{\epsilon/2}(r)} (2N)^{4\alpha} |K^N_{t-r}(k,\ell)| f^N_s(r,k) dr\;,\\
D^{N,3}_s(t,\ell) &:= \lambda_N\E\bigg[\Big|\E\Big[\int_{s}^t \sum_{k=1}^{2N-1} K^N_{t-r}(k,\ell)\big(\xi^N(r,k)\Delta\xi^N(r,k)\\
&\qquad\qquad\qquad+2\xi^N(r,k)^2\big)dr \, \big| \, \cF_s\Big]\Big|\bigg]\;.
\end{align*}
From now on, we fix a compact set $K\subset [0,T)$.
\begin{lemma}\label{Lemma:BoundDN}
Fix $\epsilon > 0$. There exists $\kappa > 0$ such that
\begin{equation}\label{Eq:BoundDNStrong}
D^N_s(t,\ell) \lesssim 1 \wedge \frac{1}{(2N)^{2\alpha+\kappa} \sqrt{t-s}}\;,
\end{equation}
uniformly over all $\ell\in B^N_{\epsilon}(t)$, all $N^{-\alpha} \leq s < t \in K$ and all $N\geq 1$.
\end{lemma}

\begin{proof}
Let us observe that we have the simple bound
\begin{equation}\label{Eq:TrivialBoundfN}
f^N_s(r,k) \lesssim b^N(r,k)^2 \gamma_N^2\;,
\end{equation}
uniformly over all the parameters. Recall also that $b^N(t,\ell)$ is of order $1$ whenever $\ell \in B^N_\epsilon(t)$.\\
Let $\bar{p}^N$ be the discrete heat kernel on the whole line $\Z$ sped up by $2c_N$, see Appendix \ref{Appendix:Kernel}, and set $\bar{K}^N_t(k,\ell) = \nabla^+ \bar{p}^N_t(\ell-k)\nabla^- \bar{p}^N_t(\ell-k)$.\\
\textit{Bound of $D^{N,1}_s$.} It suffices to bound the square of the $L^2$-norms of $\nabla^\pm I^N(t,\ell)$ and $\nabla^\pm N^t_s(\ell)$. By Proposition \ref{Prop:IC}, we deduce that $(\nabla^\pm I^N(t,\ell))^2 \lesssim  N^{-5\alpha}$ uniformly over all $N^{-\alpha} \leq t \in K$, all $\ell \in B^N_\epsilon(t)$ and all $N\geq 1$.\\
We now treat $\nabla^+ N^t_s(\ell)$ (the proof is the same with $\nabla^-$). Using again Lemmas \ref{Lemma:ExpoDecay} and \ref{Lemma:DecaySeriesKernel}, we have
\begin{align*}
\E\Big[\big(\nabla^+ N^t_s(\ell)\big)^2\Big] &\lesssim \E \Big[\sum_{k=1}^{2N-1} \int_0^s\big(\nabla^+ p^N_{t-r}(k,\ell)\big)^2 d\langle M(\cdot,k)\rangle_r\Big] \\
&\lesssim (2N)^{2\alpha} \!\!\sum_{k\in B^N_{\epsilon/2}(r)}\!\! \int_0^s\big(\nabla^+ \bar{p}^N_{t-r}(k,\ell)\big)^2 dr + \cO\big(N^{1+2\alpha}e^{-\delta N^{2\alpha}}\big)\;,
\end{align*}
uniformly over all $\ell\in B^N_\epsilon(t)$, all $t\in K$ and all $N\geq 1$. Using Lemma \ref{Lemma:BoundHeatKernelZ}, we easily deduce that the last expression is bounded by a term of order $1 \wedge 1/(\sqrt{t-s}(2N)^{4\alpha})$ as required.\\
\textit{Bound of $D^{N,2}$.} Using the exponential decay of Lemma \ref{Lemma:ExpoDecay} and (\ref{Eq:TrivialBoundfN}), we deduce that there exists $\delta > 0$ such that
\begin{equation*}
\int_s^t \sum_{k\notin B^N_{\epsilon/2}(r)} (2N)^{4\alpha} |K^N_{t-r}(k,\ell)| f^N_s(r,k) dr \lesssim \int_s^t \sum_{k\notin B^N_{\epsilon/2}(r)} (2N)^{2\alpha} e^{-\delta N^{2\alpha}} dr\;,
\end{equation*}
uniformly over all $\ell\in B^N_\epsilon(t)$, all $s\leq t \in K$ and all $N\geq 1$. This trivially yields a bound of order $N^{-3\alpha}$ as required.\\
\textit{Bound of $D^{N,3}$.} By Lemmas \ref{Lemma:DecaySeriesKernel} and \ref{Lemma:ExpoDecay}, there exists $\delta>0$ such that $D^{N,3}_s(t,\ell)$ can be rewritten as
\begin{equation}\label{Eq:ExpressionDelicateRewritten}
\lambda_N\E\bigg[\Big|\E\Big[\int_{s}^t \sum_{k\in B^N_{\epsilon/2}(r)} \bar{K}^N_{t-r}(k,\ell)\big(\xi^N(r,k)\Delta\xi^N(r,k)+2\xi^N(r,k)^2\big)dr \, \big| \, \cF_s\Big]\Big|\bigg]\;,
\end{equation}
up to an error of order $N^{2\alpha+1}e^{-\delta N^{2\alpha}}$, uniformly over all $\ell \in B^N_\epsilon(t)$, all $t\in K$ and all $N\geq 1$. The error term satisfies the bound of the statement. We bound separately the two contributions arising in (\ref{Eq:ExpressionDelicateRewritten}). First, using the almost sure bound $|\Delta \xi^N(r,k)| \lesssim \gamma_N \xi^N(r,k)$, we get
\begin{align*}
{}&\lambda_N\E\bigg[\Big|\E\Big[\int_{s}^t \sum_{k\in B^N_{\epsilon/2}(r)} \bar{K}^N_{t-r}(k,\ell) \xi^N(r,k)\Delta\xi^N(r,k)dr \, \big| \, \cF_s\Big]\Big|\bigg]\\
&\lesssim (2N)^{\alpha} \int_s^t \sum_{k\in B^N_{\epsilon/2}(r)} |\bar{K}^N_{t-r}(k,\ell)| dr \;,
\end{align*}
uniformly over all $\ell \in B^N_\epsilon(t)$, all $t\in K$ and all $N\geq 1$. Using Lemma \ref{Lemma:BoundHeatKernelZ}, this easily yields a bound of order $1/(2N)^{2\alpha + \kappa}$ with $\kappa > 0$, as required. Second, we have
\begin{align*}
{}&\int_{s}^t \sum_{k\in B^N_{\epsilon/2}(r)} \bar{K}^N_{t-r}(k,\ell)\E\Big[ \xi^N(r,k)^2 \, \big| \, \cF_s\Big]dr\\
&= \int_{s}^t \sum_{k\in B^N_{\epsilon/2}(r)} \bar{K}^N_{t-r}(k,\ell)\E\Big[ \xi^N(r,k)^2 - \xi^N(t,\ell)^2 \, \big| \, \cF_s\Big]dr\\
&+ \int_{s}^t \sum_{k\in B^N_{\epsilon/2}(r)} \bar{K}^N_{t-r}(k,\ell)dr \, \E\Big[ \xi^N(t,\ell)^2 \, \big| \, \cF_s\Big]\;.
\end{align*}
We claim that we have
\begin{equation}\label{Eq:IPP}
\int_{s}^t \sum_{k\in B^N_{\epsilon/2}(r)} \bar{K}^N_{t-r}(k,\ell)dr = -\int_{t-s}^\infty \sum_{k\in\Z} \bar{K}^N_{r}(k,\ell)dr - \int_{s}^t \sum_{k\notin B^N_{\epsilon/2}(r)} \bar{K}^N_{t-r}(k,\ell)dr\;.
\end{equation}
We postpone the proof of this identity. The second term on the right can be bounded using Lemma \ref{Lemma:ExpoDecay}: it has a negligible contribution. Using Lemma \ref{Lemma:BoundHeatKernelZ} on the first term, we easily deduce that
\begin{equation*}
\lambda_N \E\bigg[\Big|\int_{s}^t \sum_{k\in B^N_{\epsilon/2}(r)} \bar{K}^N_{t-r}(k,\ell)dr \, \E\Big[ \xi^N(t,\ell)^2 \, \big| \, \cF_s\Big]\Big|\bigg] \lesssim 1 \wedge \frac{1}{\sqrt{t-s} (2N)^{4\alpha}}\;,
\end{equation*}
uniformly over all $\ell \in B^N_\epsilon(t)$, all $t\in K$ and all $N\geq 1$. On the other hand, for any given $\beta \in (0,1/4)$, the Cauchy-Schwarz inequality together with Lemmas \ref{Lemma:IncrSpaceKPZ} and \ref{Lemma:IncrTimeKPZ} yields
\begin{align*}
\E\Big[\big|\xi^N(r,k)^2 - \xi^N(t,\ell)^2 \big| \Big] &\lesssim \E\Big[\big(\xi^N(r,k)+\xi^N(r,\ell)\big)^2\Big]^{\frac12}\E\Big[\big(\xi^N(r,k)-\xi^N(r,\ell)\big)^2\Big]^{\frac12}\\
&+ \E\Big[\big(\xi^N(r,\ell)+\xi^N(t,\ell)\big)^2\Big]^{\frac12}\E\Big[\big(\xi^N(r,\ell)-\xi^N(t,\ell)\big)^2\Big]^{\frac12}\\
&\lesssim 1 \wedge \Big( \Big|\frac{\ell-k}{(2N)^{2\alpha}}\Big|^{2\beta} + |t-r|^\beta + \frac{1}{(2N)^{\alpha}} \Big) \;, 
\end{align*}
uniformly over all $\ell \in B^N_\epsilon(t)$, all $k\in B^N_{\epsilon/2}(r)$, all $r\leq t \in K$ and all $N\geq 1$. Using Lemma \ref{Lemma:BoundHeatKernelZ}, it is simple to deduce the existence of $\kappa \in (0,1)$ such that
\begin{equation*}
\lambda_N \E\bigg[\Big|\int_{s}^t \sum_{k\in B^N_{\epsilon/2}(r)} \bar{K}^N_{t-r}(k,\ell)\E\Big[ \xi^N(r,k)^2 - \xi^N(t,\ell)^2 \, \big| \, \cF_s\Big]dr \Big|\bigg] \lesssim \frac{1}{(2N)^{2\alpha+\kappa}}\;,
\end{equation*}
uniformly over all $s<t \in K$, all $\ell \in B^N_\epsilon(t)$ and all $N\geq 1$.\\
It remains to establish \eqref{Eq:IPP}. To that end, we observe that
\begin{align*}
&\int_{s}^t \sum_{k\in B^N_{\epsilon/2}(r)} \bar{K}^N_{t-r}(k,\ell)dr + \int_{s}^t \sum_{k\notin B^N_{\epsilon/2}(r)} \bar{K}^N_{t-r}(k,\ell)dr\\
&= \int_{0}^{t-s}\sum_{k\in\Z} \bar{K}^N_{r}(k,\ell)dr\\
&=\int_{0}^{\infty}\sum_{k\in\Z} \bar{K}^N_{r}(k,\ell)dr - \int_{t-s}^{\infty}\sum_{k\in\Z} \bar{K}^N_{r}(k,\ell)dr\;.
\end{align*}
Notice that the last two integrals converge absolutely thanks to the estimates collected in Lemma \ref{Lemma:BoundHeatKernelZ}. An integration by parts shows that
$$ \sum_{k\in\Z} \bar{K}^N_{r}(k,\ell) = - \sum_{k\in\Z} \bar{p}^N_{r}(\ell-k) \Delta \bar{p}^N_{r}(\ell-k-1)\;,$$
but also that
$$ \sum_{k\in\Z} \bar{K}^N_{r}(k,\ell) = - \sum_{k\in\Z} \bar{p}^N_{r}(\ell-k-1) \Delta \bar{p}^N_{r}(\ell-k)\;.$$
Recall that $\partial_t \bar{p}^N_{r} = \Delta \bar{p}^N_{r}$. Therefore
$$ \sum_{k\in\Z} \bar{K}^N_{r}(k,\ell) = -\frac12 \sum_{k\in\Z} \partial_t\big(\bar{p}^N_{r}(\ell-k) \bar{p}^N_{r}(\ell-k-1)\big)\;.$$
The integral over $(0,\infty)$ of this last expression vanishes. Therefore,
\begin{align*}
&\int_{s}^t \sum_{k\in B^N_{\epsilon/2}(r)} \bar{K}^N_{t-r}(k,\ell)dr + \int_{s}^t \sum_{k\notin B^N_{\epsilon/2}(r)} \bar{K}^N_{t-r}(k,\ell)dr\\
&= - \int_{t-s}^{\infty}\sum_{k\in\Z} \bar{K}^N_{r}(k,\ell)dr\;,
\end{align*}
thus concluding the proof of the claim.
\end{proof}

\noindent We have all the elements at hand to prove the main result of this section.
\begin{proof}[Proof of Proposition \ref{Prop:DelicateKPZ}]
Iterating (\ref{Eq:fN}) and using Lemma \ref{Lemma:BoundK} and the bound (\ref{Eq:TrivialBoundfN}), we deduce that
\begin{equation*}
f^N_s(t,\ell) \leq D^N_s(t,\ell) + \sum_{n\geq 1} H_s(t,\ell,n)\;,
\end{equation*}
where for all $n\geq 1$, we set $t_{n+1}=t$, $k_{n+1}=\ell$ and
\begin{equation*}
H_s(t,\ell,n):= \int\limits_{s\leq t_1\leq \ldots \leq t_{n} \leq t} \sum_{k_i \in B_{\epsilon/2}^N(t_i)} D_s^N(t_1,k_1) \prod_{i=1}^{n} (2N)^{4\alpha}|K^N_{t_{i+1}-t_i}(k_i,k_{i+1})| dt_i\;.
\end{equation*}
By Lemma \ref{Lemma:BoundDN}, we already know that $D^N_s(t,\ell)$ satisfies the bound of the statement of Proposition \ref{Prop:DelicateKPZ}. To conclude the proof of the proposition, we only need to show that this is also the case for the sum over $n\geq 1$ of $H_s(t,\ell,n)$.\\
Fix $A>0$. Let $n_0=c\log N$, for an arbitrary $c > -3\alpha/\log\beta$, where $\beta < 1$ is taken from Lemma \ref{Lemma:BoundK}. Using Lemmas \ref{Lemma:BoundDN} and  \ref{Lemma:BoundK}, we easily deduce that $H_s(t,\ell,n) \lesssim \beta^n$ uniformly over all $n\geq 1$, all $\ell \in \{N-A(2N)^{2\alpha}, N+A(2N)^{2\alpha}\}$ and all $N^{-\alpha} \leq s \leq t \in K$.  Given the definition of $n_0$, we deduce that
$$ \sum_{n\geq n_0}H_s(t,\ell,n) \lesssim (2N)^{-3\alpha}\;,$$
uniformly over the same set of parameters, as required.\\
Let us now treat $\sum_{n<n_0} H_s(t,\ell,n)$. We introduce
\begin{equation*}
A_s(t,\ell,n):= \int\limits_{s\leq t_1\leq \ldots \leq t_{n} \leq t} \sum_{k_i \in B^N_\epsilon(t_i)} D_s^N(t_1,k_1) \prod_{i=1}^{n} (2N)^{4\alpha}|K^N_{t_{i+1}-t_i}(k_i,k_{i+1})| dt_i\;.
\end{equation*}
By Lemma \ref{Lemma:BoundDN}, we have
\begin{equation*}
A_s(t,\ell,n)\lesssim \!\!\!\! \int\limits_{s\leq t_1\leq \ldots \leq t_{n} \leq t}\!\! \sum_{k_i \in B^N_\epsilon(t_i)} \frac{1}{(2N)^{2\alpha+\kappa}\sqrt{t_1-s}} \prod_{i=1}^{n} (2N)^{4\alpha}|K^N_{t_{i+1}-t_i}(k_i,k_{i+1})| dt_i\;.
\end{equation*}
If we restrict the domain of integration to those $t_1$ such that $t_1-s \geq (t-s)/(n+1)$, then a simple calculation based on Lemma \ref{Lemma:BoundK} ensures that this restricted integral is bounded by a term of order
\begin{equation*}
\frac{\sqrt{n+1} \beta^n}{(2N)^{2\alpha+\kappa}\sqrt{t-s}} \lesssim \frac{1}{(2N)^{2\alpha+\kappa}\sqrt{t-s}}\;,
\end{equation*}
for all $n\leq n_0$. On the other hand, when $t_1-s < (t-s)/(n+1)$ there is at least one increment $t_{i+1}-t_i$ which is larger than $(t-s)/(n+1)$. By symmetry, let us consider the case $i=1$. By Lemma \ref{Lemma:DecaySeriesKernel}, we can replace $K^N_{t_{2}-t_1}(k_1,k_{2})$ with $\bar{K}^N_{t_{2}-t_1}(k_1,k_{2})$ up to a negligible term. By Lemma \ref{Lemma:BoundHeatKernelZ}, we bound the sum over $k_1$ of $|\bar{K}^N_{t_{2}-t_1}(k_1,k_{2})|$ by a term of order $(2N)^{-4\alpha}(t_2-t_1)^{-1}$ thus yielding the bound
\begin{align*}
{}&\int_s^{s+\frac{t-s}{n+1}} \sum_{k_1 \in B^N_\epsilon(t_1)} \frac{1}{(2N)^{2\alpha+\kappa}\sqrt{t_1-s}}\, (2N)^{4\alpha}|K^N_{t_{2}-t_1}(k_1,k_{2})| dt_1\\
&\lesssim \int_{s}^{s+\frac{t-s}{n+1}} \frac{1}{(2N)^{2\alpha+\kappa}\sqrt{t_1-s}\,(t_2-t_1)}\, dt_1 + \cO(N^{1+4\alpha}e^{-\delta N^{2\alpha}})\\
&\lesssim \frac{\sqrt{\frac{t-s}{n+1}}}{(2N)^{2\alpha+\kappa}\frac{t-s}{n+1}} \lesssim \frac{1}{(2N)^{2\alpha+\kappa'} \sqrt{t-s}}\;,
\end{align*}
for all $\kappa'\in (0,\kappa)$ and all $n< n_0$. Using Lemma \ref{Lemma:BoundK}, we can bound the integral over $t_2,\ldots,t_n$ of the remaining terms by a term of order $\beta^{n-1}$. Consequently, we have proved that there exists $\kappa'>0$ such that
\begin{equation}\label{Eq:BoundAs}
A_s(t,\ell,n) \lesssim \frac{\beta^{n-1}}{(2N)^{2\alpha + \kappa'}\sqrt{t-s}}\;,
\end{equation}
uniformly over all $\ell \in [N-A(2N)^{2\alpha},N+A(2N)^{2\alpha}]$, all $t\in K$, all $s\in [0,t]$, all $n < n_0$ and all $N\geq 1$.\\
Finally, we set $B_s(t,\ell,n) := H_s(t,\ell,n) - A_s(t,\ell,n)$. We can replace each occurrence of $p^N$ by $\bar{p}^N$ up to a negligible term, using Lemma \ref{Lemma:DecaySeriesKernel}. Among the parameters $k_1,\ldots,k_n$ involved in the definition of $B_s(t,\ell,n)$, at least one them, say $k_{i_0}$, belongs to $B^N_{\epsilon/2}(t_{i_0})\backslash B^N_\epsilon(t_{i_0})$. Then, using the bound $\big| \bar{K}^N_{t}(k,\ell) \big| \leq \bar{p}^N_{t}(k,\ell)$ together with the semigroup property of the discrete heat kernel at the second line and the exponential decay of Lemma \ref{Lemma:ExpoDecay}, we get
\begin{align*}
{}&\sum_{k_{i_0 +1},\ldots,k_n} \prod_{i=i_0}^n \big| \bar{K}^N_{t_{i+1}-t_i}(k_i,k_{i+1}) \big|\\
&\leq \sum_{k_{i_0 +1},\ldots,k_n} \prod_{i=i_0}^n \bar{p}^N_{t_{i+1}-t_i}(k_i,k_{i+1})= \bar{p}^N_{t-t_{i_0}}(k_{i_0},\ell)\lesssim e^{-\delta N^{2\alpha}}\;,
\end{align*}
uniformly over all the parameters. Using Lemma \ref{Lemma:BoundDN}, one easily gets
\begin{equation*}
B_s(t,\ell,n) \lesssim (2N)^{n4\alpha} e^{-\delta(2N)^{2\alpha}}\;,
\end{equation*}
uniformly over all $n\geq 1$, all $s<t\in K$ and all $\ell \in [N-A(2N)^{2\alpha},N+A(2N)^{2\alpha}]$. Given the definition of $n_0$, we deduce that the sum over all $n<n_0$ of the latter is negligible w.r.t.~$(2N)^{-3\alpha}$, uniformly over the same set of parameters. This concludes the proof.
\end{proof}

\section{Appendix}

\subsection{Martingale inequalities}\label{Appendix:Mgale}

Let $X(t),t\geq 0$ be a c\`adl\`ag, mean zero, square-integrable martingale. Let $\langle X \rangle_t, t\geq 0$ denote the bracket of $X$, that is, the unique predictable process such that $X^2-\langle X \rangle$ is a martingale. Let $[X]_t$ denote its quadratic variation: in the case where the martingale is of finite variation, we have
\begin{equation*}
{[X]}_t = \sum_{\tau \in (0,t]} (X_\tau - X_{\tau-})^2\;.
\end{equation*}
The Burkholder-Davis-Gundy inequality ensures that for every $p \geq 1$, there exists $c(p) > 0$ such that
\begin{equation}\label{Eq:BDG}
\E\big[ |X_t|^p \big]^{\frac{1}{p}} \leq c(p) \E\Big[ [X]_t^{\frac{p}{2}} \Big]^{\frac{1}{p}}\;.
\end{equation}
It happens that the process $D_t = [X]_t - \langle X \rangle_t$ is also a martingale. Thus, using twice the Burkholder-Davis-Gundy inequality, one gets that for every $p\geq 2$ there exists $c'(p) > 0$ such that
\begin{equation}\label{Eq:BDG2}
	\E\big[ |X_t|^p \big]^{\frac{1}{p}} \leq c'(p)\Big( \E\Big[ \langle X \rangle_t^{\frac{p}{2}} \Big]^{\frac{1}{p}} + \E\Big[ [ D ]_t^{\frac{p}{4}} \Big]^{\frac{1}{p}}\Big)\;.
\end{equation}
We will also rely on the following inequality
\begin{equation}\label{Eq:BDG3}
	\E\Big[ \sup_{s\leq t} |X_s|^p \Big]^{\frac{1}{p}} \leq c''(p) \Big( \E\Big[ \langle X \rangle_t^{\frac{p}{2}} \Big]^{\frac{1}{p}} + \E\Big[ \sup_{s\leq t}|X_s-X_{s-}|^p \Big]^{\frac{1}{p}} \Big)\;,
\end{equation}
which can be found in~\cite{Lepingle} for instance.

\subsection{Discrete heat kernel estimates}\label{Appendix:Kernel}

We introduce the fundamental solution $p^N_t(k,\ell)$ of the discrete heat equation sped up by a factor $c_N > 0$
\begin{align}\label{DiscreteHeat}
	\begin{cases}
	\partial_t p^N_t(k,\ell) = c_N \Delta p^N_t(k,\ell)\;,\\
	p^N_0(k,\ell) = \delta_k(\ell)\;,\\
	p^N_t(k,0) = p^N_t(k,2N) = 0\;,\end{cases}
\end{align}
for all $k,\ell \in \{1,\ldots,2N-1\}$, as well as its analogue $\bar{p}^N_t(\ell)$ on $\Z$:
\begin{align}\label{DiscreteHeatKPZWholeLine}
	\begin{cases}\partial_t \bar{p}^N_t(\ell) = c_N \Delta \bar{p}^N_t(\ell)\;,\\
	\bar{p}^N_0(\ell) = \delta_0(\ell)\;,\end{cases}
\end{align}
for all $\ell \in \Z$. The latter is more tractable than the former since it is translation invariant. Using a coupling between a simple random walk on $\Z$ and a simple random walk killed at $0$ and $2N$, we get the elementary bound $p^N_t(k,\ell) \leq \bar{p}^N_t(\ell-k)$ for all $k,\ell \in \{1,\ldots,2N-1\}$ and all $t\geq 0$. The following estimates are standard, see for instance Lemma A.1 in~\cite{DemboTsai} or Lemma 26 in~\cite{EthLab15}.

\begin{lemma}\label{Lemma:BoundHeatKernel}
For all $\beta \in [0,1]$, we have
\begin{align*}
p^N_t(k,\ell) &\lesssim 1 \wedge \frac{1}{\sqrt{t c_N}}\;,\\
|p^N_t(k,\ell)-p^N_t(k,\ell')| &\lesssim 1 \wedge \frac{1}{\sqrt{t c_N}}\Big|\frac{\ell-\ell'}{\sqrt{c_N}}\Big|^\beta\;,\\
|p^N_t(k,\ell)-p^N_{t'}(k,\ell)| &\lesssim 1 \wedge \frac{1}{\sqrt{t c_N}}\Big|\frac{t-t'}{t}\Big|^\beta\;,\\
\end{align*}
uniformly over all $0 \leq t < t'$, all $k,\ell,\ell' \in \{1,\ldots,2N-1\}$ and all $N\geq 1$. The same bounds hold for $\bar{p}^N$.
\end{lemma}

Let us also state the following simple bounds.
\begin{lemma}\label{Lemma:BoundHeatKernelZ}
We have $\sum_k \bar{p}^N_t(k)|k| \lesssim \sqrt{c_N t}\;,\quad \bar{p}^N_t(\ell) \lesssim 1/ \sqrt{c_N t}$ as well as
\begin{equation*}
\sum_{k\in\Z} |\nabla \bar{p}^N_t(k)| \leq 2\bar{p}^N_t(0)\;,\quad \sum_{k\in\Z} |\nabla \bar{p}^N_t(k)| |k| \lesssim 1\;,\quad |\nabla \bar{p}^N_t(\ell)| \lesssim 1\wedge \frac{1}{tc_N}\;,
\end{equation*}
uniformly over all $\ell\in\Z$, all $t\geq 0$ and all $N\geq 1$.
\end{lemma}
\begin{proof}
Notice that $\sum_k \bar{p}^N_t(k)|k|$ is smaller than the square root of the variance of a simple random walk on $\Z$ at time $2c_N t$. This easily yields the first bound. The second bound follows from the Fourier decomposition of $\bar{p}^N$. We turn to the bounds involving the gradient of $\bar{p}^N$. First, $\nabla \bar{p}^N_t(k)$ is positive if $k<0$, and negative otherwise. Then, we have the simple identity
\begin{equation*}
\sum_{k<0} \nabla\bar{p}^N_t(k) = -\sum_{k\geq 0} \nabla\bar{p}^N_t(k) = \bar{p}^N_t(0) \;,
\end{equation*}
which yields the first bound. Regarding the second bound, a simple integration by parts yields
\begin{equation*}
-\sum_{k<0} \nabla\bar{p}^N_t(k) k = \sum_{k\leq 0} \bar{p}^N_t(k)\;,\quad -\sum_{k\geq 0} \nabla\bar{p}^N_t(k) k = \sum_{k> 0} \bar{p}^N_t(k)\;,
\end{equation*}
so that we get $\sum_{k\in\Z} |\nabla \bar{p}^N_t(k)| |k|=1$. To get the third bound, it suffices to use the Fourier decomposition of $\bar{p}^N$.
\end{proof}
From now on, we work in the setting of Section \ref{Section:KPZ}. Recall the definition of $q^N$ from (\ref{Def:qN}).
\begin{lemma}\label{Lemma:HeatKernel}
Uniformly over all $0\leq s < t$, all $\ell \in \{1,\ldots,2N-1\}$ and all $N\geq 1$, we have
\begin{align}\label{Eq:BoundHeatMass}
	\sum_{k=1}^{2N-1} q^N_{s,t}(k,\ell) \lesssim b^N(t,\ell)\;,\quad q^N_{s,t}(k,\ell) \lesssim b^N(t,\ell)\Big( 1\wedge \frac{1}{\sqrt{t-s}\, (2N)^{2\alpha}} \Big)\;.\;\;
\end{align}
\end{lemma}
\begin{proof}
By symmetry, it is sufficient to prove the lemma under the further assumption $\ell \leq N$. If $b^N(s,k) \leq 3$, then we have $b^N(s,k)/b^N(t,\ell) \leq 3$ while if $b^N(s,k) > 3$, then a simple calculation ensures that
\begin{equation*}
\frac{b^N(s,k)}{b^N(t,\ell)} \lesssim \begin{cases} e^{-\lambda_N(t-s) + \gamma_N(\ell-k)}\quad\mbox{ if }k\leq N\;,\\
 e^{-\lambda_N(t-s) + \gamma_N(k-\ell)}\quad\mbox{ if }k\geq N\;,\end{cases}
\end{equation*}
uniformly over all $s < t$, all $k\in \{1,\ldots,2N-1\}$, all $\ell\in\{1,\ldots,N\}$ and all $N\geq 1$. Therefore, it suffices to show that
\begin{equation}\label{Eq:BoundKernelCrude}
\sum_{k=1}^{2N-1} p^N_{t-s}(k,\ell) \lesssim 1\;,\quad p^N_{t-s}(k,\ell) \lesssim 1\wedge \frac{1}{\sqrt{t-s}(2N)^{2\alpha}}\;, 
\end{equation}
as well as
\begin{equation}\label{Eq:BoundKernelWeighted}\begin{split}
\sum_{k=1}^{2N-1} p^N_{t-s}(k,\ell)e^{-\lambda_N(t-s)+\gamma_N |\ell-k|} &\lesssim 1\;,\\
p^N_{t-s}(k,\ell)e^{-\lambda_N(t-s)+\gamma_N |\ell-k|} &\lesssim 1\wedge \frac{1}{\sqrt{t-s}(2N)^{2\alpha}}\;. 
\end{split}\end{equation}
Regarding (\ref{Eq:BoundKernelCrude}), the first bound is immediate since $p^N_t(\cdot,\ell)$ is a sub-probability measure, while the second bound is proved in Lemma \ref{Lemma:BoundHeatKernel}. We turn to (\ref{Eq:BoundKernelWeighted}). Since $k\mapsto e^{a k}$ is an eigenvector of the discrete Laplacian on $\Z$ with eigenvalue $2(\cosh a -1)$, we deduce that
\begin{align*}
\sum_{k=1}^{2N-1} p^N_{t-s}(k,\ell) e^{\gamma_N |\ell-k|} &\leq \sum_{k\in\Z} \bar{p}^N_{t-s}(\ell-k) e^{\gamma_N |\ell-k|}\\
&\leq 2 e^{2c_N(t-s)( \cosh \gamma_N - 1)} = 2 e^{\lambda_N (t-s)}\;,
\end{align*}
thus yielding the first bound. To get the second bound, it suffices to use the Fourier decomposition of $\bar{p}^N_{t-s}(\cdot)e^{\gamma_N \cdot}$.
\end{proof}

We also recall a simple bound on the heat kernel, in the flavour of large deviations techniques. For all $a>0$, $t>0$ and $N\geq 1$, we have
\begin{equation}\label{Eq:LargeDevKernel}
\sum_{k\geq a} \bar{p}^N_t(k) \leq e^{2t c_N g\big(\frac{a}{2c_N t}\big)}\;,
\end{equation}
where $g(x) =\cosh(\argsh x) - x\,\argsh x - 1$ for all $x\in\R$. By studying the function $g(x)/x$, one easily deduces that the term on the r.h.s.~is increasing with $t$.\\
We let $\bar{q}^N_{s,t}(k,\ell) := \bar{p}^N_{t-s}(k,\ell) b^N(s,k)$.

\begin{lemma}\label{Lemma:ExpoDecay}
Fix a compact set $K\subset [0,T)$ and $\epsilon>0$. There exists $\delta >0$ such that
\begin{equation*}
q^N_{s,t}(k,\ell) \leq \bar{q}^N_{s,t}(k,\ell) \lesssim e^{-\delta N^{2\alpha}}\;,
\end{equation*}
uniformly over all $k\notin B^N_{\epsilon/2}(s)$, all $\ell\in B^N_\epsilon(t)$ and all $0 \leq s \leq t \in K$.
\end{lemma}
\begin{proof}
Let us consider the case where $b^N(s,k)\geq 3$; by symmetry we can assume that $k \in \{1,\ldots,N\}$. Then, we apply (\ref{Eq:LargeDevKernel}) to get
\begin{equation}\label{Eq:IntermediateExpoDecay}
\bar{p}^N_{t-s}(\ell-k) e^{\lambda_N s - \gamma_N k} \leq e^{2(t-s) c_N g\big(\frac{\ell-k}{2c_N (t-s)}\big) + \lambda_N s - \gamma_N k}\;.
\end{equation}
We argue differently according to the value of $\alpha$. If $4\alpha \leq 1$, then $(\ell-k)/c_N$ is bounded away from $0$ uniformly over all $N\geq 1$, all $k\notin B_{\epsilon/2}(s)$ and all $\ell\in B_\epsilon(t)$. Using the concavity of $g$, we deduce that there exists $d>0$ such that the logarithm of the r.h.s.~of (\ref{Eq:IntermediateExpoDecay}) is bounded by
\begin{equation*}
-d(\ell-k) + \lambda_N s -\gamma_N k \lesssim -d \epsilon \frac{N}{2}\;,
\end{equation*}
thus concluding the proof in that case.\\
We now treat the case $4\alpha > 1$. Let $\eta > 0$. First, we assume that $s \in [0,t-\eta]$. For any $c >1/4!$, we have $g(x)\leq -x^2/2 + cx^4$ for all $x$ in a neighbourhood of the origin. Then, for $N$ large enough we bound the logarithm of the r.h.s.~of (\ref{Eq:IntermediateExpoDecay}) by
\begin{equation*}
f(s) = -\frac{1}{2}\frac{(\ell-k)^2}{2c_N(t-s)} + c \frac{(\ell-k)^4}{(2c_N(t-s))^3} + \lambda_N s - \gamma_N k\;.
\end{equation*}
A tedious but simple calculation shows the following. There exists $\delta' > 0$, only depending on $\epsilon$, such that $\sup_{s\in [0,t-\eta]} f(s) \leq -\delta' N^{2\alpha}$ for all $N$ large enough. This ensures the bound of the statement in the case where $s\in[0,t-\eta]$.\\
Using the monotonicity in $t$ of (\ref{Eq:LargeDevKernel}), we easily deduce that for all $s\in[t-\eta,t]$, we have
\begin{equation*}
\bar{p}^N_{t-s}(\ell-k) e^{\lambda_N s - \gamma_N k}\leq e^{f(t-\eta) + \lambda_N \eta}\;.
\end{equation*}
Recall that $\lambda_N$ is of order $N^{2\alpha}$. Choosing $\eta < \delta'$ small enough and applying the bound obtained above, we deduce that the statement of the lemma holds true.\\
The case where $b^N(s,k)$ is smaller than $3$ is simpler, one can adapt the above arguments to get the required bound.
\end{proof}

\begin{proof}[Proof of (\ref{Eq:BoundTailDiscreteKernel})]
The quantity $1-\sum_{k=1}^{2N-1} p^N_{t-s}(k,\ell)$ is equal to the probability that a simple random walk, sped up by $2c_N$ and started from $\ell$, has hit $0$ or $2N$ by time $t-s$. By the reflexion principle, this is smaller than twice
\begin{equation*}
\sum_{k\geq \ell} \bar{p}^N_{t-s}(k) + \sum_{k\geq 2N-\ell} \bar{p}^N_{t-s}(k)\;.
\end{equation*}
We restrict ourselves to bounding the first term, since one can proceed similarly for the second term. Using (\ref{Eq:LargeDevKernel}), we deduce that it suffices to bound $\exp(2(t-s)c_N g\big(\ell/(2(t-s)c_N)\big) +\lambda_N s)$. This is equal to the l.h.s.~of (\ref{Eq:IntermediateExpoDecay}) when $k=0$, so that the required bound follows from the arguments presented in the last proof.
\end{proof}
Finally, we rely on the following representation of $p^N$:
\begin{equation*}
p^N_t(k,\ell) = \sum_{j\in\Z} \bar{p}^N_t(k+ j4N-\ell) - \bar{p}^N_t(-k+ j4N-\ell)\;.
\end{equation*}
The next lemma shows that $p^N_t(k,\ell)$ can be replaced by $\bar{p}^N_t(\ell-k)$ up to some negligible term, whenever $\ell$ is in the $\epsilon$-bulk at time $t$. 

\begin{lemma}\label{Lemma:DecaySeriesKernel}
Fix $\epsilon > 0$ and a compact set $K\subset [0,T)$. There exists $\delta > 0$ such that uniformly over all $s\leq t \in K$, all $k\in\{1,\ldots,2N-1\}$, all $\ell\in B_\epsilon^N(t)$ and all $N\geq 1$, we have
\begin{equation*}
\big| p^N_{t-s}(k,\ell) - \bar{p}^N_{t-s}(k,\ell) \big| b^N(s,k) \lesssim e^{-\delta N^{2\alpha}}\;.
\end{equation*}
\end{lemma}
\begin{proof}
We only consider the case where $b^N(s,k) > 3$ since the other case is simpler. Observe that there exists $C>0$ such that $\log b^N(s,k) \leq C N^{2\alpha}$ for all $s\in K$ and all $k\in\{1,\ldots,2N-1\}$. Arguing differently according to the relative values of $4\alpha$ and $1$, and using the bound (\ref{Eq:LargeDevKernel}), we deduce that there exists $j_0\geq 1$ such that
\begin{equation}\label{Eq:PeriodicKernel}
\sum_{j\in\Z: |j| \geq j_0} \bar{p}^N_{t-s}(k+ j4N-\ell)b^N(s,k) + \bar{p}^N_{t-s}(-k+ j4N-\ell)b^N(s,k) \lesssim e^{-\delta N^{2\alpha}}\;,
\end{equation}
uniformly over all $s\leq t \in K$, all $k\in\{1,\ldots,2N-1\}$ and all $\ell\in B_\epsilon(t)$. On the other hand, the arguments in the proof of Lemma \ref{Lemma:ExpoDecay} yield that
\begin{align*}
\sum_{j\in\Z: |j| < j_0} \bar{p}^N_{t-s}(-k+ j4N-\ell)b^N(s,k) &\lesssim e^{-\delta N^{2\alpha}}\;,\\
\sum_{j\in\Z: 0 < |j| < j_0,} \bar{p}^N_{t-s}(k+ j4N-\ell)b^N(s,k) &\lesssim e^{-\delta N^{2\alpha}}\;,
\end{align*}
uniformly over the same set of parameters, thus concluding the proof.
\end{proof}

\begin{lemma}\label{Lemma:BoundK}
Fix $\epsilon > 0$. There exist $\beta \in (0,1)$ such that
\begin{equation*}
\int_s^t \sum_{k\in B^N_{\epsilon/2}(r)} |K_{t-r}^N(k,\ell)|(2N)^{4\alpha} dr < \beta\;,
\end{equation*}
uniformly over all $s \leq t \in K$, all $\ell\in B^N_\epsilon(t)$ and all $N$ large enough.
\end{lemma}
\begin{proof}
Lemma \ref{Lemma:DecaySeriesKernel} ensures that
\begin{align*}
\int_s^t \!\!\sum_{k\in B_{\epsilon/2}(r)}\!\! |K^N_{t-r}(k,\ell)| (2N)^{4\alpha} dr &= \int_s^t \!\!\sum_{k\in B_{\epsilon/2}(r)}\!\! |\bar{K}^N_{t-r}(k,\ell)| (2N)^{4\alpha} dr\\
&+ \cO(N^{1+4\alpha}e^{-\delta N^{2\alpha}})\;,
\end{align*}
uniformly over all $\ell \in B_\epsilon(t)$, all $t\in K$ and all $N\geq 1$. Lemma A.3 in~\cite{BG97} ensures that the first term on the r.h.s.~is smaller than some $\beta' \in (0,1)$. Since the second term vanishes as $N\rightarrow\infty$, the bound of the statement follows.
\end{proof}

\section*{Acknowledgments}
I am indebted to Reda Chhaibi for some deep discussions on this work at an early stage of the project. I would also like to thank Christophe Bahadoran for a fruitful discussion on the notion of entropy solutions for the inviscid Burgers equation, Nikos Zygouras for pointing out the article~\cite{DobrushinHryniv} and Julien Reygner for his helpful comments on a preliminary version of the paper.\\
I am grateful to an anonymous referee for his/her very careful reading of the article and for making many constructive comments that helped me to improve presentation and readability.

\bibliographystyle{abbrvnat}
\bibliography{library}

\end{document}